\documentclass[11pt, reqno]{amsart}
\usepackage[margin=3cm]{geometry}
\usepackage{graphicx} 
\usepackage{adjustbox}
\usepackage{amsmath,amssymb,amsxtra,tikz,stackengine,bm,mathrsfs,bbm,tikz-cd,comment,hyperref,combelow,amsopn,float}
\usepackage{enumitem}
\usepackage{ytableau}
\usepackage{quiver}
\usetikzlibrary{calc}%
\usetikzlibrary{shapes}%
\usetikzlibrary{patterns}%
\usetikzlibrary{positioning}%
\usetikzlibrary{arrows.meta}
\usetikzlibrary{knots}
\usetikzlibrary{decorations.markings}
\usetikzlibrary{hobby}
\allowdisplaybreaks

\numberwithin{equation}{section}

\definecolor{mediumpurple}{rgb}{0.58, 0.44, 0.86}
\definecolor{cadmiumgreen}{rgb}{0.0, 0.42, 0.24}
\definecolor{navyblue}{rgb}{0.0, 0.0, 0.5}
\definecolor{babyblueeyes}{rgb}{0.63, 0.79, 0.95}
\newtheorem{theorem}{Theorem}[subsection]
\newtheorem{proposition}[theorem]{Proposition}
\newtheorem{corollary}[theorem]{Corollary}
\newtheorem{lemma}[theorem]{Lemma}

\theoremstyle{definition}
\newtheorem{remark}[theorem]{Remark}

\newtheorem{definition}[theorem]{Definition}
\newtheorem{example}[theorem]{Example}

\newcommand{\C}{\mathbb{C}}

\newcommand{\Z}{\mathbb{Z}}

\newcommand{\skewschubert}{\lambda/\mu}

\newcommand{\skewschvar}{S_{\skewschubert}^\circ}

\def\rank{\mathrm{rank}}

\newcommand{\Fl}{\mathcal{F}\ell}
\newcommand{\Fr}{\Fl^{\mathrm{fr}}}
\newcommand{\CF}{\mathcal{F}}
\newcommand{\CFr}{\CF^{\mathrm{fr}}}

\newcommand{\CL}{\mathcal{L}}

\newcommand{\sq}{\square}

\newcommand{\pos}[1]{\Pi^{\circ}_{#1}}
\newcommand{\posone}[1]{\Pi^{\circ,1}_{#1}}
\newcommand{\cpos}[1]{\widehat{\Pi^{\circ}}_{#1}}
\newcommand{\relpos}[1]{\buildrel #1 \over \longrightarrow}
\newcommand{\wrelpos}[1]{\buildrel #1 \over \Longrightarrow}

\DeclareMathOperator{\spann}{span}

\newcommand{\std}{\mathrm{std}}
\newcommand{\ant}{\mathrm{ant}}
\def\br{\beta}
\def\BS{\mathrm{BS}}
\def\minor{\Delta}

\def\GL{\mathrm{GL}}

\def\seed{\Sigma}

\def\x2{x^{(2)}}

\def\u1{u^{(1)}}
\def\w0{\Delta}
\def\Br{\mathrm{Br}}
\def\GN{\mathsf{GN}}
\def\BA{\mathsf{BA}}
\def\forget{\pi}
\def\parabolic{\mathsf{P}}
\def\GR{\mathrm{G}R}

\newcommand{\Gr}{\mathrm{Gr}}
\newcommand{\cox}{\mathbf{c}}

\newcommand{\borel}{\mathsf{B}}
\newcommand{\unipotent}{\mathsf{U}}

\newcommand{\iso}{\iota}

\newcommand{\mubar}{\overline{\mu}}
\newcommand{\lambdabar}{\overline{\lambda}}

\newcommand{\skewright}{\lambda^{a,R}/\mu^{a,R}}
\newcommand{\skewleft}{\lambda^{a,L}/\mu^{a,L}}
\newcommand{\skewmiddle}{\lambda^{a,M}/\mu^{a,M}}

\newcommand{\lift}[1]{\beta(#1)}
\newcommand{\subs}[1]{\langle #1 \rangle}
\newcommand{\longest}{\underline{w_{0}}}

\newcommand{\Wop}{W^{\mathrm{op}}}

\newcommand{\newword}[1]{\emph{\textbf{#1}}}

\newcommand{\TS}[1]{{\color{purple}{{\bf Tonie}: #1}}}

\newcommand{\EG}[1]{{\color{red}{{\bf Eugene}: #1}}}

\newcommand{\JS}[1]{{\color{cyan}{{\bf Jos\'e}: #1}}}

\newcommand{\SK}[1]{{\color{orange}{{\bf Soyeon}: #1}}}
\definecolor{aquamarine}{rgb}{0.5, 1.0, 0.83}
\definecolor{aqua}{rgb}{0.0, 1.0, 1.0}

\title{Splicing skew shaped positroids}
\author[E. Gorsky]{Eugene Gorsky}
\address{Department of Mathematics, University of California Davis\\ One Shields Avenue, Davis CA 95616}
\email{egorskiy@ucdavis.edu}
\author[S. Kim]{Soyeon Kim}
\address{Department of Mathematics, University of California Davis\\ One Shields Avenue, Davis CA 95616}
\email{syxkim@ucdavis.edu}
\author[T. Scroggin]{Tonie Scroggin}
\address{Department of Mathematics, University of California Davis\\ One Shields Avenue, Davis CA 95616}
\email{tmscroggin@ucdavis.edu}
\author[J. Simental]{Jos\'e Simental}
\address{Instituto de Matem\'aticas, Universidad Nacional Aut\'onoma de M\'exico. Ciudad Universitaria, CDMX,  M\'exico}
\email{simental@im.unam.mx}
\date{}

\begin{document}

\begin{abstract}
Skew shaped positroids (or skew shaped positroid varieties) are certain Richardson varieties in the flag variety that admit a realization as explicit subvarieties of the Grassmannian $\Gr(k,n)$. They are parametrized by a pair of Young diagrams $\mu \subseteq \lambda$ fitting inside a $k \times (n-k)$-rectangle. For every $a = 1, \dots, n-k$, we define an explicit open set $U_a$ inside the skew shaped positroid $S^{\circ}_{\lambda/\mu}$, and show that $U_a$ is isomorphic to the product of two smaller skew shaped positroids. Moreover, $U_a$ admits a natural cluster structure and the aforementioned isomorphism is quasi-cluster in the sense of Fraser. Our methods depend on realizing the skew shaped positroid as an explicit braid variety, and generalize the work of the first and third authors for open positroid cells in the Grassmannian.
\end{abstract}
\maketitle

\section{Introduction}

In this paper, we study cutting and gluing operations for skew shaped positroid varieties, extending the constructions in \cite{GS} for the maximal positroid cell.

\subsection{Main results}
Let $\mu \subseteq \lambda$ be partitions fitting inside a $k \times (n-k)$ rectangle. The skew diagram $\skewschubert$ defines the \newword{skew shaped positroid} (or skew shaped positroid variety) $S_{\skewschubert}^{\circ}$  which can be thought of as the intersection of a Schubert cell in the complete flag variety $\Fl(n)$ labeled by $\mu$ and an opposite Schubert cell labeled by $\lambda$, see Definition \ref{def: schubert}. 
Skew shaped positroids can be thought of as special cases of open Richardson \cite{KL79} or positroid varieties \cite{KLS}. Crucially, even though the skew shaped positroid is a Richardson variety in the full flag variety, it can also be defined inside the Grassmannian $\Gr(k,n)$. Thus we can, and for the most part will, think of $S_{\skewschubert}^{\circ}$ as being a subvariety of the Grassmannian $\Gr(k,n)$. When $\mu=\varnothing$ then skew shaped positroid is an open Schubert variety, considered in \cite{SSBW19}.

\begin{remark}
    While the skew shaped positroid  $S^{\circ}_{\skewschubert}$ is defined in a similar way to the skew Schubert varieties of \cite{SSBW19}, when $\mu\neq\varnothing$ the variety $S^{\circ}_{\skewschubert}$ is generally not a skew Schubert variety. The variety $S^{\circ}_{\skewschubert}$ is closely related to Grassmannian Richardson varieties, see Remark \ref{rmk: grassmannian richardson}.
\end{remark}
Skew shaped positroids admit a cluster structure coming from source labels of plabic graphs, established in work of Galashin--Lam \cite{GL19} following earlier work of Leclerc \cite{Leclerc} and Serhiyenko--Sherman-Bennett--Williams  \cite{SSBW19}; see Section \ref{sec: skew cluster} for more details. 

Using the lattice path labeling of a skew Young diagram  described in Section \ref{subsec:lattice path labeling for skew}, we explicitly define the Grassmann necklace of $S_{\skewschubert}^{\circ}$ based on the boundary ribbon of the skew diagram $\skewschubert$ in Section \ref{sec: Grassmann necklac}.

The main construction of the paper concerns cutting the skew diagram $\skewschubert$ into two.  Choose an integer $a$ such that $1\le a\le n-k$. 
Let $\lambda^{a,R}$ be the partition obtained by considering only the first $a-1$ columns (from the right) of $\lambda$, and define $\mu^{a,R}$ similarly. Note that $\mu^{a,R} \subseteq \lambda^{a,R}$ are partitions fitting inside a $k \times (a-1)$-rectangle, so we can consider the skew shaped positroid $S_{\skewright}^{\circ} \subseteq \Gr(k,k+a-1)$. 

Now let $\lambda^{a,L}$ be the partition obtained by considering the $a, \ldots, n-k$-columns  (from the right) of $\lambda$, and similarly for $\mu^{a,L}$. Note that $\mu^{a,L} \subseteq \lambda^{a,L}$ fit inside a $k \times (n-k-a+1)$-rectangle, so we can consider the skew shaped  positroid $S_{\skewleft}^\circ \subseteq \Gr(k,n-a+1)$. We are ready to state the main result of the paper.

\begin{theorem}
\label{thm: main}
Let $n > 0$ and $0 < k < n$. Let $\mu \subseteq \lambda$ be partitions fitting inside a $k \times (n-k)$-rectangle. Choose $1 \leq a \leq n-k$, and let $U_a$ be the principal open set in $S_{\skewschubert}^\circ$ defined by the non-vanishing of the cluster variables in the $a$-th column of $\skewschubert$  (from the right), see Definition \ref{def: open subset in skew}. Then, we have an isomorphism
\[
\Phi_a: U_a \xrightarrow{\sim} S_{\skewleft}^\circ \times S_{\skewright}^\circ. 
\]
Furthermore, $\Phi_a$ is a quasi-cluster equivalence.
\end{theorem}

In the theorem above, \emph{quasi-cluster equivalence} means that the map $\Phi_a^{\ast}: \C[S^{\circ}_{\lambda^{a, L}/\mu^{a, L}} \times S^{\circ}_{\lambda^{a, R}/\mu^{a, r}}] \to \C[U_a]$ sends cluster variables to cluster variables, up to nicely behaved monomials in frozen variables, see Section \ref{subsec: cluster}. In particular, $\Phi_a^{\ast}$ preserves cluster monomials, and the map $\Phi_a$  sends cluster tori to cluster tori.

In \cite{Muller}, Muller defines a \emph{cluster chart} of a cluster variety to be the open set obtained by localizing some cluster variables in a single cluster, provided some technical conditions are satisfied, see e.g. \cite[Lemma 3.2]{Muller}. Theorem \ref{thm: main}, then, provides examples of cluster charts in skew shaped positroids that are themselves isomorphic to products of skew shaped positroids. We expect that this will allow for inductive methods to, for example, compute the cohomology \cite{Scroggin} or deep locus \cite{CGSS} of skew shaped positroids. We remark, however, that outside of some special cases, the open sets $U_{a}$ do not cover the skew shaped positroid $S^{\circ}_{\skewschubert}$.

We describe the construction of the map $\Phi_a$ and prove Theorem \ref{thm: main} in Section \ref{sec: splicing}. Instead of going into details in the introduction, we chose to illustrate the complexity of $\Phi_a$ in an example.

\begin{example}
\label{ex: intro}
 Let $n=12,\,k=5$ so that $n-k=7$. For the Young diagrams $\lambda=(7,7,5,3,1)$ and $\mu=(3,1)$ the corresponding skew diagram is shown in Figure \ref{fig: intro before}. 
\begin{figure}[ht!]
\begin{minipage}{0.5\textwidth}
\centering
\begin{tikzpicture}[scale=0.8]
\filldraw[color=lightgray] (0,-3)--(0,-5)--(3,-5)--(3,-4)--(1,-4)--(1,-3)--(0,-3);
\draw (0,0)--(0,-5)--(7,-5)--(7,-3)--(5,-3)--(5,-2)--(3,-2)--(3,-1)--(1,-1)--(1,0)--(0,0);
\draw [dotted, ultra thick, color=blue] (2,-5.5)--(2,0.5);
\draw (0,-4)--(7,-4);
\draw (0,-3)--(5,-3);
\draw (0,-2)--(3,-2);
\draw (0,-1)--(1,-1);
\draw (1,-5)--(1,-1);
\draw (2,-5)--(2,-1);
\draw (3,-5)--(3,-2);
\draw (4,-5)--(4,-2);
\draw (5,-5)--(5,-3);
\draw (6,-5)--(6,-3);
\end{tikzpicture}
\end{minipage}\hfill
   \begin{minipage}{0.5\textwidth}
   \centering
\begin{tikzpicture}[scale=0.8]
\filldraw[color=lightgray] (0,-3)--(0,-5)--(2,-5)--(2,-4)--(1,-4)--(1,-3)--(0,-3);
\draw (0,0)--(0,-5)--(2,-5)--(2,-1)--(1,-1)--(1,0)--(0,0);
\draw (0,-4)--(2,-4);
\draw (0,-3)--(2,-3);
\draw (0,-2)--(2,-2);
\draw (0,-1)--(1,-1);
\draw (1,-5)--(1,-1);

\filldraw[color=lightgray] (3,-4)--(3,-5)--(4,-5)--(4,-4)--(3,-4);
\draw (3,-5)--(3,-1)--(4,-1)--(4,-2)--(6,-2)--(6,-3)--(8,-3)--(8,-5)--(3,-5);
\draw (3,-4)--(8,-4);
\draw (3,-3)--(6,-3);
\draw (3,-2)--(4,-2);
\draw (4,-5)--(4,-2);
\draw (5,-5)--(5,-2);
\draw (6,-5)--(6,-2);
\draw (7,-5)--(7,-3);
\end{tikzpicture}
\end{minipage}
\caption{On the left is the diagram for $\skewschubert$ with corresponding matrix $(v_1,\dots,v_{12})$, where $\lambda=(7,7,5,3,1)$ and $\mu=(3,1)$. Performing the cut at $a=6$ is equivalent to cutting the diagram along the boundary of the 5-th and 6-th column from the right, as indicated by the \textcolor{blue}{blue} dotted line. The resulting skew diagrams under the map $\Phi_6$ are shown on the right, corresponding to the diagrams for $\lambda^{6,L}/\mu^{6,L}$ with matrix $(w_1,\dots,w_7)$ and $\lambda^{6,R}/\mu^{6,R}$ with matrix $(u_1,\dots,u_{10})$, respectively.}
\label{fig: intro before} 
\end{figure}
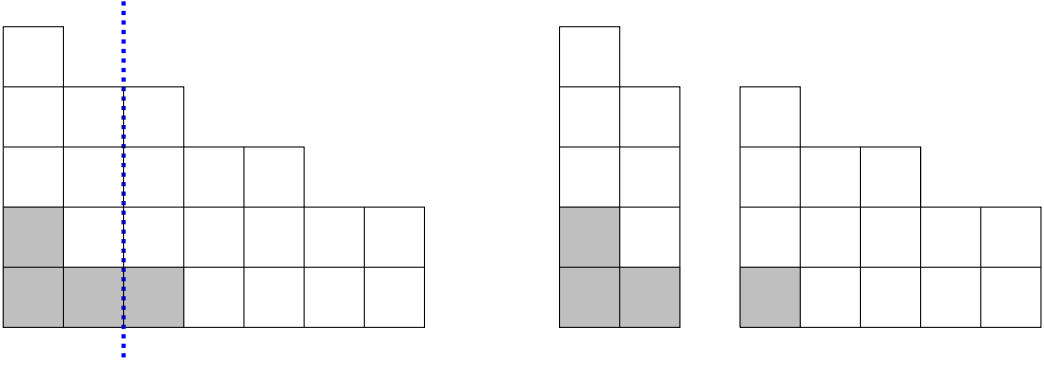
Points in the skew shaped positroid $S_{\skewschubert}^{\circ}\subset \Gr(5,12)$ can be represented by $5\times 12$ matrices with columns $(v_1,\ldots,v_{12})$.

We choose $a=6$ and cut the skew diagram as shown in Figure \ref{fig: intro before}.
We get two skew diagrams
$$
\lambda^{6,L}/\mu^{6,L}=(2,2,2,2,1)/(2,1),\ \lambda^{6,R}/\mu^{6,R}=(5,5,3,1)/(1)
$$
and the corresponding skew shaped positroids
$$
S_{\lambda^{6,L}/\mu^{6,L}}^\circ\subset \Gr(5,7),\ 
S_{\lambda^{6,R}/\mu^{6,R}}^\circ\subset \Gr(5,10).
$$

We denote the corresponding matrices by $V^L=(w_1,\ldots,w_7)$ and $V^R=(u_1,\ldots,u_{10})$.
The open subset $U_6\subset S_{\skewschubert}^{\circ}$ is defined by the inequalities 
$$
\Delta_{5,7,10,11,12}\neq 0,\  \Delta_{5,7,8,11,12}\neq 0,\ 
\Delta_{5,7,8,9,12}\neq 0.
$$
Now the map $\Phi_6: U_6\to S_{\lambda^{6,L}/\mu^{6,L}}^\circ\times S_{\lambda^{6,R}/\mu^{6,R}}^\circ$ 
is defined by the following equations:

$$
w_1=v_5, w_2=v_7, w_3=v_8, w_4=v_9, w_5=v_{10}, w_6=v_{11}, w_7=v_{12},
$$
while
$$
u_1=v_1,\ u_2=v_2,\ u_3=v_3,\ u_4=v_4,\ u_5=v_5,\ u_6=v_6,\ 
$$
$$
u_7\in \langle v_5,v_7\rangle\cap \langle v_8,v_{10},v_{11},v_{12}\rangle, u_8\in \langle v_5,v_7, v_8\rangle\cap \langle v_{10},v_{11}, v_{12}\rangle, 
$$
$$ 
u_9\in \langle v_5,v_7,v_8,v_9\rangle\cap \langle v_{11},v_{12}\rangle, u_{10}=v_{12},
$$
We prove that all intersections in question are one-dimensional, and provide explicit formulas for the vectors $u_7,u_8,u_9$, see  Section \ref{sec: intro example details} for all details.
\end{example}

In general, the construction of the map $\Phi_a$ has a similar flavor. The vectors $(w_1,\ldots,w_{n-a+1})$ corresponding to $S_{\skewleft}^{\circ}$ are chosen as a specific subset of vectors $(v_1,\ldots,v_n)$ described in Section \ref{sec: splicing left}. The vectors $(u_1,\ldots,u_{a+k-1})$ corresponding to $S_{\skewright}^{\circ}$ are instead obtained by a rather intricate  procedure that involves intersecting certain subspaces spanned by $v_i$, see  Section \ref{sec: splicing right}. 

Sometimes, the cluster variables in the $a$-th column of the skew diagram $\skewschubert$ will all be frozen, and we obtain the following result.

\begin{corollary}
\label{cor: factorization}
Assume that all boxes in the $a$-th column from the right in $\skewschubert$ belong to the boundary ribbon $R(\skewschubert)$ (see Definition \ref{def: ribbon}). Then $U_a$ coincides with the whole skew shaped positroid $S_{\skewschubert}^{\circ}$, and we have the isomorphism
$$
S_{\skewschubert}^{\circ}\simeq S_{\skewright}^\circ \times S_{\skewleft}^\circ.
$$
\end{corollary}

\begin{remark}
One can define and study the compositions of the splicing maps $\Phi_a$ as follows. Suppose that $1\le a<b\le n-k$ and $\skewschubert$, then we can draw two  vertical cuts, one between $a$-th and $(a-1)$-th columns, and another between $b$-th and $(b-1)$-th columns. After performing these cuts we obtain three pieces: left $\skewleft$, middle $\skewmiddle$, and right $\skewright$.

Now, we can make an initial cut at $a$, then follow with a cut at $b$ to get the following maps:
$$
S_{\skewschubert}^{\circ}\xleftarrow{\Phi_a^{-1}}
S_{\skewleft\cup \skewmiddle}^{\circ}\times S_{\skewright}^{\circ}\xleftarrow{\Phi_{b-a+1}^{-1}\times \mathrm{Id}} S_{\skewleft}^{\circ}\times S_{ \skewmiddle}^{\circ}\times S_{\skewright}^{\circ}.
$$
On the other hand, if we make the initial cut at $b$, then follow with a cut at $a$ we then get the maps:
$$
S_{\skewschubert}^{\circ}\xleftarrow{\Phi_b^{-1}}
S_{\skewleft}^{\circ}\times S_{\skewmiddle\cup \skewright}^{\circ} \xleftarrow{\mathrm{Id}\times \Phi_{a}^{-1}} S_{\skewleft}^{\circ}\times S_{ \skewmiddle}^{\circ}\times S_{\skewright}^{\circ}.
$$
The image for both maps is the the principal open subset $U_{a,b}\subset S_{\skewright}^\circ$  defined by non-vanishing of all cluster variables in $a$-th and $b$-th columns of $\skewschubert$. However, we expect that the maps do not coincide but differ by some explicit monomials in frozen cluster variables on $U_{a,b}$. 

See \cite[Section 4]{Scroggin} for detailed computations for $k=2$. We plan to study the general case in future work.\footnote{There is slight discrepancy for the notation of the splicing map between \cite{Scroggin} and our paper. In \cite{Scroggin}, the splicing (or cutting) map is denoted $\Phi^{-1
}_{ij}$, whereas in our paper it is denoted $\Phi_a$.}
\end{remark}

\subsection{Braid varieties} 
The key insight for the construction of the map $\Phi_a$ comes from the construction of \newword{braid varieties} \cite{CGGS2,CGGS}. A braid variety $X(\beta)$ is an affine algebraic variety which can be associated to a positive braid $\beta$ on $k$ strands. It parametrizes sequences of complete flags in $\C^k$ with relative positions governed by the crossings in $\beta$, see Definition \ref{def: braid variety} for more details. Braid varieties have a cluster structure constructed in \cite{CGGLSS22,GLSBS22,GLSB}.

Given a skew diagram $\skewschubert$, we define a $k$-strand braid $\beta_{\skewschubert}$, see Section \ref{sec:braid diagram configuration for skew}. The crossings in $\beta_{\skewschubert}$ are in bijection with the boxes in $\skewschubert$, except for the top boxes in each column.

\begin{example}\label{ex:running k-strand}
    For $\lambda=(7,7,5,3,1), \mu=(3,1)$ we get the braid $$\br_{\skewschubert} = (\sigma_4\sigma_3)(\sigma_3\sigma_2)(\sigma_3\sigma_2)(\sigma_2\sigma_1)(\sigma_2\sigma_1)(\sigma_1)(\sigma_1).$$  
    The braid is shown in \textcolor{blue}{blue} and the crossed out boxes are not part of it.
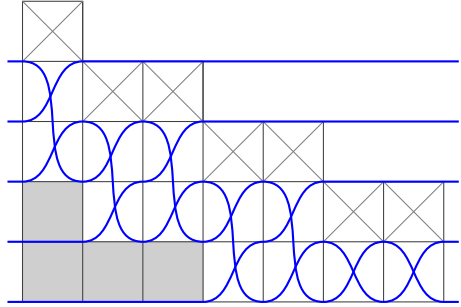
\begin{center}
\begin{tikzpicture}[scale=0.8]
\filldraw[color=lightgray!75] (0,0) -- (3,0) -- (3,1) -- (1,1) -- (1,2) -- (0,2) -- cycle;
\draw[color=gray!150] (0,0) -- (0,5) -- (1,5) -- (1,4) -- (3,4) -- (3,3) -- (5,3) -- (5,2) -- (7,2) -- (7,0)-- cycle;
\draw[color=gray!150] (0,1) to (7,1);
\draw[color=gray!150] (0,2) to (7,2);
\draw[color=gray!150] (0,3) to (5,3);
\draw[color=gray!150] (0,4) to (3,4);
\draw[color=gray!150] (1,0) to (1,5);
\draw[color=gray!150] (2,0) to (2,4);
\draw[color=gray!150] (3,0) to (3,4);
\draw[color=gray!150] (4,0) to (4,3);
\draw[color=gray!150] (5,0) to (5,2);
\draw[color=gray!150] (6,0) to (6,2);

\draw[color=gray] (0,5) to (1,4); \draw[color=gray] (0,4) to (1,5);

\draw[color=gray] (1,3) to (2,4); \draw[color=gray] (1,4) to (2,3);

\draw[color=gray] (2,3) to (3,4); \draw[color=gray] (2,4) to (3,3);

\draw[color=gray] (3,2) to (4,3); \draw[color=gray] (3,3) to (4,2); \draw[color=gray](4,2) to (5,3); \draw[color=gray] (4,3) to (5,2);

\draw[color=gray](5,1) to (6,2); \draw[color=gray] (5,2) to (6,1); \draw[color=gray] (6,1) to (7,2); \draw[color=gray] (6,2) to (7,1);
\draw[color=blue, thick] (-0.25,4) to (0,4) to[out=0, in=180] (1,2) to[out=0, in=180] (2,3) to[out=0, in=180] (3,1)  to[out=0, in=180] (4,2) to[out=0, in=180] (5,0) to[out=0, in=180] (6,1) to[out=0, in=180] (7,0) to (7.25,0);
\draw[color=blue, thick] (-0.25,3) to (0,3) to[out=0, in=180] (1,4) to[out=0, in=180] (7.25,4) ;
\draw[color=blue, thick] (-0.25,2) to (0,2) to[out=0, in=180] (1,3) to[out=0, in=180] (2,1) to[out=0, in=180] (3,2) to[out=0, in=180] (4,0) to[out=0, in=180] (5,1) to[out=0, in=180] (6,0) to[out=0, in=180] (7,1) to (7.25,1);
\draw[color=blue, thick] (-0.25,1) to (0,1) to[out=0, in=180] (1,1) to[out=0, in=180] (2,2) to[out=0, in=180] (3,3) to[out=0, in=180] (7.25,3);
\draw[color=blue, thick] (-0.25,0) to[out=0, in=180] (3,0) to[out=0, in=180] (4,1) to[out=0, in=180] (5,2) to[out=0, in=180] (7.25,2);
\end{tikzpicture} 
\end{center}
 
\end{example}

\begin{remark}
Note that all boxes in the $k$-th (topmost) row of $\skewschubert$ are crossed out and not used in the braid $\beta_{\skewschubert}$. In fact, the topmost strand of  $\beta_{\skewschubert}$ goes right below the topmost row of $\skewschubert$.
\end{remark}

We construct an explicit isomorphism between the skew shaped positroid and a certain braid variety. This is inspired by the constructions of \cite{CGGS2,CLSBW23,STWZ} but we do not check if our isomorphism coincides with either of these.

\begin{theorem}
\label{thm: schubert to braid intro}
We have an isomorphism of algebraic varieties
\begin{equation}
\label{eq: intro skew Schubert to braid}
\skewschvar \cong X\left(\longest\br_{\skewschubert}\right) \times (\C^{\times})^{s}.
\end{equation}
Here $\longest$ is the half-twist braid on $k$ strands and $s$ is the number of crossed out boxes. Moreover, \eqref{eq: intro skew Schubert to braid} is a quasi-cluster equivalence.
\end{theorem}
We can visualize Theorem \ref{thm: schubert to braid intro} by the diagram 
\begin{equation}
\label{eq: omega intro}
\Omega(V)=\left[\CF^{W}\stackrel{\longest}{\dashrightarrow}\CF^{\Wop}\stackrel{\beta_{\skewschubert}}{\dashrightarrow}\CF^0\right],
\end{equation}
where $\CF^W$ and $\CF^{\Wop}$ are a specific pair of opposite flags, and dotted arrows encode sequences of flags in $\C^k$. All the subspaces in all these flags can be defined as spans of some explicit subsets of columns of $V$, see Figure \ref{fig: skew braid diagram} for an example. Our construction of both the isomorphism $\Omega$ and its inverse uses only elementary linear algebra, avoiding the use of microlocal sheaves \cite{STWZ} or the Chekanov--Eliashberg DGA \cite{CGGS2}, or marking points in the braid closure of $\longest\br_{\skewschubert}$ \cite{CLSBW23}. 


Next, we can describe the analogue of the map $\Phi_a$ for braid varieties. We factor the braid $\beta_{\skewschubert}$ as $\beta_{\skewschubert}=\beta_L\beta_R$ where $\beta_L$ and $\beta_R$ are two braids corresponding to $\skewleft$ and $\skewright$, respectively. This allows us to expand \eqref{eq: omega intro} to the diagram
\begin{equation}
\label{eq: intro omega long}
\Omega(V)=\left[\CF^{W}\stackrel{\longest}{\dashrightarrow}\CF^{\Wop}\stackrel{\beta_{L}}{\dashrightarrow}\CF^a\stackrel{\beta_{R}}{\dashrightarrow}\CF^0\right],
\end{equation}
where $\CF^a$ is a specific flag on the boundary between $\beta_L$ and $\beta_R$, refer to (\ref{eq: flag cut}) for further details.

 The key observation (Lemma \ref{lem: freeze column}) is that the open subset $U_a\subset \skewschvar$ corresponds under the isomorphism \eqref{eq: intro skew Schubert to braid} to the set of sequences of complete flags \eqref{eq: intro omega long} such that the flag $\CF^a$ is transversal to $\CF^W$. 
This allows us to define a map 
$$
(\Omega^L,\Omega^R): U_a\rightarrow X\left(\longest\beta_R\right)\times X\left(\longest\beta_L\right)
$$
given by
\begin{equation}
\label{eq: Phi flags intro}
\Omega^L(V)=\left[\CF^{W}\stackrel{\longest}{\dashrightarrow}\CF^{\Wop}\stackrel{\beta_{L}}{\dashrightarrow}\CF^a\right],\quad  \Omega^R(V)=\left[\CF^{W}\stackrel{\longest}{\dashrightarrow}\CF^{a}\stackrel{\beta_{R}}{\dashrightarrow}\CF^0\right].
\end{equation}

By using \eqref{eq: intro skew Schubert to braid} again and taking into account the $\C^{\times}$ factors, we arrive at our construction of the map $\Phi_a$. This argument shows that $\Phi_a$ is indeed well defined, and is an isomorphism of algebraic varieties. In Section \ref{sec: splicing}, we prove that this construction of $\Phi_a$ is consistent with the linear algebra construction of $V^L=(w_1,\ldots,w_{n-a+1})$ and $V^R=(u_1,\ldots,u_{a+k-1})$, alluded to in Example \ref{ex: intro}, and provide a precise definition of the latter.

To prove that $\Phi_a$ is a quasi-isomorphism of cluster structures, we use the triangularity of $u_i$ with respect to $v_i$ and compare the corresponding minors in Lemma \ref{lem: minors triang}. The computation of exchange ratios is then fairly straightforward (yet tedious), and is done in Lemmas \ref{lem: exchange ratios not bdry} and \ref{lem: exchange ratios bdry}. This step does not require braid varieties.

To summarize, the most technically challenging part of the paper is to construct the element $V^R = (u_1, \dots, u_{k+a-1})$ and verify that it indeed represents an element in the skew shaped positroid $S^{\circ}_{\lambda^{a, R}/\mu^{a, R}}$.


In the forthcoming paper \cite{forthcoming}, we use the cluster structure on braid varieties \cite{CGGLSS22,GLSBS22,SW} to study more general splicing maps for braid varieties and to prove that these are cluster quasi-equivalences. This would give an alternative proof that $\Phi_a$ is a quasi-equivalence.

\section*{Acknowledgments}

The authors would like to thank Jo\~ao Pedro Carvalho, Roger Casals, Mikhail Gorsky, Thomas Lam, Seung Jin Lee, Eunjeong Lee, Matteo Parisi, Melissa Sherman-Bennett,  David Speyer, Daping Weng, and Lauren Williams   for the useful discussions. 
Eugene Gorsky, Soyeon Kim and Tonie Scroggin were partially supported by the NSF grant DMS-2302305. Jos\'e Simental was partially supported by CONAHCyT project CF-2023-G-106 and UNAM’s PAPIIT Grant IA102124.

\section{Background}

\subsection{Flags and framed flags} We will work with the variety $\Fl(k)$ of complete flags   in the $k$-dimensional space $\C^k$. Given a matrix $M = (m_1,m_2,\cdots, m_k) \in \GL(k)$ with columns $m_1, \dots, m_k$, its associated flag is
\[
\CF(M) := \left(\{0\} \subseteq \subs{m_1} \subseteq \subs{m_1, m_2} \subseteq \cdots \subseteq \subs{m_1, \dots, m_k} = \C^k\right).
\]
Note that $\CF(M) = \CF(MU)$ for any upper triangular matrix $U$, so that we have an identification 
\[
\Fl(k) = \GL(k)/\borel(k),
\]
where $\borel(k) \subseteq \GL(k)$ is the Borel subgroup of upper triangular matrices. 

Let $\unipotent(k) \subseteq \borel(k)$ be the subgroup of upper triangular matrices with $1$'s on the diagonal. The variety of \emph{framed flags} is the space
\[
\Fr(k) := \GL(k)/\unipotent(k).
\]
An element of $\Fr(k)$ is known as a framed flag -- it is the data of a flag $\CF \in \Fl(k)$, together with a choice of a nonzero element in $\CF_i/\CF_{i-1}$ for all $i = 1, \dots, k$ (referred to as a framing). Forgetting the framing gives rise to a natural map $\forget: \Fr(k) \to \Fl(k)$. As above, every element $M \in \GL(k)$ gives rise to a framed flag, that we will denote by $\CFr(M)$.

\subsection{Relative position} Let $w \in S_k$. Abusing the notation, we denote by $w \in \GL(k)$ the corresponding permutation matrix. In particular, we have the flag (resp. framed flag) $\CF(w)$ (resp. $\CFr(w)$). The \emph{standard flag} $\CF^{\std}$ is the flag $\CF(e)$, where $e \in S_k$ is the identity, and the \emph{antistandard flag} $\CF^{\ant}$ is the flag $\CF(w_0)$, where $w_0 \in S_k$ is the longest element. The standard and antistandard framed flags are defined analogously.

Note that we have an action of the group $\GL(k)$ both on $\Fl(k)$ and $\Fr(k)$. We say that two flags $\CF^1, \CF^2 \in \Fl(k)$ are in relative position $w \in S_k$ if there exists an element $g \in \GL(k)$ such that
\[
g\CF^1 = \CF^{\std}, \qquad g\CF^2 = \CF(w). 
\]
We will write $\CF^1 \relpos{w} \CF^2$ to express that $\CF^1$ and $\CF^2$ are in relative position $w$. 
It is easy to verify that $\CF^1\relpos{w}\CF^2$ if and only if for every $i, j = 1, \dots, k$ we have
\[
\dim(\CF^1_i \cap \CF^2_j) = \big|\{1, \dots, i\}\cap\{w(1), \dots, w(j)\}\big|. 
\]
In particular, for a simple reflection $s_j \in S_k$, we have that $\CF^1 \relpos{s_j} \CF^2$ if and only if $\CF^1_j \neq \CF^2_j$, but $\CF^1_i = \CF^2_i$ for every $i \neq j$. 

The notion of relative position for framed flags is slightly more complicated. We say that two framed flags $(\CFr)^1$ and $(\CFr)^2$ are in \emph{strong relative position} $w$ if there exists $g \in \GL(k)$ such that
\[
g(\CFr)^1 = (\CFr)^{\std}, \qquad g(\CFr)^2 = \CFr(w).
\]
As before, we write $(\CFr)^1 \relpos{w} (\CFr)^2$ if the framed flags $(\CFr)^1$ and $(\CFr)^2$ are in strong relative position $w$. 

In particular, consider two framed flags $\CFr,(\CFr)'$ with underlying flags $\CF=\pi(\CFr),\CF'=\pi((\CFr)')$. Then  $\CFr\relpos{s_i}(\CFr)'$  if and only if   $\CF \relpos{s_i} \CF'$ and the framing on $(\CFr)'$ is determined by the framing on $\CFr$ as follows. 
Let
$\CL_j=\CF_j/\CF_{j-1}$ and $\CL'_j=\CF'_j/\CF'_{j-1}$ for $j=1,\ldots,n$. Then we have canonical isomorphisms
\begin{equation}
\label{eq: propagate framing}
\CL_i\simeq \CL'_{i+1},\ \CL_{i+1}\simeq \CL'_i,\ \CL_j\simeq \CL'_j\ (j\neq i,i+1).
\end{equation}
Indeed, the following composition is an isomorphism whenever $\CF \relpos{s_i} \CF'$:
$$
\CL_i=\CF_i/\CF_{i-1}\hookrightarrow \CF_{i+1}/\CF_{i-1}=\CF'_{i+1}/\CF'_{i-1}\rightarrow \CF'_{i+1}/\CF'_i=\CL'_{i+1}.
$$
The other isomorphisms are constructed similarly. The framing on $(\CFr)'$ is then determined by the framing on $\CFr$ using these isomorphisms.

We also say that the framed flags $(\CF^r)^1$ and $(\CF^r)^2$ are in \emph{weak relative position $w$} if there exists $g \in \GL(k)$ and a diagonal matrix $t$ such that
\[
g(\CFr)^1 = (\CFr)^{\std}, \qquad g(\CFr)^2 = t\CFr(w).
\]
We write $(\CFr)^1 \wrelpos{w} (\CFr)^2$ to express this relationship. Note that $(\CFr)^1 \relpos{w} (\CFr)^2$ implies $(\CFr)^1 \wrelpos{w} (\CFr)^2$ but the converse is not true. Also, note that $(\CFr)^1 \wrelpos{w} (\CFr)^2$ if and only if the corresponding flags $\forget((\CFr)^1)$ and $\pi((\CFr)^2)$ are in relative position $w$. 

\subsection{Transversality}   We say that two flags $\CF^1$ and $\CF^2$ are transversal, and write $\CF^1 \pitchfork \CF^2$, if $\CF^1 \relpos{w_0} \CF^2$. Note that this is the case if and only if
\[
\CF^1_i \cap \CF^2_{k-i} = \{0\}, \qquad \text{equivalently} \qquad \CF^1_i + \CF^2_{k-i}= \C^k
\]
for every $i = 1, \dots, k$. By definition, if $\CF^1 \pitchfork \CF^2$ then there exists a matrix $g \in \GL(k)$ such that $g\CF^1 = \CF^{\std}$ and $g\CF^2 = \CF^{\ant}$, and this matrix is unique up to left multiplication by an element of $\borel (k)\cap w_0\borel(k) w_0$, i.e., a diagonal matrix. Equivalently, one can say that there exists a basis $u_1, \dots, u_k$ of $\C^k$ such that
\[
\CF^1 = \left(\subs{u_1} \subseteq \subs{u_1, u_2} \subseteq \cdots \subseteq \C^k\right) \qquad \text{and} \qquad \CF^2 = \left(\subs{u_k} \subseteq \subs{u_k, u_{k-1}} \subseteq \cdots \subseteq \C^k\right),
\]
such a basis is unique up to scalar multiplication. 

For framed flags, we say that $(\CFr)^1$ and $(\CFr)^2$ are transversal, and again write $(\CFr)^1 \pitchfork (\CFr)^2$, if $(\CFr)^1 \wrelpos{w_0} (\CFr)^2$ (note that we are using \emph{weak relative position} here). In this case, there exists a matrix $g \in \GL(k)$ and a diagonal matrix $t$ such that $g(\CFr)^1 = (\CFr)^{\std}$ and $g(\CFr)^2 = t(\CFr)^{\ant}$. Since $\unipotent(k) \cap w_0\borel(k) w_0 = \{1\}$, such a matrix $g$ is unique. If we have framed flags $(\CFr)^1$ and $(\CFr)^2$ that are in strong relative position $w_0$, we say that $(\CFr)^1$ and $(\CFr)^2$ are strongly transversal. 


\subsection{Braid varieties}

Now we are ready to define braid varieties.

\begin{definition}
\label{def: braid variety}
Let $\beta=\sigma_{i_1}\cdots\sigma_{i_r}\in \Br_k^+$ be a positive braid word.  
The braid variety $X(\beta)$ is defined as the space of configurations of flags 
\begin{equation}
\label{eq: rel position beta}
\CF^0\xrightarrow{s_{i_1}}\CF^1\xrightarrow{s_{i_2}}\ldots \xrightarrow{s_{i_r}}\CF^r.
\end{equation}
such that $\CF^0$ is standard and $\CF^r$ is antistandard.
We will often abbreviate \eqref{eq: rel position beta} as $\CF^0\stackrel{\beta}{\dashrightarrow}\CF^r$.
\end{definition}


Following \cite{CGGS,CZ} we will often use the following equivalent version of Definition \ref{def: braid variety}. Given the braid diagram of $\beta$ drawn horizontally, we encode a point in the braid variety $X(\beta)$ by labeling all regions by subspaces of $\C^k$ as follows:
\begin{itemize}
\item[(a)] The unbounded region below $\beta$ is labeled by $0$, and the unbounded region above $\beta$ is labeled by $\C^k$.
\item[(b)] If a subspace $W$ labels a region directly above one labeled by a subspace $V$, then $V\subset W$ and $\dim W=\dim V+1$.
\item[(c)] At a crossing in $\beta$, we have a configuration of subspaces
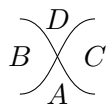
\begin{center}
\begin{tikzpicture}
\draw (0,1)..controls (0.5,1) and (0.5,0)..(1,0);
\draw (0,0)..controls (0.5,0) and (0.5,1)..(1,1);
\draw (0,0.5) node {$B$};
\draw (0.5,0) node {$A$};
\draw (0.5,1) node {$D$};
\draw (1,0.5) node {$C$};
\end{tikzpicture}
\end{center}
where $A\subset B\subset D$, $A\subset C\subset D$ and $B\neq C$. 
\item[(d)] On the left boundary of $\beta$ we get the flag $\CF^{0}=\CF^{\std}$ and on the right boundary of $\beta$ we get the flag $\CF^{r}=\CF^{\ant}$.
\end{itemize}
See Figure \ref{fig: skew braid diagram} for an example of such labeling. The conditions (a) and (b) imply that every vertical cross-section of the diagram corresponds to a complete flag in $\C^k$, and the condition (c) ensures that the neighboring flags are in a correct relative position.

Furthermore, we can think of elements of $X(\beta)$ as configurations of framed flags \begin{equation}
\label{eq: rel position beta framed}
(\CFr)^0\xrightarrow{s_{i_1}}(\CFr)^1\xrightarrow{s_{i_2}}\ldots \xrightarrow{s_{i_r}}(\CFr)^r.
\end{equation}
such that $(\CFr)^0$ is transversal to $(\CFr)^r$, modulo simultaneous left action of $\GL(k)$. 
 Indeed, if  $(\CFr)^0$ is transversal to $(\CFr)^r$ then there exists a unique matrix $g\in \GL(k)$ such that $g(\CFr)^0$ is the standard framed flag and $g(\CFr)^r$ projects to the antistandard flag. See \cite{comparisonpaper,GLSB,GLSBS22}. 


\begin{lemma}
\label{lem: flags}
Let $\longest$ be the positive braid lift of $w_0$.
The braid variety $X(\longest\beta)$ can be equivalently described as the space of following data, up to simultaneous left action of $\GL(k)$: 

\noindent (a) Configuration of framed flags $(\CFr)^0\stackrel{\longest}{\dashrightarrow}(\CFr)^1\stackrel{\beta}{\dashrightarrow}(\CFr)^2$
such that $(\CFr)^0$ is transversal to $(\CFr)^2$. 

\noindent (b) Configuration of framed flags $(\CFr)^1\stackrel{\beta}{\dashrightarrow}(\CFr)^2$ and a framed flag $(\CFr)^0$ which is strongly transversal to $(\CFr)^1$ and weakly transversal to  $(\CFr)^2$.

\noindent (c) Configuration of flags $\CF^1\stackrel{\beta}{\dashrightarrow}\CF^2$ and a choice of basis $u_1,\ldots,u_k$ such that 
$$
\CF^1=\left\{\langle u_1\rangle, \langle u_1,u_2\rangle,\ldots \langle u_1,u_2,\ldots,u_k\rangle \right\}
$$
and 
$\CF^2$ is transversal to the framed flag
$$
\{\subs{u_k}, \subs{u_k, u_{k-1}},  \cdots \subseteq \subs{u_k, \dots, u_1}\}.
$$
\end{lemma}

\begin{proof}
Part (a) follows from the definition. Part (b) follows from (a) since    $(\CFr)^0\stackrel{\longest}{\dashrightarrow} (\CFr)^1$ is equivalent to strong transversality of $(\CFr)^0$ and $(\CFr)^1$. Part (c) follows from (b) since for strongly transversal $(\CFr)^0$ and $(\CFr)^1$  there is a unique basis $u_1,\ldots,u_k$ such that 
$$
\CF^1=\left\{\langle u_1\rangle,\langle u_1,u_2\rangle,\ldots \langle u_1,u_2,\ldots,u_k\rangle \right\}
$$
and 
$$
\CF^0=\left\{\langle u_k\rangle,\langle u_k,u_{k-1}\rangle,\ldots \langle u_k,u_{k-1},\ldots,u_1\rangle \right\},
$$
where we denote by $\CF^0$ and $\CF^1$ the standard projections of the framed flags $(\CFr)^0, (\CFr)^1$, respectively. Finally, if we are given the data as in (c) then the basis $u_1, \dots, u_k$ naturally endows $\CF^1$ with a framing, that we extend through the configuration $\CF^1 \stackrel{\br}{\dashrightarrow}\CF^2$ to obtain a framing on $\CF^2$, and we recover the data of (b). 
\end{proof}

\begin{remark}
\label{rem: propagate framing}
In what follows, we only record the framing on $(\CFr)^0$ and treat all other flags as unframed. As above, the strong relative position requirement induces a framing on all other flags using \eqref{eq: propagate framing}.
\end{remark}

\section{Skew shaped positroid varieties}

\subsection{Positroid varieties}

First, we recall equivalent ways of describing positroid varieties in the Grassmannian: Grassmann necklaces and bounded affine permutations.

\begin{definition}\label{def: abstract defn of Grassmann necklace}
A \newword{$(k,n)$-source Grassmann necklace} $\mathcal{I} = (I_1, I_2, \dots, I_n)$ is an $n$-tuple of $k$-element subsets $I_{i} \subseteq [n]$ where 
\begin{itemize}
    \item if $i\in I_{i}$ then there exists $j \in [n]$  such that $I_{i-1} = (I_{i} \setminus \{i\})\cup\{j\}$, and 
    \item if $i\not\in I_{i}$ then $I_i=I_{i-1}$.
\end{itemize}
We denote by $\GN(k,n)$ the set of $(k,n)$-Grassmann necklaces. 

\end{definition}

\begin{remark}
    Note that if $i \in I_{i}$ then it may be that $I_{i-1} = I_{i} = (I_{i}\setminus\{i\})\cup\{i\}$. Finally, note that $I_{i-1}\setminus I_{i}$ is either empty or a singleton. 
\end{remark}

\begin{remark}
Grassmann necklaces were first defined in \cite{postnikov}. Our definition is slightly different, as we use $I_{i-1}$ instead of $I_{i+1}$, this difference is indicated by the word {\bf source} in the name.
\end{remark}

\begin{definition}[\cite{KLS}]
A bijection $f: \mathbb{Z}\to \mathbb{Z}$ is called a \newword{$(k,n)$-bounded affine permutation} if it satisfies the following conditions:
\begin{enumerate}
\item $f(i+n) = f(i) + n$ for every $i \in \Z$.
\item For every $i \in \Z$, $i \leq f(i) \leq f(i)+n$.
\item We have:
\[
\sum_{i = 1}^{n}(f(i) - i) = nk.
\]
\end{enumerate}
Note that, by (1), $f$ is uniquely determined by the values $f(1), \dots, f(n)$, which are pairwise distinct modulo $n$. We often simply denote $f = [f(1), \dots, f(n)]$. Note also that, upon the presence of (1) and (2), (3) is equivalent to $|\{i \in [n] : f(i) > n\}| = k$. We denote the set of $(k,n)$-bounded affine permutations by $\BA(k,n)$. We say that $f$ is a \newword{bounded $n$-affine permutation} if there exists $k$ such that $f$ is a $(k,n)$-bounded affine permutation.
\end{definition}

\begin{lemma}
\label{lem: GN to BA bijection}
    The sets $\GN(k,n)$ and $\BA(k,n)$ are in natural bijection.
\end{lemma}

In the proof, we will need the following definition.

\begin{definition}
We define the cyclic order $\leq_i$ on the set $[n]$: $i <_i i+1 <_i i+2 <_i \cdots <_i n <_i 1 <_i \cdots <_i i-1$.

\end{definition}

\begin{proof}[Proof of Lemma \ref{lem: GN to BA bijection}]
    This is essentially a combination of \cite[Corollary 3.13]{KLS} and \cite[Remark 2.4]{MS17}. For the reader's convenience and to fix notation, we provide the bijection $\varphi: \GN(k,n) \to \BA(k,n)$. Let us first describe $\bar{f}_{\mathcal{I}} := \varphi(\mathcal{I}) \pmod n$ through its inverse $\bar{f}_{\mathcal{I}}^{-1}$:
    we have  $\bar{f}^{-1}_{\mathcal{I}}(i) = i$ if $I_{i-1} = I_{i}$, and $\bar{f}^{-1}_{\mathcal{I}}(i) = j$ if $I_{i-1}\setminus I_{i} = \{j\}$. To lift this to a bounded affine permutation, we set $f_{\mathcal{I}}(i) = i+n$ if $I_{i-1} = I_{i}$ and $i \in I_{i}$, and $f_{\mathcal{I}}(i) = i$ if $i \not\in I_{i}$. 
   On the other hand, if $f: \Z \to \Z$ is a $k$-bounded affine permutation, let $I = \{i \in [n] \mid f(i) = i+n\}$, and $\bar{f} = f \pmod n$. Then $I_{i}$ consists of $I$ together with the elements $\{a \in [n] \mid a <_{i} \bar{f}(a)\}$. The collection $\mathcal{I}_f = (I_1, \dots, I_n)$ is a Grassmann necklace, and $f \mapsto \mathcal{I}_{f}$, $\mathcal{I} \mapsto f_{\mathcal{I}}$ are inverse bijections. 
\end{proof}
Now, we associate a Grassmann necklace and a bounded affine permutation to an element $V$ in the Grassmannian $\Gr(k,n)$ following \cite{KLS,postnikov}. We represent $V$ by a $k\times n$ matrix of rank $k$, up to row operations, and denote by $v_1,\ldots,v_n\in \C^k$ the columns of $V$. Furthermore, we define $v_{i}$ for all $i\in \mathbb{Z}$ by 
setting $v_{i+n}:=(-1)^{k-1}v_{i}$.
    Given an (ordered) $k$-element subset $J=\{j_1,\ldots,j_k\}\subset \{1,\ldots,n\}$, we denote by $\minor_J=\minor_J(V)=v_{j_1}\wedge\cdots \wedge v_{j_k}$ the maximal minor of $V$ in the columns of $J.$ 
\begin{definition}(\cite{MS17})
\label{def: Grassmann necklace for V}
Note that the order $\leq_i$ naturally extends to a partial order on the set of $k$-element subsets of $[n]$:
\[
\{j_1<_ij_2<_i\cdots <_i j_k\} \leq_i \{\ell_1<_i \ell_2<_i \cdots <_i \ell_k\} \qquad \text{if} \qquad j_s \leq_i \ell_s \; \text{for every} \; s = 1, \dots, k.
\]
Given this ordering, the \newword{source Grassmann necklace} associated to $V$ is a collection of $k$-element sets $\mathcal{I}_{V}=(I_{1},I_{2}\ldots,I_{n})$ where $I_{i}$ is defined as  \[
I_{i}:=\text{max}_{\leq_{i+1}}\left\{J \in \binom{[n]}{k} \mid \minor_{J}(V) \neq 0\right\}.
\]
In other words, $I_{i}$ is the maximal element, with respect to the $\leq_{i+1}$ ordering, among all $k$-element subsets $J$ such that $\Delta_J(V)\neq 0$. 
\end{definition}

\begin{remark}
    Note that the order $\leq_{i+1}$ is the cyclic ordering where $i$ is the \emph{largest} element and $i+1$ is the smallest.
\end{remark}

\begin{remark}
Here we have used the source Grassmann necklace. Other works in the literature use \emph{target} Grassmann necklaces, see e.g. \cite{MS17, FSB22}. These are related to source and target labelings for plabic graphs, see Section \ref{sec: skew cluster}.

\end{remark}
\begin{definition}(\cite{KLS})
Let $V = (v_1|\cdots|v_n) \in \Gr(k,n).$ The bounded affine permutation $f_{V}$ is defined as \[
f_{V}(i):=\text{min}\{j\geq i\ \vline\  v_{i}\in \text{span}(v_{i+1},v_{i+2},\ldots,v_{j})\}.
\]
\end{definition}

\begin{definition}(\cite[Theorem 5.1]{KLS})
\label{def: positroid}
Given a bounded affine permutation $f$ (or, equivalently, the corresponding Grassmann necklace $\mathcal{I} = \varphi^{-1}(f)$) we define the \newword{open positroid variety} by
$$
\Pi_f^{\circ}=\Pi_{\mathcal{I}}^{\circ}:=\{V\in \Gr(k,n)\ :\ f_V=f\} = \{V \in \Gr(k,n)\ :\ \mathcal{I}_V = \mathcal{I}\}.
$$
\end{definition}

\subsection{Skew shaped positroids}

We fix $0 < k < n$ and consider a partition $\lambda$ whose Young diagram fits inside the $(n-k)\times k$-rectangle, that is, $\lambda = (\lambda_1, \dots, \lambda_k)$ with $\lambda_1 \leq n-k$. We draw partitions using the French convention, i.e., the boxes are lower left justified. Given partitions $\lambda$ and $\mu$ with $\mu_i \leq \lambda_i$ for all $1 \leq i \leq k$, the \newword{skew diagram} $\skewschubert$ denotes the set-theoretic difference of the Young diagrams of $\lambda$ and $\mu$. We also define 
$$
\mubar_a:=\mu^t_{n-k+1-a},\quad \lambdabar_a:=\lambda^t_{n-k+1-a}.
$$
In other words, $\lambdabar_a$ is the height of the $a$-th column of $\lambda$, reading from right-to-left. 

Recall that $w\in S_n$ is a $k$-Grassmannian permutation if $w(1) < \cdots < w(k)$ and $w(k+1) < \cdots < w(n)$.

\begin{definition}
Given a partition $\lambda\subseteq (n-k)^k$ we define the $k$-Grassmannian permutation 
\begin{equation}\label{eq: w lambda}
    w_{\lambda}=[n-k+1-\lambda_1,\ldots,n-\lambda_k,1+\lambda_{n-k}^{t},\ldots,k+\lambda_{1}^{t}].
\end{equation}

\end{definition}

\begin{remark}\label{rmk:comparison-CGGS2}
    Note that the permutation $w_\lambda$ as defined by \eqref{eq: w lambda} differs from its namesake element in \cite[(2.1)]{CGGS2}, that here we call $\overline{w}_{\lambda}$. In fact, the two elements $w_{\lambda}$ and $\overline{w}_{\lambda}$ are related as follows. For $\lambda \subseteq (n-k)^{k}$, consider the complementary partition $$\lambda' = (n-k-\lambda_k, \dots, n-k-\lambda_1) \subseteq (n-k)^{k}.$$ Then, $w_{\lambda} = \overline{w}_{\lambda'}$. 
\end{remark}

Let us now obtain reduced decompositions for $w_{\lambda}$ and $w_{\mu}$. For this, we use \cite[(2.2)]{CGGS2}, together with Remark \ref{rmk:comparison-CGGS2}. Note that if $\lambda'$ is as defined in Remark \ref{rmk:comparison-CGGS2} then $(\lambda')^{t} = (k-\lambda_{n-k}^{t}, k-\lambda_{n-k-1}^{t}, \dots, k-\lambda_1^{t})$. Thus,
\begin{equation}\label{eq:dec-w-lambda}
\begin{array}{rl}
w_{\lambda} = \overline{w}_{\lambda'} = & (s_{n-(\lambda')^{t}_{n-k}}\cdots s_{n-1})\cdots (s_{k+1-(\lambda')^{t}_{1}}\cdots s_k) \\ = & (s_{n-k+\lambda_1^{t}}\cdots s_{n-1})\cdots (s_{1 + \lambda^{t}_{n-k}}\cdots s_k)
\end{array}
\end{equation}
More precisely, we can write $w_\lambda = C_{\lambda, n-k}\cdots C_{\lambda, 1}$, where
\[
C_{\lambda, j} = (s_{j + \lambda^{t}_{n-k-j+1}}\cdots s_{k+j-1}).
\]
If some $\lambda_{j}^{t} = k$, we just skip these factors so $C_{\lambda, j} = 1$. 

We can picture the permutation $w_{\lambda}$ as follows. Starting on the southeast corner of the rectangle, label the vertical steps on the east side of the rectangle by $1, \dots, k$, and the horizontal steps on the north side of the rectangle by $k+1, \dots, n$. Similarly, starting on the southeast corner of the rectangle move towards the northwest corner following the shape of the partition $\lambda$, labeling these steps consecutively -- this is referred to as a \newword{$\lambda$-labeling.} The permutation $w_{\lambda}$ then starts at the rectangle labels and moves west or south to the $\lambda$ labels. See Figure \ref{fig:wlambda}.

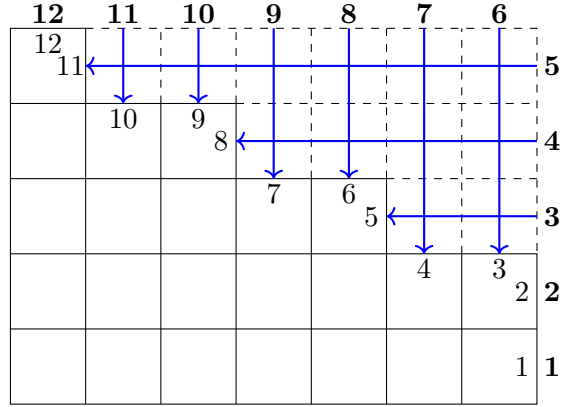
\begin{figure}
    \centering
    \begin{tikzpicture}
   
\draw (0,0)--(0,-5)--(7,-5)--(7,-3)--(5,-3)--(5,-2)--(3,-2)--(3,-1)--(1,-1)--(1,0)--(0,0);

\draw (0,-4)--(7,-4);
\draw (0,-3)--(5,-3);
\draw (0,-2)--(3,-2);
\draw (0,-1)--(1,-1);
\draw (1,-5)--(1,-1);
\draw (2,-5)--(2,-1);
\draw (3,-5)--(3,-2);
\draw (4,-5)--(4,-2);
\draw (5,-5)--(5,-3);
\draw (6,-5)--(6,-3);

\draw[dashed] (1,0) to (7,0) to (7,-3);
\draw[dashed] (3, -1) to (7,-1);
\draw[dashed] (5, -2) to (7,-2);
\draw[dashed] (2, -1) to (2,0);
\draw[dashed] (3, -1) to (3,0);
\draw[dashed] (4, -2) to (4,0);
\draw[dashed] (5, -2) to (5,0);
\draw[dashed] (6,-3) to (6,0);

\draw[color=blue,thick, ->] (6.5, 0) to (6.5, -3);
\draw[color=blue,thick, ->] (5.5, 0) to (5.5, -3);
\draw[color=blue,thick, ->] (4.5, 0) to (4.5, -2);
\draw[color=blue,thick, ->] (3.5, 0) to (3.5, -2);
\draw[color=blue,thick, ->] (2.5, 0) to (2.5, -1);
\draw[color=blue,thick, ->] (1.5, 0) to (1.5, -1);
\draw[color=blue,thick, ->] (7, -2.5) to (5, -2.5);
\draw[color=blue,thick, ->] (7, -1.5) to (3, -1.5);
\draw[color=blue,thick, ->] (7, -0.5) to (1, -0.5);

\node at (7.2, -4.5) {$\bf 1$};
\node at (7.2, -3.5) {$\bf 2$};
\node at (7.2, -2.5) {$\bf 3$};
\node at (7.2, -1.5) {$\bf 4$};
\node at (7.2, -0.5) {$\bf 5$};
\node at (6.5, 0.2) {$\bf 6$};
\node at (5.5, 0.2) {$\bf 7$};
\node at (4.5, 0.2) {$\bf 8$};
\node at (3.5, 0.2) {$\bf 9$};
\node at (2.5, 0.2) {$\bf 10$};
\node at (1.5, 0.2) {$\bf 11$};
\node at (0.5, 0.2) {$\bf 12$};

\node at (6.8, -4.5) {$1$};
\node at (6.8, -3.5) {$2$};
\node at (6.5, -3.2) {$3$};
\node at (5.5, -3.2) {$4$};
\node at (4.8, -2.5) {$5$};
\node at (4.5, -2.2) {$6$};
\node at (3.5, -2.2) {$7$};
\node at (2.8, -1.5) {$8$};
\node at (2.5, -1.2) {$9$};
\node at (1.5, -1.2) {$10$};
\node at (0.8, -0.5) {$11$};
\node at (0.5, -0.2) {$12$};

\end{tikzpicture}
    \caption{The permutation $w_\lambda$ for $\lambda = (7,7,5,3,1)$ considered inside the rectangle $(7^5)$. The factors $C_{\lambda, j}$ correspond to the columns of the complement of $\lambda$ inside the rectangle.}
    \label{fig:wlambda}
\end{figure}

\begin{remark}
    Note that \eqref{eq:dec-w-lambda} is a reduced decomposition of $w_{\lambda}$, that has length $k(n-k) - |\lambda| = |\lambda'|$.
\end{remark}

\begin{lemma}
\label{lem:additive-factorization}
    If $\mu \subseteq \lambda$ then there is a length-additive decomposition
    \[
    w_{\mu} = w_{\skewschubert}w_{\lambda},
    \]
    where
    \[
    w_{\skewschubert} = (s_{n-k+\mu_1^{t}}\cdots s_{n-k+\lambda_1^{t}-1})\cdots (s_{1+\mu_{n-k}^{t}}\cdots s_{\lambda_{n-k}^{t}})
    \]
    is a reduced decomposition. In particular, $w_{\mu}\geq w_{\lambda}$ in Bruhat order.
\end{lemma}

See Figure \ref{fig:additive-factorization-wmu} 
for a pictorial interpretation of Lemma \ref{lem:additive-factorization}. The permutation $w_{\skewschubert}$ starts at the $\lambda$ labels and moves west or south to the $\mu$ labels. 

\begin{definition}
\label{def: schubert}
Recall that  $\borel(n)$ is the Borel subgroup of upper triangular matrices in $\GL(n)$, 
 and $\borel_{-}(n)$ the opposite Borel subgroup of lower triangular matrices. 
Given two partitions $\mu\subseteq\lambda\subseteq (n-k)^k$, we define the \newword{skew shaped positroid} as the intersection of $\borel(n) w_{\mu}\borel(n)/\borel(n)$ and $\borel_{-}(n)w_{\lambda}\borel(n)/\borel(n)$. This is a subvariety of $\Fl(n)=\GL(n)/\borel(n)$ which will be denoted by $\skewschvar$. 
\end{definition}

Since $w_{\lambda} \leq w_{\mu}$ and $w_{\mu}$ is $k$-Grassmannian, we have that the function $f_{\lambda/\mu} = w_{\lambda}t_kw_{\mu}^{-1}$ is a $(k,n)$-bounded affine permutation, where $t_k = [1+n, \dots, k+n, k+1, \dots, n]$. We denote the corresponding positroid variety by $\Pi^{\circ}_{w_{\lambda}, w_{\mu}}$.

\begin{remark}\label{rmk: schubert cell isom}
The skew shaped positroid $\skewschvar$ is, by definition, the open Richardson variety $R_{w_{{\lambda}}, w_{\mu}}^{\circ}\subset \Fl(n)$. It is known \cite{KLS} that under the projection map $ \Fl(n)\to \Gr(k,n)$, the open Richardson variety $R_{w_{\lambda},w_{\mu}}^{\circ}$ is isomorphic to the open positroid variety $\Pi_{w_{\lambda},w_{\mu}}^{\circ}$.
\end{remark}

\begin{remark}\label{rmk: grassmannian richardson}
The skew shaped positroid $S^{\circ}_{\skewschubert}$ differs from the skew Schubert varieties studied in \cite{SSBW19}. The variety $S^{\circ}_{\skewschubert}$ is closely related to Grassmannian Richardson varieties \cite{eberhardt-stroppel} as follows. Let $\parabolic = \parabolic(k,n) \subseteq \GL(n)$ be the parabolic subgroup of block-upper triangular matrices, with diagonal blocks of sizes $k$ and $n-k$. It is known and easy to show that
\[
\Gr(k,n) = \GL(n)/\parabolic.
\]
For an element $w \in S_n$, the Grassmannian Schubert cell (resp. opposite Grassmannian Schubert cell) is the subvariety
\[
\borel(n)w\parabolic/\parabolic \subseteq \Gr(k,n) \qquad \text{(resp.} \qquad \borel_{-}(n)w\parabolic/\parabolic\text{).}
\]
These depend only on the coset $w(S_{k} \times S_{n-k})$ and not on $w$ itself. In particular, we can take minimal-length coset representatives, which are precisely the $k$-Grassmannian permutations. If $w_{\lambda}, w_{\mu}$ are $k$-Grassmannian permutations, the Grassmannian Richardson variety is
\[
\GR^{\circ}_{w_{\lambda}, w_{\mu}} := \left(\borel_{-}(n)w_{\lambda}\parabolic/\parabolic\right)\cap\left(\borel(n)w_{\mu}\parabolic/\parabolic\right) \subseteq \Gr(k,n).
\]
The variety $\GR^{\circ}_{w_{\lambda}, w_{\mu}}$ is nonempty if and only if $w_{\lambda} \leq w_{\mu}$ in Bruhat order, in which case it is an affine, smooth algebraic variety of dimension $\ell(w_{\mu}) - \ell(w_{\lambda})$. In fact, $w_{\lambda} \leq w_{\mu}$ if and only if $\mu \subseteq \lambda$, and in this case $\ell(w_{\mu}) - \ell(w_{\lambda}) = |\skewschubert|$. We will denote $\GR^{\circ}_{\skewschubert} := \GR^{\circ}_{w_{\lambda}, w_{\mu}}$.

The skew shaped positroid $S^{\circ}_{\skewschubert}$ is a dense subvariety of the Grassmannian Richardson variety $\GR^{\circ}_{\skewschubert}$. Moreover, the Grassmannian Richardson variety $\GR^{\circ}_{\skewschubert}$ may be expressed as the union of positroids
\[
\GR^{\circ}_{\skewschubert} = \bigcup_{\substack{w \leq w_{\mu} \\ w(S_{k} \times S_{n-k}) = w_{\lambda}(S_{k} \times S_{n-k})}}\Pi^{\circ}_{w, w_{\mu}}
\]
and $\dim(\Pi^{\circ}_{w, w_{\mu}}) \leq \dim(\GR^{\circ}_{\skewschubert})$, with equality if and only if $w = w_{\lambda}$. Indeed, if $w(S_{k}\times S_{n-k}) = w_{\lambda}(S_{k} \times S_{n-k})$, then the opposite Schubert cell in the flag variety $\borel_{-}(n)w\borel(n)/\borel(n)$ projects to the opposite Schubert cell $\borel_{-}(n)w_{\lambda}\parabolic/\parabolic$ in the Grassmannian, while the Schubert cell $\borel(n)w_{\mu}\borel(n)/\borel(n)$ projects to $\borel(n)w_{\mu}\parabolic/\parabolic$ so $\Pi^{\circ}_{w, w_{\mu}} \subseteq \GR^{\circ}_{\skewschubert}$. We also have $\ell(w_{\lambda}) \leq \ell(w)$, with equality if and only if $w = w_{\lambda}$. We thank Melissa Sherman-Bennett for discussions pertaining to this remark. 

\end{remark}
\begin{figure}
    \centering
    \begin{tikzpicture}
     \filldraw[color=lightgray] (0,-2)--(0,-5)--(3,-5)--(3,-3)--(2,-3)--(2,-2)--(0,-2);
\draw (0,0)--(0,-5)--(7,-5)--(7,-3)--(5,-3)--(5,-2)--(3,-2)--(3,-1)--(1,-1)--(1,0)--(0,0);

\draw (0,-4)--(7,-4);
\draw (0,-3)--(5,-3);
\draw (0,-2)--(3,-2);
\draw (0,-1)--(1,-1);
\draw (1,-5)--(1,-1);
\draw (2,-5)--(2,-1);
\draw (3,-5)--(3,-2);
\draw (4,-5)--(4,-2);
\draw (5,-5)--(5,-3);
\draw (6,-5)--(6,-3);

\draw[dashed] (1,0) to (7,0) to (7,-3);
\draw[dashed] (3, -1) to (7,-1);
\draw[dashed] (5, -2) to (7,-2);
\draw[dashed] (2, -1) to (2,0);
\draw[dashed] (3, -1) to (3,0);
\draw[dashed] (4, -2) to (4,0);
\draw[dashed] (5, -2) to (5,0);
\draw[dashed] (6,-3) to (6,0);

\draw[color=red, thick, ->] (6.5, 0) to (6.5, -3);
\draw[color=red, thick, ->] (5.5, 0) to (5.5, -3);
\draw[color=red, thick, ->] (4.5, 0) to (4.5, -2);
\draw[color=red, thick, ->] (3.5, 0) to (3.5, -2);
\draw[color=red, thick, ->] (2.5, 0) to (2.5, -1);
\draw[color=red, thick, ->] (1.5, 0) to (1.5, -1);
\draw[color=red, thick, ->] (7, -2.5) to (5, -2.5);
\draw[color=red, thick, ->] (7, -1.5) to (3, -1.5);
\draw[color=red, thick, ->] (7, -0.5) to (1, -0.5);

\draw[color=blue, thick, ->] (7, -4.5) to (3, -4.5);
\draw[color=blue, thick, ->] (7, -3.5) to (3, -3.5);
\draw[color=blue, thick, ->] (5, -2.5) to (2, -2.5);
\draw[color=blue, thick, ->] (3, -1.5) to (0, -1.5);
\draw[color=blue, thick, ->] (1, -0.5) to (0, -0.5);
\draw[color=blue, thick, ->] (0.5, 0) to (0.5, -2);
\draw[color=blue, thick, ->] (1.5, -1) to (1.5, -2);
\draw[color=blue, thick, ->] (2.5, -1) to (2.5, -3);
\draw[color=blue, thick, ->] (3.5, -2) to (3.5, -5);
\draw[color=blue, thick, ->] (4.5, -2) to (4.5, -5);
\draw[color=blue, thick, ->] (5.5, -3) to (5.5, -5);
\draw[color=blue, thick, ->] (6.5, -3) to (6.5, -5);
\end{tikzpicture}
    \caption{The length-additive decomposition $w_\mu = {\color{blue} w_{\skewschubert}}{\color{red} w_{\lambda}}$.}
    \label{fig:additive-factorization-wmu}
\end{figure}

\subsection{Lattice path labeling of a skew  Young diagram}\label{subsec:lattice path labeling for skew}

Given $1\le a\le n-k,\,1\le i\le k$ we define the square $\sq_{a,i}$ as the $i$-th box from the bottom in $a$-th column from the right. It is clear that  
$$
\sq_{a,i}\in \skewschubert\Leftrightarrow \mubar_a<i\leq \lambdabar_a.
$$
We denote
\begin{equation}
\label{eq: def di}
d_i=n-k+1-\lambda_i
\end{equation}
so that $\sq_{d_i,i}$ is the rightmost box in the $i$-th row of $\lambda$. 

\begin{definition}
\label{def: ribbon}
The \newword{boundary ribbon} $R(\lambda)$ of the diagram $\lambda$ consists of those boxes in $\lambda$ such that the box immediately northeast of it does not belong to $\lambda$. More precisely, $R(\lambda)$ consists of the boxes $\sq_{a,i}$ for $d_i\le a\le  d_{i+1}$ which correspond to the horizontal steps, and of boxes 
$\sq_{a,i}$ for $\lambdabar_{a-1}\le i\le \lambdabar_{a}$ which correspond to the vertical steps. The ribbon $R(\lambda)$ has exactly $\lambda_1+\lambdabar_{n-k}-1$ boxes in total. 

Given $\mu \subseteq \lambda$, we decompose
$$
R(\lambda)=R(\skewschubert)\cup \overline{R}(\skewschubert),
$$
where $$R(\skewschubert)=R(\lambda)\cap (\skewschubert),\quad \overline{R}(\skewschubert)=R(\lambda)\cap\mu.
$$
\end{definition}

\begin{example}
\label{ex: running}
Throughout the paper, we consider the running example with $\lambda=(7,7,5,3,1)$ and $\mu=(3,3,2)$. Here $n-k=7$ and $k=5$, so $n=12$. 
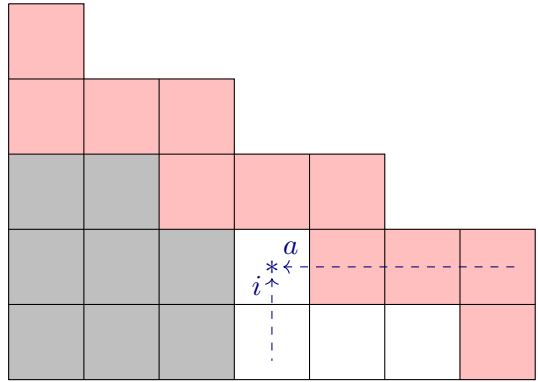
\begin{center}
\begin{tikzpicture}
\filldraw[color=lightgray] (0,-2)--(0,-5)--(3,-5)--(3,-3)--(2,-3)--(2,-2)--(0,-2);
\filldraw[color=pink] (0,0)--(0,-2)--(2,-2)--(2,-3)--(4,-3)--(4,-4)--(6,-4)--(6,-5)--(7,-5)--(7,-3)--(5,-3)--(5,-2)--(3,-2)--(3,-1)--(1,-1)--(1,0)--(0,0);
\draw (0,0)--(0,-5)--(7,-5)--(7,-3)--(5,-3)--(5,-2)--(3,-2)--(3,-1)--(1,-1)--(1,0)--(0,0);
\draw (0,-4)--(7,-4);
\draw (0,-3)--(5,-3);
\draw (0,-2)--(3,-2);
\draw (0,-1)--(1,-1);
\draw (1,-5)--(1,-1);
\draw (2,-5)--(2,-1);
\draw (3,-5)--(3,-2);
\draw (4,-5)--(4,-2);
\draw (5,-5)--(5,-3);
\draw (6,-5)--(6,-3);
\draw [navyblue](3.5,-3.5) node {$*$};
\draw [dashed,<-,navyblue] (3.65,-3.5)--(6.75,-3.5);
\draw [dashed,<-,navyblue] (3.5,-3.65)--(3.5,-4.75);
\draw (3.3,-3.75) node {\textcolor{navyblue}{$i$}};
\draw (3.75,-3.25) node {\textcolor{navyblue}{$a$}};
\end{tikzpicture}
\end{center}
The boundary ribbon $R(\lambda)=R(\skewschubert)$ is marked in \textcolor{pink!250}{pink}, $\overline{R}(\skewschubert)$ is empty and the box $\sq_{4,2}$ is labeled by $\textcolor{navyblue}{*}$. 
\end{example}

\begin{example}
\label{ex: disconnected}
For $n=9,\,k=4$, $\lambda=(5,5,2,2)$ and $\mu=(3,3)$ we get the following picture:
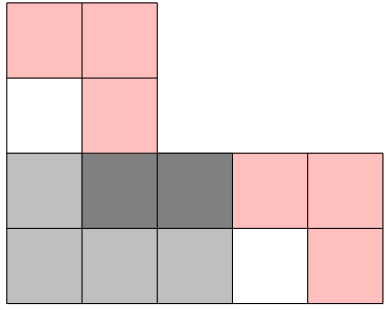
\begin{center}
\begin{tikzpicture}
\filldraw[color=lightgray] (0,-2)--(3,-2)--(3,-4)--(0,-4)--(0,-2);
\filldraw[color=gray] (1,-2)--(1,-3)--(3,-3)--(3,-2)--(1,-2);
\filldraw[color=pink] (0,0)--(0,-1)--(1,-1)--(1,-2)--(2,-2)--(2,0)--(0,0);
\filldraw[color=pink] (3,-2)--(3,-3)--(4,-3)--(4,-4)--(5,-4)--(5,-2)--(3,-2);
\draw (0,-4)--(5,-4);
\draw (0,-3)--(5,-3);
\draw (0,-2)--(5,-2);
\draw (0,-1)--(2,-1);
\draw (0,0)--(2,0);
\draw (0,0)--(0,-4);
\draw (1,-4)--(1,0);
\draw (2,-4)--(2,0);
\draw (3,-4)--(3,-2);
\draw (4,-4)--(4,-2);
\draw (5,-4)--(5,-2);
\end{tikzpicture}
\end{center}
The boundary ribbon 
$R(\skewschubert)$ is marked in \textcolor{pink!250}{pink} and $\overline{R}(\skewschubert)$  is marked in dark gray. Note that $R(\skewschubert)$ is disconnected.
\end{example}
\begin{definition}
Given a square $\sq_{a,i} \in (n-k)^{k}$, we define $\mu_{a,i}$ as the minimal Young diagram such that $\mu \subseteq \mu_{a,i}$ and $\sq_{a,i} \in \mu_{a,i}$. Note that $\mu_{a,i} = \mu$ if and only if $\sq_{a,i} \in \mu$,  and if $\sq_{a,i} \in \skewschubert$ then $\mu_{a,i} \subseteq \lambda$. We will identify $\mu_{a,i}$ with its \newword{boundary path} $P_{a,i}$ from the southeast to the northwest corner of the $(n-k)\times k$-rectangle.
\end{definition}

\begin{remark}In \cite[eq. (13.1)]{RW24} Rietsch and Williams define ``shapes" 
$$
\mathrm{sh}(a,i)=\mathrm{Rect}(a,i)\cup \mu
$$ 
where $\mathrm{Rect}(a,i)$ is the rectangle with northeast corner at $\sq_{a,i}$ and southwest corner at $\sq_{n-k,1}$. Clearly, the boundary of $\mathrm{sh}(a,i)$ is our path $P_{a,i}.$
\end{remark}

\begin{definition}
\label{def: I' and Imu}
Given a square $\sq_{a,i}\in \skewschubert$, we define its \newword{label} $I'(a,i)\subset \{1,\ldots,n\}$ as follows.
Label all the steps of the path $P_{a,i}$ with numbers from $1$ to $n$, starting at the southeast corner and  increasing as we go west and north. The labels of the {\bf vertical steps} in $P_{a,i}$ are the elements of $I'(a,i)$. Identifying $\mu$ (resp. $\lambda$) with its boundary path $P_{\mu}$ (resp. $P_{\lambda}$), we define the label $I_{\mu}$ (resp. $I_{\lambda}$) to be the set consisting of the labels of the vertical steps of $P_{\mu}$ (resp. $P_{\lambda}$). 

\end{definition}

\begin{example}
In our running example for $(a,i)=(4,2)$ we show the path $P_{a,i}$ as follows:
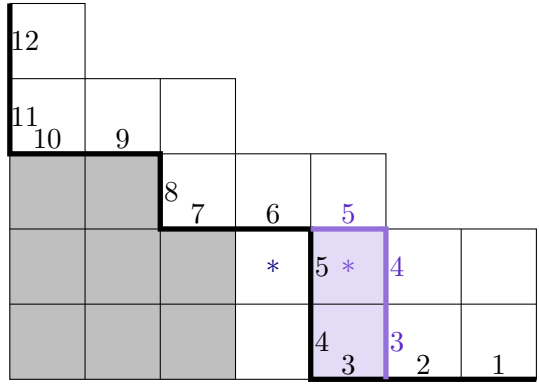
\begin{center}
\begin{tikzpicture} \label{fig: path picture in skew}
\filldraw[color=lightgray] (0,-2)--(0,-5)--(3,-5)--(3,-3)--(2,-3)--(2,-2)--(0,-2);
\filldraw[color=mediumpurple!25] (4,-3)--(4,-5)--(5,-5)--(5,-3)--(4,-3);
\draw (0,0)--(0,-5)--(7,-5)--(7,-3)--(5,-3)--(5,-2)--(3,-2)--(3,-1)--(1,-1)--(1,0)--(0,0);
\draw (0,-4)--(7,-4);
\draw (0,-3)--(5,-3);
\draw (0,-2)--(3,-2);
\draw (0,-1)--(1,-1);
\draw (1,-5)--(1,-1);
\draw (2,-5)--(2,-1);
\draw (3,-5)--(3,-2);
\draw (4,-5)--(4,-2);
\draw (5,-5)--(5,-3);
\draw (6,-5)--(6,-3);
\draw [navyblue](3.5,-3.5) node {$*$};
\draw[line width=2] (0,0)--(0,-2)--(2,-2)--(2,-3)--(4,-3)--(4,-5)--(7,-5);
\draw[line width=2,mediumpurple] (4,-3) -- (5,-3) -- (5,-5);
\draw (6.5,-4.8) node {$1$};
\draw (5.5,-4.8) node {$2$};
\draw (4.5,-4.8) node {$3$};
\draw (4.15,-4.5) node {$4$};
\draw (4.15,-3.5) node {$5$};
\draw (3.5,-2.8) node {$6$};
\draw (2.5,-2.8) node {$7$};
\draw (2.15,-2.5) node {$8$};
\draw (1.5,-1.8) node {$9$};
\draw (0.5,-1.8) node {$10$};
\draw (0.2,-1.5) node {$11$};
\draw (0.2,-0.5) node {$12$};
\draw [mediumpurple!130](4.5,-3.5) node {$*$};
\draw [mediumpurple!150](5.15,-4.5) node {$3$};
\draw [mediumpurple!150](5.15,-3.5) node {$4$};
\draw [mediumpurple!150](4.5,-2.8) node {$5$};
\end{tikzpicture}
\end{center}
Therefore, $I'(4,2)=\{4,5,8,11,12\}$ which is labeled by \textcolor{navyblue}{*}. We also draw an another path $P_{3,2}$ for $\sq_{3,2}$ with $I'(3,2)=\{3,4,8,11,12\}$, labeled by \textcolor{mediumpurple!130}{*}. We can see that $P_{4,2}$ and $P_{3,2}$ are the same except for the \textcolor{mediumpurple!150}{purple} rectangle region. We also note that $I_{\mu} = \{5,6,8,11,12\}$.
\end{example}

It should be noted that $I'(a,i)$ is a $k$-element subset of $\{1,\ldots,n\}$. A nice way to describe this $k$-element set is given by starting with the $k$-element set $I_{\mu}=\{b_{1}<b_{2}<\cdots<b_{k}\}$ and expressing the ordered subset $I(a,i)$ as 
\begin{equation}
    I(a,i)=(a,a+1,\ldots,a+i-1,b_{i+1},\ldots,b_{k}),
\end{equation}
where $1\le a\le n-k$ and $1\le i\le k$. From this set $I(a,i)$, we describe an ordered set 

\begin{equation}
    \label{eqn: I'}
    I'(a,i)=\{a_1,\ldots,a_i,b_{i+1},\ldots,b_k\},
\end{equation}
where $a_j = \min\{a+j-1, b_j\}$. Then we label $\sq_{a,i}$ with $I'(a,i)$, 
and the diagram $\mu$ with $I_{\mu}$.

\begin{example}
\label{ex: running labels}
In  Example \ref{ex: running} we label all squares $\sq_{a,i}\in  \skewschubert$ by $I'(a,i)$. We also put the label $I_{\mu}$ inside $\mu$. 
\begin{center}
\resizebox{4in}{2.5in}{
\begin{tikzpicture}
\filldraw[color=lightgray,scale=2] (0,-2)--(0,-5)--(3,-5)--(3,-3)--(2,-3)--(2,-2)--(0,-2);
\draw[scale=2] (0,0)--(0,-5)--(7,-5)--(7,-3)--(5,-3)--(5,-2)--(3,-2)--(3,-1)--(1,-1)--(1,0)--(0,0);
\draw[scale=2] (0,-4)--(7,-4);
\draw[scale=2] (0,-3)--(5,-3);
\draw[scale=2] (0,-2)--(3,-2);
\draw[scale=2] (0,-1)--(1,-1);
\draw[scale=2] (1,-5)--(1,-1);
\draw[scale=2] (2,-5)--(2,-1);
\draw[scale=2] (3,-5)--(3,-2);
\draw[scale=2] (4,-5)--(4,-2);
\draw[scale=2] (5,-5)--(5,-3);
\draw[scale=2] (6,-5)--(6,-3);
\draw (13,-9) node {\scriptsize $1,6,8,11,12$};
\draw (11,-9) node {\scriptsize $2,6,8,11,12$};
\draw (9,-9) node {\scriptsize $3,6,8,11,12$};
\draw (7,-9) node {\scriptsize $4,6,8,11,12$};
\draw (13,-7) node {\scriptsize $1,2,8,11,12$};
\draw (11,-7) node {\scriptsize $2,3,8,11,12$};
\draw (9,-7) node {\scriptsize $3,4,8,11,12$};
\draw (7,-7) node {\scriptsize $4,5,8,11,12$};
\draw (9,-5) node {\scriptsize $3,4,5,11,12$};
\draw (7,-5) node {\scriptsize $4,5,6,11,12$};
\draw (5,-5) node {\scriptsize $5,6,7,11,12$};
\draw (5,-3) node {\scriptsize $5,6,7,8,12$};
\draw (3,-3) node {\scriptsize $5,6,8,9,12$};
\draw (1,-3) node {\scriptsize $5,6,8,10,12$};
\draw (1,-1) node {\scriptsize $5,6,8,10,11$};
\draw (3,-7) node{\scriptsize $5,6,8,11,12$};
\end{tikzpicture}}
\end{center}
\end{example}
It will be convenient to shorten the labeling $I'(a,i)$ as follows.
\begin{definition}
Given a square $\sq_{a,i}\in \skewschubert$, we define its \newword{short path} $\overline{P}_{a,i}$ as the part of $P_{a,i}$ which stops after reaching the northeast corner of $\sq_{a,i}$ (but, in particular, includes the eastern boundary of $\sq_{a,i}$). 
We define the \newword{short label} $J(a,i)$ by reading the numbers at the vertical steps of $\overline{P}_{a,i}$. 
\end{definition}
Clearly, we have $|J(a,i)|=i$. Furthermore, the following result is clear from definitions and allows us to go back and forth between the short labels $J(a,i)$ and longer labels $I'(a,i)$.

\begin{proposition}\label{prop: labeling relationship between I'(a,i) and I mu}
The short label $J(a,i)$ consists of the first $i$ elements in $I'(a,i)$. The remaining $(k-i)$ elements of $I'(a,i)$ are the same as the last $(k-i)$ elements of the fixed subset $I_{\mu}$.
\end{proposition}

\begin{example}
\label{ex: running J}
In our running example  we can label all squares $\sq_{a,i}\in  \skewschubert$ by $J(a,i)$:
\begin{center}
\resizebox{4in}{2.5in}{
\begin{tikzpicture}
\filldraw[color=lightgray,scale=2] (0,-2)--(0,-5)--(3,-5)--(3,-3)--(2,-3)--(2,-2)--(0,-2);
\draw[scale=2] (0,0)--(0,-5)--(7,-5)--(7,-3)--(5,-3)--(5,-2)--(3,-2)--(3,-1)--(1,-1)--(1,0)--(0,0);
\draw[scale=2] (0,-4)--(7,-4);
\draw[scale=2] (0,-3)--(5,-3);
\draw[scale=2] (0,-2)--(3,-2);
\draw[scale=2] (0,-1)--(1,-1);
\draw[scale=2] (1,-5)--(1,-1);
\draw[scale=2] (2,-5)--(2,-1);
\draw[scale=2] (3,-5)--(3,-2);
\draw[scale=2] (4,-5)--(4,-2);
\draw[scale=2] (5,-5)--(5,-3);
\draw[scale=2] (6,-5)--(6,-3);
\draw (13,-9) node {$1$};
\draw (11,-9) node {$2$};
\draw (9,-9) node {$3$};
\draw (7,-9) node {$4$};
\draw (13,-7) node {$1,2$};
\draw (11,-7) node {$2,3$};
\draw (9,-7) node {$3,4$};
\draw (7,-7) node {$4,5$};
\draw (9,-5) node {$3,4,5$};
\draw (7,-5) node {$4,5,6$};
\draw (5,-5) node {$5,6,7$};
\draw (5,-3) node { $5,6,7,8$};
\draw (3,-3) node {$5,6,8,9$};
\draw (1,-3) node { $5,6,8,10$};
\draw (1,-1) node {\scriptsize $5,6,8,10,11$};
\end{tikzpicture}}
\end{center}
\end{example}

The following gives us a recursive description of $J(a,i)$.

\begin{lemma}
\label{lem: J inductive}
$ $\\\vspace{-0.5cm}
\begin{enumerate}[label=\alph*)]
    \item Assume that the boxes $\sq_{a,i}$ and $\sq_{a+1,i}$ are both in the diagram $\skewschubert$. Then $J(a+1,i)$ is obtained from $J(a,i)$ by swapping $a+\mubar_a$ by $a+i$. 
    \item Assume that the boxes $\sq_{a,i}$ and $\sq_{a,i+1}$ are both in the diagram $\skewschubert$. Then $J(a,i+1)$ is obtained from $J(a,i)$ by adding the element $a+i$. 
    \item In the above assumptions, $J(a,i)\subset J(a,i+1)$ and $J(a+1,i)\subset J(a,i+1)$.
\end{enumerate} 
\end{lemma}

\begin{proof}
Part (a) follows from the fact that $J(a,i)\setminus I(a,i)=\{a+\mubar_a,a+1+\mubar_a,\ldots, a+i-1\}$ and $J(a+1,i)\setminus I(a,i)=\{a+1+\mubar_{a},a+2+\mubar_{a}+\cdots,a+i\}$ (see Example \ref{fig: path picture in skew}). 
Part (b) follows from the observation that the short path $\overline{P}_{a, i+1}$ is obtained from $\overline{P}_{a,i}$ by adding an additional vertical step at the end. Thus, we must include the label of this vertical step, which equals the length of $\overline{P}_{a,i} + 1 = (a+i-1)+1$.
The first claim of part (c) follows from (b), and the second claim of (c) follows from (a). 
\end{proof}

\begin{corollary}
\label{cor: contained in the bottom}
We have $J(a,i)\subset J(a,\lambdabar_a)$.
\end{corollary}

We will also need the analogues of the labeling $J(a,i)$ for the rightmost boxes in each row of $\mu$.

\begin{definition}
Suppose that $\sq_{a,i}\in \mu$ but $\sq_{a-1,i}\in \skewschubert$. We define $\widetilde{J}(a,i)=\{b_1,\ldots,b_i\}$ where $I_{\mu}=\{b_1,\ldots,b_k\}$ as above.
\end{definition}

\begin{lemma}
\label{lem: J tilda inductive}
Suppose that $\sq_{a+1,i}\in \mu$ but $\sq_{a,i}\in \skewschubert$. Then:
\begin{enumerate}[label=\alph*)]
    \item $\widetilde{J}(a+1,i)$ is obtained from $J(a,i)$ by swapping $(a+\mubar_a)$ for $(a+i)$.
    \item Assume that  $\sq_{a,i},\sq_{a,i+1}$ are both in the diagram $\skewschubert$.  Then $J(a,i+1)=J(a,i)\cup \widetilde{J}(a+1,i)$.
\end{enumerate} 
\end{lemma}
\begin{proof}
This is completely analogous to Lemma \ref{lem: J inductive}.
\end{proof}

\subsection{Grassmann necklace for a skew shaped positroid}
\label{sec: Grassmann necklac}
We have seen in Remark \ref{rmk: schubert cell isom} that the skew shaped positroid is in fact a positroid variety. In this subsection, we use the combinatorics from the previous subsection to explicitly describe the corresponding Grassmann necklace. 
We define a Grassmann necklace associated to a skew diagram $\skewschubert$ using the boundary ribbon
$R(\lambda)=R(\skewschubert)\cup \overline{R}(\skewschubert)$ from Definition \ref{def: ribbon}. Recall the definition of labels $I'(a,i)$ and $I_{\mu}$ from Definition \ref{def: I' and Imu}.

\begin{definition}\label{def: skew grassmann necklace}
Given a skew diagram $\skewschubert$, we define the corresponding Grassmann necklace $\mathcal{I}_{\skewschubert}$ as follows:
\begin{itemize}
\item If $\lambda_1<n-k$, we define $\mathcal{I}_{\skewschubert,1}=\cdots=\mathcal{I}_{\skewschubert,n-k-\lambda_1}=I_{\mu}$.
\item For a box $\sq_{a,i}\in R(\skewschubert)$, we define $\mathcal{I}_{\skewschubert,a+i-1}=I'(a,i)$.
\item For a  box $\sq_{a,i}\in \overline{R}(\skewschubert)$ we define $\mathcal{I}_{\skewschubert,a+i-1}=I_{\mu}$.
\item If $\lambdabar_1<k$, we define
$\mathcal{I}_{\skewschubert,n-k+\lambdabar_1}=\cdots=\mathcal{I}_{\skewschubert,n-1}=I_{\mu}$.
\item Finally, $\mathcal{I}_{\skewschubert,n}=I_{\mu}$.
\end{itemize}
\end{definition}

\begin{lemma}
\label{lem: necklace for skew}
The collection of subsets $\mathcal{I}_{\skewschubert}$ is in fact a Grassmann necklace.
\end{lemma}

\begin{proof}
We first notice that $I'(a,i)$ and $I_{\mu}$ all have the same cardinality $k$. For the steps between $I_{\mu}$ and $I_{\mu}$, no additional verification is required. We now consider several cases:

{\bf Case 1a:} Horizontal step where both $\sq_{a,i},\sq_{a+1,i}$ are in $R(\skewschubert)$. Here $i=\lambdabar_a$, and by Lemma \ref{lem: J inductive} we have $I'(a+1,\lambdabar_a)=I'(a,\lambdabar_a)\setminus \{a+\mubar_a\}\cup \{a+\lambdabar_a\}$. Therefore
\begin{equation}
\label{eq: GN case 1}
\mathcal{I}_{\skewschubert,a+\lambdabar_a}=\mathcal{I}_{\skewschubert,a+\lambdabar_a-1}\setminus \{a+\mubar_a\}\cup \{a+\lambdabar_a\}.
\end{equation}

{\bf Case 1b:} Horizontal step where $\sq_{a,i}\in R(\skewschubert)$ but $\sq_{a+1,i}\in \overline{R}(\skewschubert)$. In this case, we similarly get
$I_{\mu}=I'(a,\lambdabar_a)\setminus \{a+\mubar_a\}\cup \{a+\lambdabar_a\}$ and \eqref{eq: GN case 1} holds.

{\bf Case 1c:} From $I'(n-k,\lambdabar_{n-k})$ to $I_{\mu}$.
If $\mubar_{n-k}<\lambdabar_{n-k}$ then $\mathcal{I}_{\skewschubert,n-k+\lambdabar_{n-k}-1}=I'(n-k,k)$ and (similarly to Case 1)  we get
\begin{equation}
\label{eq: GN case 1c}
\mathcal{I}_{\skewschubert,n-k+\lambdabar_{n-k}}=I_{\mu}=\mathcal{I}_{\skewschubert,n-k+\lambdabar_{n-k}-1}\setminus \{n-k+\mubar_{n-k}\}\cup \{n-k+\lambdabar_{n-k}\}.
\end{equation}

{\bf Case 2a:} Vertical step where both $\sq_{a,i},\sq_{a,i+1}$ are in $R(\skewschubert)$. In this case 
by  Lemma \ref{lem: J inductive} we have $I'(a,i+1)=I'(a,i)\setminus \{b_{i+1}\}\cup \{a+i\}$ and 
\begin{equation}
\label{eq: GN case 2}
\mathcal{I}_{\skewschubert,a+i}=\mathcal{I}_{\skewschubert,a+i-1}\setminus \{b_{i+1}\}\cup \{a+i\}.
\end{equation}

{\bf Case 2b:} Vertical step where $\sq_{a,i}\in \overline{R}(\skewschubert)$ but $\sq_{a,i+1}\in R(\skewschubert)$. In this case $I'(a,i+1)=I_{\mu}\setminus \{b_{i+1}\}\cup \{a+i\}$ and \eqref{eq: GN case 2} holds.

{\bf Case 2c:} From $I_{\mu}$ to $I'(n-k-\lambda_1,1)$. If $\mu_1<\lambda_1$, then
$\sq_{n-k-\lambda_1+1,1}\in R(\skewschubert)$ and $I'(n-k-\lambda_1+1,1)=\{n-k-\lambda_1+1,b_2,\ldots,b_k\}$. Therefore
\begin{equation}
\label{eq: GN case 2c}
\mathcal{I}_{\skewschubert,n-k-\lambda_1+1}= \mathcal{I}_{\skewschubert,n-k-\lambda_1}\setminus\{b_1\}\cup\{n-k-\lambda_1+1\}.
\end{equation}

\end{proof}

Let us denote $f_{\skewschubert}$ to be the bounded affine permutation 
corresponding to $\mathcal{I}_{\skewschubert}$ via the bijection from Lemma \ref{lem: GN to BA bijection}.
\begin{lemma}
\label{lem: baf for skew}
The bounded affine permutation $f_{\skewschubert}$
satisfies 
\begin{align}
\label{eq: f horizontal}
f_{\skewschubert}(a+\overline{\mu_a})&=a+\overline{\lambda_a},\quad  1\le a\le n-k\\
\label{eq: f vertical}
f_{\skewschubert}(b_i)=f_{\skewschubert}(n-k-\mu_i+i)&=(n-k-\lambda_{i}+i)+n,\quad 1\le i\le k.
\end{align}
\end{lemma}


\begin{proof}
It is easy to see that
$$
b_i=n-k-\mu_i+i.
$$
Furthermore, in \eqref{eq: GN case 2} we have $b_{i+1}=(n-k-\mu_{i+1})+i+1$ and $a=n-k-\lambda_{i+1}+1$, so $b_{i+1}>a+i$.
Now the statement mostly follows from the proof of Lemma \ref{lem: necklace for skew} and Lemma \ref{lem: GN to BA bijection}, with cases 1c and 2c corresponding to $a=n-k$ and $i=1$ respectively.
It remains to consider  the steps in $\mathcal{I}_{\skewschubert}$ from $I_{\mu}$ to $I_{\mu}$ which split into several cases:

{\bf Case 1:} If $\mu_1\le \lambda_1<n-k$, then $b_1=n-k-\mu_1+1$. Therefore for $1\le j\le n-k-\lambda_1$ we have $j\notin I_{\mu}$, so $f(j)=j$.

{\bf Case 2:} If $\lambdabar_{n-k}<k$ and $n-k+\lambdabar_{n-k}\le j\le n$, then $j\in I_{\mu}$, so $f(j)=j+n$.

{\bf Case 3:} If we have a horizontal step with both $\sq_{a,i},\sq_{a+1,i}\in \overline{R}(\skewschubert)$ then 
$i=\mubar_a=\lambdabar_a$ while $\mathcal{I}_{\skewschubert,a+i-1}=\mathcal{I}_{\skewschubert,a+i}=I_{\mu}$. Since $a+i=a+\mubar_a\notin I_{\mu}$, we get $f(a+\mubar_a)=a+\mubar_a=a+\lambdabar_a$, in agreement with \eqref{eq: f horizontal}.

{\bf Case 4:} If we have a vertical step with both $\sq_{a,i},\sq_{a,i+1}\in \overline{R}(\skewschubert)$ then
$a+i-1=b_i\in I_{\mu}$, so $f(b_i)=b_i+n$. On the other hand, in this case $b_i=n-k-\mu_i+i=n-k-\lambda_i+i$, in agreement with \eqref{eq: f vertical}.
\end{proof}

\begin{example}
\label{ex: necklace and f}
In our running example we have 
\[
    \mathcal{I}_{\skewschubert}=
    \begin{Bmatrix}        
    \{1{,}6{,}8{,}11{,}12\},\{1{,}2{,}8{,}11{,}12\},\{2{,}3{,}8{,}11{,}12\},\{3{,}4{,}8{,}11{,}12\},\{3{,}4{,}5{,}11{,}12\},\{4{,}5{,}6{,}11{,}12\},\\
    \{5{,}6{,}7{,}11{,}12\},\{5{,}6{,}7{,}8{,}12\},\{5{,}6{,}8{,}9{,}12\},\{5{,}6{,}8{,}10{,}12\},
    \{5{,}6{,}8{,}10{,}11\},\{5{,}6{,}8{,}11{,}12\}\end{Bmatrix}.
\]

Therefore the affine permutation $f=f_{\skewschubert}$ has the following values:
\begin{equation}
\label{eq: baf example horizontal}
f(1)=3,\ f(2)=4,\ f(3)=6,\ f(4)=7,\ f(7)=9,\ f(9)=10,\ f(10)=12
\end{equation}
and
\begin{equation}
\label{eq: baf example vertical}
f(6)=2+12,\ f(8)=5+12,\ f(11)=8+12,\ f(12)=11+12,\ f(5)=1+12.
\end{equation}
Note that the equations \eqref{eq: baf example horizontal} correspond to the horizontal steps in $R(\skewschubert)$ and \eqref{eq: baf example vertical} correspond to the vertical steps in it. Also, \eqref{eq: baf example vertical} describes the values $f(b_i)$ for $i=1,\ldots,k$.
\end{example}

\begin{example}
In Example \ref{ex: disconnected} we get
\begin{center}
\resizebox{3.5in}{2.5in}{
\begin{tikzpicture}
\filldraw[color=lightgray, scale=2] (0,-2)--(3,-2)--(3,-4)--(0,-4)--(0,-2);
\draw[scale=2] (0,-4)--(5,-4);
\draw[scale=2] (0,-3)--(5,-3);
\draw[scale=2] (0,-2)--(5,-2);
\draw[scale=2] (0,-1)--(2,-1);
\draw[scale=2] (0,0)--(2,0);
\draw[scale=2] (0,0)--(0,-4);
\draw[scale=2] (1,-4)--(1,0);
\draw[scale=2] (2,-4)--(2,0);
\draw[scale=2] (3,-4)--(3,-2);
\draw[scale=2] (4,-4)--(4,-2);
\draw[scale=2] (5,-4)--(5,-2);

\draw (9,-7) node {\scriptsize $1,4,8,9$};
\draw (9,-5) node {\scriptsize $1,2,8,9$};
\draw (7,-7) node {\scriptsize $2,4,8,9$};
\draw (7,-5) node {\scriptsize $2,3,8,9$};
\draw (3,-3) node {\scriptsize $3,4,6,9$};
\draw (3,-1) node {\scriptsize $3,4,6,7$};
\draw (1,-3) node {\scriptsize $3,4,7,9$};
\draw (1,-1) node {\scriptsize $3,4,7,8$};

\draw (3,-7) node {\scriptsize $3,4,8,9$};
\end{tikzpicture}}
\end{center}
The Grassmann necklace is 
$$
\mathcal{I}_{\skewschubert}=\begin{Bmatrix}
    \{1,4,8,9\},\{1,2,8,9\},\{2,3,8,9\},\{3,4,8,9\},\{3,4,8,9\},\\\{3,4,6,9\},\{3,4,6,7\},\{3,4,7,8\},\{3,4,8,9\}\}
\end{Bmatrix}.
$$

The bounded affine permutation is given by
$$
f(1)=3,\ f(2)=4,\ f(5)=5,\ f(6)=8, f(7)=9,
$$
and
$$
f(3)=1+9,\ f(4)=2+9,\ f(8)=6+9,\ f(9)=7+9.
$$
\end{example}

Let us define a wiring diagram for the bounded affine permutation $f_{\skewschubert}$. 

Given partitions $\mu \subseteq \lambda$, draw both the $\lambda$ and $\mu$-labeling in the same picture. For each horizontal step of $\mu$, labeled by $i$ in the $\mu$-labeling, draw a vertical arrow going up until reaching the boundary of $\lambda$, with $\lambda$-label $j$. Then, $f(i) = j$. If on the other hand $i$ labels a vertical step of $\mu$, draw a horizontal arrow going right until reaching the boundary of $\lambda$ with label $j$, and $f(i) = j + n$.

\begin{figure}[ht!]
    \centering
     \begin{tikzpicture}[scale=0.9]
     \filldraw[color=lightgray] (0,-2)--(0,-5)--(3,-5)--(3,-3)--(2,-3)--(2,-2)--(0,-2);
\draw (0,0)--(0,-5)--(7,-5)--(7,-3)--(5,-3)--(5,-2)--(3,-2)--(3,-1)--(1,-1)--(1,0)--(0,0);
\draw (0,-4)--(7,-4);
\draw (0,-3)--(5,-3);
\draw (0,-2)--(3,-2);
\draw (0,-1)--(1,-1);
\draw (1,-5)--(1,-1);
\draw (2,-5)--(2,-1);
\draw (3,-5)--(3,-2);
\draw (4,-5)--(4,-2);
\draw (5,-5)--(5,-3);
\draw (6,-5)--(6,-3);

\node at (-0.2,-0.5){\tiny \bf 12};
\node at (-0.2,-1.5){\tiny \bf 11};
\node at (0.5,-2.2){{\tiny\bf 10}};
\node at (1.5,-2.2){\tiny \bf 9};
\node at (1.7,-2.5){\tiny \bf 8};

\node at (2.5, -3.2) {\tiny \bf 7};
\node at (2.7, -3.5) {\tiny \bf 6};
\node at (2.7, -4.5) {\tiny \bf 5};
\node at (3.5, -5.2) {\tiny \bf 4};
\node at (4.5, -5.2) {\tiny \bf 3};
\node at (5.5, -5.2) {\tiny \bf 2};
\node at (6.5, -5.2) {\tiny \bf 1};

\node at (0.5, 0.25) {\tiny 12};
\node at (1.25, -0.3) {\tiny 11};
\node at (1.6, -0.75) {\tiny  10};
\node at (2.5, -0.75) {\tiny  9};
\node at (3.25, -1.3) {\tiny 8}; 
\node at (3.6, -1.75) {\tiny 7};
\node at (4.5, -1.75) {\tiny 6};
\node at (5.25, -2.3) {\tiny 5};
\node at (5.6, -2.75) {\tiny 4};
\node at (6.5, -2.75) {\tiny 3};
\node at (7.25, -3.5) {\tiny 2};
\node at (7.25, -4.5) {\tiny 1};

\draw[color=blue,<-] (0.5, 0) to (0.5, -2);
\draw[color=blue,<-] (1.5, -1) to (1.5, -2);
\draw[color=blue,<-] (2.5, -1) to (2.5, -3);
\draw[color=blue,<-] (3.5, -2) to (3.5, -5);
\draw[color=blue,<-] (4.5, -2) to (4.5, -5);
\draw[color=blue,<-] (5.5, -3) to (5.5, -5);
\draw[color=blue,<-] (6.5, -3) to (6.5, -5);

\draw[color=blue,<-] (1, -0.5) to (0,-0.5);
\draw[color=blue,<-] (3, -1.5) to (0,-1.5);
\draw[color=blue,<-] (5, -2.5) to (2,-2.5);
\draw[color=blue,<-] (7, -3.5) to (3, -3.5);
\draw[color=blue,<-] (7, -4.5) to (3, -4.5);
\end{tikzpicture} \qquad 
\begin{tikzpicture}[scale=0.9]
\filldraw[color=lightgray] (0,-2)--(3,-2)--(3,-4)--(0,-4)--(0,-2);
\draw (0,-4)--(5,-4);
\draw (0,-3)--(5,-3);
\draw (0,-2)--(5,-2);
\draw (0,-1)--(2,-1);
\draw (0,0)--(2,0);
\draw (0,0)--(0,-4);
\draw (1,-4)--(1,0);
\draw (2,-4)--(2,0);
\draw (3,-4)--(3,-2);
\draw (4,-4)--(4,-2);
\draw (5,-4)--(5,-2);

\node at (0.5, 0.25) {\tiny 9};
\node at (1.5, 0.25) {\tiny 8};
\node at (2.2, -0.5) {\tiny 7};
\node at (2.2, -1.5) {\tiny 6};
\node at (2.5, -1.75) {\tiny 5};
\node at (3.5, -1.75) {\tiny 4};
\node at (4.5, -1.75) {\tiny 3};
\node at (5.2, -2.5) {\tiny 2};
\node at (5.2, -3.5) {\tiny 1};

\node at (-0.2, -0.5) {\tiny \bf 9};
\node at (-0.2, -1.5) {\tiny \bf 8};
\node at (0.5, -2.2) {\tiny \bf 7};
\node at (1.5, -2.2) {\tiny \bf 6};
\node at (2.5, -2.2) {\tiny \bf 5};
\node at (2.7, -2.5) {\tiny \bf 4};
\node at (2.7, -3.5) {\tiny \bf 3};
\node at (3.5, -4.2) {\tiny \bf 2};
\node at (4.5, -4.2) {\tiny \bf 1};

\draw[color=blue,<-] (0.5, 0) to (0.5, -2);
\draw[color=blue,<-] (1.5, 0) to (1.5, -2);
\draw[color=blue,<-] (3.5, -2) to (3.5, -4);
\draw[color=blue,<-] (4.5, -2) to (4.5, -4);

\draw[color=blue,<-] (2, -0.5) to (0, -0.5);
\draw[color=blue,<-] (2, -1.5) to (0, -1.5);
\draw[color=blue,<-] (5, -2.5) to (3,-2.5);
\draw[color=blue,<-] (5, -3.5) to (3, -3.5);
\end{tikzpicture}

    \caption{The wiring diagram associated to the bounded affine permutation $f_{\skewschubert}$. For horizontal arrows, we have to add $n$ to the value of the target. Note that $5$ is a fixed point of $f_{\skewschubert}$ for the right-hand figure.}
    \label{fig:wiring-diag-affine}
\end{figure}
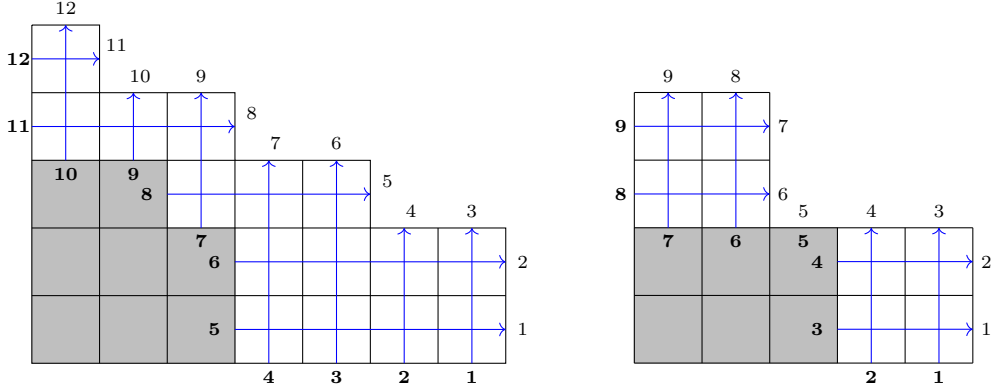

\begin{lemma}
\label{lem: f factorization}
We have $f_{\skewschubert} = w_{\lambda}t_kw_{\mu}^{-1},$ where 
$t_k=[1+n, \dots, k+n, k+1, \dots, n]$.
\end{lemma}
\begin{proof}
We have $w_{\lambda}(i)=n-k-\lambda_i+i$ for $1\le i\le k$ and $w_{\lambda}(k+a)=a+\lambdabar_a$ for $1\le a\le n-k$. Therefore 
$$
w_{\lambda}t_kw_{\mu}^{-1}(a+\mubar_a)=w_{\lambda}t_k(k+a)=w_{\lambda}(k+a)=a+\lambdabar_a
$$
and
$$
w_{\lambda}t_kw_{\mu}^{-1}(n-k-\mu_i+i)=w_{\lambda}t_k(i)=w_{\lambda}(i+n)=n-k-\lambda_i+i.
$$
This agrees with $f_{\skewschubert}$ by Lemma \ref{lem: baf for skew}.
\end{proof}

\begin{corollary}
    Given a skew diagram $\skewschubert$, we have that $f_{\skewschubert} \pmod{n} = w_{\skewschubert}^{-1}$.%
\end{corollary}
\begin{proof}
    Indeed, we have that $f = w_{\lambda}t_kw_{\mu}^{-1}$, while $w_{\skewschubert} = w_{\mu}w_{\lambda}^{-1}$ and $t_k \pmod n$ is the identity. 
\end{proof}




The following statement can be easily deduced from the results of \cite{KLS}.

\begin{theorem}
\label{thm: skew as positroid}
The skew shaped positroid $S_{\skewschubert}^{\circ}$ is isomorphic to the open positroid variety $\Pi^{\circ}_{\mathcal{I}_{\skewschubert}}$. 
\end{theorem}

\begin{proof}
By definition, the skew shaped positroid $\skewschvar$ is the open Richardson variety $R_{w_\lambda,w_\mu}^\circ$. Since the permutation $w_{\mu}$ is $k$-Grassmannian, the natural projection from the flag variety $\Fl(n)$ to $\Gr(k,n)$ identifies $R_{w_\lambda, w_\mu}^\circ$ with its image, cf. \cite{KLS}. By \cite{KLS}, this projection identifies $R_{w_\lambda,w_\mu}^\circ$ with the positroid $\Pi_f^{\circ}$ where 
$f=w_{\lambda}t_{k}w_{\mu}^{-1}$ which agrees with $f_{\skewschubert}$ by Lemma \ref{lem: f factorization}.  
\end{proof}

\begin{remark}\label{rmk:Imu1}
    Let $V$ be a $(k \times n)$-matrix representing a point in the positroid $\Pi^{\circ}_{\mathcal{I}_{\skewschubert}}$. Since $I_{\mu}$ is an element of the corresponding Grassmann necklace, we have that $\minor_{I_{\mu}}(V) \neq 0$. By rescaling, we may and will assume that $\minor_{I_{\mu}}(V) = 1$ for every element $V \in \Pi^{\circ}_{\mathcal{I}_{\skewschubert}}$.
\end{remark}

\subsection{Skew shaped positroid as a braid variety}\label{sec:braid diagram configuration for skew}
Following \cite[Lemma 2.5]{CGGS2} we define a $k$-strand braid $\br_{\skewschubert}$  associated to a skew diagram $\skewschubert$ via the following procedure.
 
\begin{enumerate}
    \item Cross out the top box of every column of $\lambda$.
    \item For each column of $(n-k)^{k}$, put strands starting on the lower left corner of each box in this column, so that there are precisely $k$ strands. We join the strands to corners in the right border of this column, as follows:
    \begin{enumerate}
    \item Join the strand starting in the lower left corner of the crossed box belonging to $\skewschubert$ (from Step 1) in this column to the lower right corner of the lowest box in the same column belonging to $\skewschubert$. If the crossed box also belongs to $\mu$, then simply join the strand, in a straight line, to the lower right corner of the same box.
    \item If a strand starts in the lower left corner of a non-crossed box in $\skewschubert$, join it to the upper right corner of the same box.
    \item If a strand starts on the lower left corner of any box which is not of the above type, join it (in a straight line) to the lower right corner of the same box.  
    \end{enumerate}
    \item The braid $\br_{\skewschubert}$ is defined to be the concatenation of the $(n-k)$-braids from Step 2. In formulas,
    \begin{equation}\label{eq: k-strand-braid}
\br_{\skewschubert} = C_1C_2\cdots C_{n-k}, \qquad C_{j} = \sigma_{\lambda_{j}^{t}-1}\sigma_{\lambda_j^{t}-2}\cdots \sigma_{\mu_j^{t}+1},
    \end{equation}
    where $C_j$ is the empty braid if $\mu_j^{t}+1 > \lambda_{j}^{t}-1$. 
\end{enumerate}

See Example \ref{ex:running k-strand} for an example of $\beta_{\skewschubert}$.
 We say that a box $\sq \in \skewschubert$ is a \newword{braid box} if it is not the top row of its column. 
 By definition, 
    the braid boxes of $\skewschubert$ are in bijection with the crossings of $\beta_{\skewschubert}$. 
    
    Next, we define a family of subspaces of $\C^k$ labeled by the braid boxes of $\skewschubert$ or, equivalently, by the regions of $\beta_{\skewschubert}$.

\begin{definition}\label{def: V(a,i)}
We define the subspace $V(a,i)=\spann\{v_j:j\in J(a,i)\}\subset \C^k$.
\end{definition}

\begin{lemma}
\label{lem: dimension V}
We have $\dim V(a,i)=i$.
\end{lemma}

\begin{proof}
We have $J(a,i)\subset J(a,\lambdabar_a)\subset I'(a,\lambdabar_a)$  by Corollary \ref{cor: contained in the bottom}, and $\Delta_{I'(a,\lambdabar_a)}(V)\neq 0$ by definition of the Grassmann necklace $\mathcal{I}_{\skewschubert}$. Therefore the vectors $\{v_{j}:j\in I'(a,\lambdabar_a)\}$ are linearly independent, and the vectors $\{v_{j}:j\in J(a,i)\}$ are linearly independent as well.
\end{proof}


\begin{lemma}
\label{lem: containment V}
We have containment
$$
V(a,i)\subset V(a,i+1),\ V(a+1,i)\subset V(a,i+1).
$$
\end{lemma}

\begin{proof}
This follows from Lemma \ref{lem: J inductive}.
\end{proof}

\begin{lemma}
\label{lem: equality V}
Suppose that $\sq_{a,i}$ and $\sq_{a+1,i}$ are any two adjacent squares chosen from the $i$-th row of $R(\skewschubert)$. Then $V(a,i)=V(a+1,i)$.
\end{lemma}

\begin{proof}
Assume that $\sq_{a,i}$ and $\sq_{a+1,i}$ are both in $R(\skewschubert)$. Lemma \ref{lem: J inductive} implies that $J(a+1,i)$ is obtained from $J(a,i)$ by swapping $a+\mubar_a$ by $a+\lambdabar_a$. Combining Lemma \ref{lem: baf for skew} with Definition \ref{def: Grassmann necklace for V}, we have a containment 
$V(a+1,i)\subset V(a,i)$. Since $V(a,i)$ and $V(a+1,i)$ have the same dimension $i$, we can conclude that $V(a,i)=V(a+1,i)$.
\end{proof}


\begin{example}
In Example \ref{ex: running J} we have $J(5,4)=\{5,6,7,8\}$ and $J(6,4)=\{5,6,8,9\}$. Since $f(7)=9$, we have $v_7\in \langle v_8,v_9\rangle$, so $V(5,4)=\langle v_5,v_6,v_7,v_8\rangle\subset \langle v_5,v_6,v_8,v_9\rangle=V(6,4)$. On the other hand, $\dim V(5,4)=\dim V(6,4)=4$, so $V(5,4)=V(6,4)$. 
\end{example}

\begin{lemma}
\label{lem: transversality V}
Suppose that $\sq_{a,i}$ is not in $R(\skewschubert)$ and $\sq_{a+1,i}\in \skewschubert$. Then $$V(a,i)\neq V(a+1,i).$$
\end{lemma}

\begin{proof}
By Lemma \ref{lem: J inductive} we get $J(a,i)\cup J(a+1,i)=J(a,i+1)$, so $V(a,i)+V(a+1,i)=V(a,i+1)$. If $V(a,i)=V(a+1,i)$ then $V(a,i+1)=V(a,i)$ which contradicts Lemma \ref{lem: dimension V}.
\end{proof}

\begin{definition}\label{def: skew schubert standard flag}
Let $I_{\mu}=\{b_1,\ldots,b_k\}$ as above, we define 
$$
W_i=\left\langle v_{b_{k-i+1}},\ldots,v_{b_k}\right\rangle,\quad \Wop_i=\left\langle v_{b_1},\ldots,v_{b_i}\right\rangle.
$$
\end{definition}

\begin{lemma}
\label{lem: V tilde}
The subspaces $\Wop_i$ have the following properties:
\begin{enumerate}[label=\alph*)]
    \item $\dim \Wop_i=i$.
    \item Suppose that $\sq_{a,i}\in \mu$ but $\sq_{a-1,i}\in \skewschubert$. Then $\Wop_i\subset V(a-1,i+1)$ and $\Wop_i\neq V(a-1,i)$.
\end{enumerate}
\end{lemma}

\begin{proof}
Part (a) follows from $\Delta_{I_{\mu}}\neq 0$. For part (b), we first observe that
$$
\Wop_i=\spann\left\{v_j\ :\ j\in \widetilde{J}(a,i)\right\}, 
$$
so by Lemma \ref{lem: J tilda inductive} we get 
$\Wop_i+V(a-1,i)=V(a-1,i+1)$, in particular $\Wop_i\subset V(a-1,i+1)$. If $\Wop_i=V(a-1,i)$ then $\dim V(a-1,i+1)=i$, contradiction.
\end{proof}
It is easy to see that the vector $v_{a+\mubar_a}$ vanishes if and only if $\mubar_a=\lambdabar_a$, that is, the $a$-th column of the skew diagram $\skewschubert$ is empty (this includes the case $\mubar_a=\lambdabar_a=0$). Indeed, by Lemma \ref{lem: baf for skew} $f(a+\mubar_a)=a+\lambdabar_a$ and by definition of the bounded affine permutation $f$ the vector $v_j$ vanishes if and only if $f(j)=j$. Below we will mostly focus on the case $\mubar_a<\lambdabar_a$ where   $v_{a+\mubar_a}\neq 0$.

\begin{lemma}
\label{lem: one dimensional}
For all $a$, if $\mu_a<\lambda_a$ then the space 
$$
V(a,\mubar_a+1)\cap \left\langle v_{a+\mubar_a+1},\ldots,v_{a+\lambdabar_a}\right\rangle
$$
is one-dimensional and spanned by $v_{a+\mubar_a}$. 
\end{lemma}

\begin{proof}
We have $a+\mubar_a\in J(a,\mubar_a+1)$, so
$v_{a+\mubar_a}\in V(a,\mubar_a+1)$. Furthermore, by Lemma \ref{lem: baf for skew} we have $f(a+\mubar_a)=a+\lambdabar_a$, so 
$$
v_{a+\mubar_a}\in \left\langle v_{a+\mubar_a+1},\ldots,v_{a+\lambdabar_a}\right\rangle.
$$
Next, observe that 
$$
J(a,\lambdabar_a)=J(a,\mubar_a)\cup 
\left\{a+\mubar_a+1,\ldots,a+\lambdabar_a-1\right\}
$$
so 
\begin{align*}
    \dim \left[V(a,\mubar_a+1)+\left\langle v_{a+\mubar_a+1},\ldots,v_{a+\lambdabar_a}\right\rangle\right]&\ge \dim \left[V(a,\mubar_a+1)+\left\langle v_{a+\mubar_a+1},\ldots,v_{a+\lambdabar_a-1}\right\rangle\right]\\&=\dim V(a,\lambdabar_a)=\lambdabar_a
\end{align*}

Also, $\dim \left\langle v_{a+\mubar_a+1},\ldots,v_{a+\lambdabar_a}\right\rangle\le \lambdabar_a-\mubar_a$.
Therefore
$$
\dim \left[V(a,\mubar_a+1)\cap \left\langle v_{a+\mubar_a+1},\ldots,v_{a+\lambdabar_a}\right\rangle\right]=\dim V(a,\mubar_a+1)+\dim \left\langle v_{a+\mubar_a+1},\ldots,v_{a+\lambdabar_a}\right\rangle-
$$
$$
\dim \left[V(a,\mubar_a+1)+\left\langle v_{a+\mubar_a+1},\ldots,v_{a+\lambdabar_a}\right\rangle\right]\le \left(\mubar_a+1\right)+\left(\lambdabar_a-\mubar_a\right)-\lambdabar_a=1.
$$
\end{proof}

\begin{lemma}
\label{lem: v in W}
For all $a$, we have $v_{a+\mubar_a}\in W_{k-\mubar_a}$.
\end{lemma}

\begin{proof}
We prove it by decreasing induction on $a$ where $1\le a\le n-k$. Note that 
$$
W_{k-\mubar_a}=\left\langle v_{b_{\mubar_a+1}},\ldots,v_{b_k}\right\rangle=\spann\left\{v_{b_i}\mid b_i>a+\mubar_a\right\}.
$$
By Lemma \ref{lem: baf for skew} we have $f(a+\mubar_a)=a+\lambdabar_a$, so 
$$
v_{a+\mubar_a}\in \left\langle v_{a+\mubar_a+1},\ldots,v_{a+\lambdabar_a}\right\rangle.
$$
For all $t$ such that $a+\mubar_a+1\le t\le a+\lambdabar_a$, we either have $t\in I_{\mu}$ and then $v_t\in W_{k-\mubar_a}$, or $t=a'+\mubar_{a'}$ for $a'>a$ and by the assumption of induction $v_t\in W_{k-\mubar_{a'}}\subset W_{k-\mubar_{a}}$. Therefore $\left\langle v_{a+\mubar_a+1},\ldots,v_{a+\lambdabar_a}\right\rangle\subset W_{k-\mubar_{a}}$ and we are done.
\end{proof}

\begin{lemma}
\label{lem: 1d intersection new}
For all $a$, if $\mu_a<\lambda_a$ then the intersection 
$$
V(a,\mubar_a+1)\cap W_{k-\mubar_a}
$$
is one-dimensional and spanned by $v_{a+\mubar_a}$.
\end{lemma}

\begin{proof}
Let $L= V(a,\mubar_a+1)\cap W_{k-\mubar_a}$.
We have $J(a,\mubar_a+1)=\{b_1,\ldots,b_{\mubar_a},a+\mubar_a\}$, so
$$
V(a,\mubar_a+1)=\Wop_{\mubar_a}+\left\langle v_{a+\mubar_a}\right\rangle.
$$
By Lemma \ref{lem: v in W} we have $v_{a+\mubar_a}\in V(a,\mubar_a+1)\cap W_{k-\mubar_a}$ and $\dim L\ge 1$. On the other hand, $\Wop_{\mubar_a}\cap W_{k-\mubar_a}=0$, so $L\cap \Wop_{\mubar_a}=0$. At the same time, $L+\Wop_{\mubar_a}\subset V(a,\mubar_a+1)$, therefore $\dim L+\dim \Wop_{\mubar_a}\le \dim V(a,\mubar_a+1)$ and $\dim L\le 1.$
\end{proof}

\begin{definition}
We define $R^1(\skewschubert)\subset R(\skewschubert)$ as the set of boxes   $\sq_{a,\lambdabar_a}$ for $a=1,\ldots,n-k$ (whenever $\sq_{a,\lambdabar_a}\in \skewschubert$) . Equivalently, $R^1(\skewschubert)$ the top boxes of columns of $\lambda$ that are not in $\mu$, that is, the boxes in $\skewschubert$ that are crossed out in the construction of $\beta_{\skewschubert}.$

We define $S^{\circ,1}_{\skewschubert}\subset S^{\circ}_{\skewschubert}$ by the conditions
$$
\minor_{I'(a,i)}=1,\ \sq_{a,i}\in R^1(\skewschubert).
$$
\end{definition}

\begin{theorem}
\label{thm: from skew schubert to braid variety}
For any skew diagram $\skewschubert$ let $s=|R^1(\skewschubert)|$. Then there exist isomorphisms
$$
\widetilde{\Omega}: S^{\circ}_{\skewschubert}\xrightarrow{\sim} X\left(\longest\beta_{\skewschubert}\right)\times \left(\C^\times\right)^{s},\ 
\widetilde{\Omega}^1: S^{\circ,1}_{\skewschubert}\xrightarrow{\sim} X\left(\longest\beta_{\skewschubert}\right).
$$
\end{theorem}


\begin{proof}
{\bf Step 1:} 
Recall the equivalent descriptions of $X\left(\longest\beta_{\skewschubert}\right)$ from Lemma \ref{lem: flags}.
First, we construct a map  
\begin{equation}
\label{eq: Schubert to braid}
\widetilde{\Omega}: S^{\circ}_{\skewschubert}\rightarrow X\left(\longest\beta_{\skewschubert}\right)\times \left(\C^\times\right)^{s}
\end{equation}
by its components.
The map $S^{\circ}_{\skewschubert}\to \left(\C^\times\right)^{s}$ simply sends a matrix $V$ to the tuple of minors $\minor_{I'(a,i)}$ for $\sq_{a,i}\in R^1(\skewschubert)$. To construct the map 
$
\Omega: S^{\circ}_{\skewschubert}\rightarrow X\left(\longest\beta_{\skewschubert}\right),
$ we label the regions of the braid diagram of $\beta_{\skewschubert}$ by subspaces of $\C^k$ as follows.

We have a bijection between the braid boxes of $\skewschubert$ and the crossings of $\beta_{\skewschubert}$. Given a crossing, we label the region right of it by the subspace $V(a,i)$ where $\sq_{a,i}$ is the corresponding box. Furthermore, we choose the  basis $v_{b_1},\ldots,v_{b_k}$ for the flag on the left so that  the regions on the left  boundary of $\beta_{\skewschubert}$ are labeled by the subspaces $\Wop_i$. 
We need to check that such labeling 
satisfies the conditions in Lemma \ref{lem: flags}(c). 

All subspaces associated to the regions of $\beta_{\skewschubert}$ have correct dimension by Lemma \ref{lem: dimension V}. Furthermore, $\minor_{I_{\mu}}\neq 0$ and $v_{b_1},\ldots,v_{b_k}$ are linearly independent and $\dim \Wop_i=i$.

By Lemmas \ref{lem: containment V}, \ref{lem: equality V} and \ref{lem: transversality V} we ensure that the subspaces in the regions of $\beta_{\skewschubert}$ are in correct relative position. Note that for $i=k$ we get $k$-dimensional space which coincides with $\C^k$. By Lemma \ref{lem: V tilde} the subspaces $\Wop_i$ are in correct relative position with the subspaces in regions of $\beta_{\skewschubert}$. 

Finally, we need to check that the rightmost flag is transversal to $\CF^W$ which is the opposite of $\CF^{\Wop}$. 
By construction, the rightmost flag $\CF^0$ consists of subspaces $V(d_i,i)=\spann J(d_i,i)$ where $\sq_{d_i,i}$ is the  rightmost square in the $i$-th row of $\lambda$ as in \eqref{eq: def di}. For such squares, we have
$$
I'(d_i,i)=J(d_i,i)\cup \{b_{i+1},\ldots,b_k\},
$$
and  $W_{k-i}=\left\langle v_{b_{i+1}},\ldots,v_{b_k}\right\rangle$. Since $\Delta_{I'(d_i,i)}\neq 0$, the subspaces $V(d_i,i)$ and $W_{k-i}$ are indeed transversal. To sum up, we get the configuration
\begin{equation}
\label{eq: flags from V}
\Omega(V)=\left[\CF^{W}\stackrel{\longest}{\dashrightarrow}\CF^{\Wop}\stackrel{\beta_{\skewschubert}}{\dashrightarrow}\CF^0\right].
\end{equation}

\begin{figure}
    \centering
    \begin{tikzpicture}
        \draw (-2,1) to (-1.5,1) to[out=0, in=180] (-1,2) to[out=0, in=180] (-0.5,3) to (0,3) to[out=0, in=180] (0.5,4) to (1,4) to[out=0, in=180] (1.5,5) to (4,5) to[out=0, in=180] (4.5,4) to[out=0, in=180] (5,3) to[out=0, in=180] (5.5,2) to[out=0, in=180] (6,1) to (7,1) to[out=0, in=180] (7.5,2) to (8,2) to[out=0, in=180] (8.5,1) to (9,1) to[out=0, in=180] (9.5,2) to (10,2);

        \draw (-2,2) to (-1.5,2) to[out=0, in=180] (-1,1) to (-0.5,1) to[out=0, in=180] (0,2) to (0.5,2) to[out=0, in=180] (1,3) to (1.5,3) to[out=0, in=180] (2,4) to (4,4) to[out=0, in=180] (4.5,5) to (10,5);

        \draw (-2,3) to (-1,3) to[out=0, in=180] (-0.5,2) to[out=0, in=180] (0,1) to (2,1) to[out=0, in=180] (2.5,2) to[out=0, in=180] (3,3) to (4.5,3) to[out=0, in=180] (5,4) to (10,4);

        \draw (-2,4) to (0,4) to[out=0, in=180] (0.5,3) to[out=0, in=180] (1,2) to (2,2) to[out=0, in=180] (2.5,1) to (3,1) to[out=0, in=180] (3.5,2) to (5,2) to[out=0, in=180] (5.5,3) to (6.5,3) to[out=0, in=180] (7,2) to[out=0, in=180] (7.5,1) to (8,1) to[out=0, in=180] (8.5,2) to (9,2) to[out=0, in=180] (9.5,1) to (10,1);

        \draw (-2,5) to (1,5) to[out=0, in=180] (1.5,4) to[out=0, in=180] (2,3) to (2.5,3) to[out=0, in=180] (3,2) to[out=0, in=180] (3.5,1) to (5.5,1) to[out=0, in=180] (6,2) to (6.5,2) to[out=0, in=180] (7,3) to (10,3);

        \draw[color=blue] (9.75,1.5) node {$1$}; 
        \draw[color=blue] (8.75,1.5) node {$2$}; 
        \draw[color=blue] (7.75,1.5) node {$3$}; 
        \draw[color=blue] (6.5,1.5) node {$4$}; 
        \draw[color=blue] (4.5,1.5) node {$5$}; 
        \draw[color=blue] (2.75,1.5) node {$6$}; 
        \draw[color=blue] (1,1.5) node {$8$}; 
        \draw[color=blue] (-0.75,1.5) node {$11$}; 
        \draw[color=blue] (-1.75,1.5) node {$12$}; 
        
        \draw[color=blue] (8.6,2.5) node {$3|4=2|3=1|2$}; 
        \draw[color=blue] (6,2.5) node {$4|5$}; 
        \draw[color=blue] (4,2.5) node {$5|6$}; 
        \draw[color=blue] (1.75,2.5) node {$6|8$}; 
        \draw[color=blue] (0,2.5) node {$8|11$}; 
        \draw[color=blue] (-1.5,2.5) node {$11|12$}; 
        
        \draw[color=blue] (8.2,3.5) node {$5|6|7=4|5|6=3|4|5$}; 
        \draw[color=blue] (3.2,3.5) node {$5|6|8$}; 
        \draw[color=blue] (1,3.5) node {$6|8|11$}; 
        \draw[color=blue] (-1.3,3.5) node {$8|11|12$}; 

        \draw[color=blue] (7.65,4.5) node {$5|6|8|10=5|6|8|9=5|6|7|8$}; 
        \draw[color=blue] (2.75,4.5) node {$5|6|8|11$}; 
        \draw[color=blue] (-1.15,4.5) node {$6|8|11|12$}; 
    \end{tikzpicture}
    \caption{Braid $\longest\beta_{\skewschubert}$ for $\lambda=(7,7,5,3,1)$,   and $\mu=(3,3,2)$ with subspaces $V(a,i)$ and $W_i,\Wop_i$ filled in.}
    \label{fig: skew braid diagram}
\end{figure}
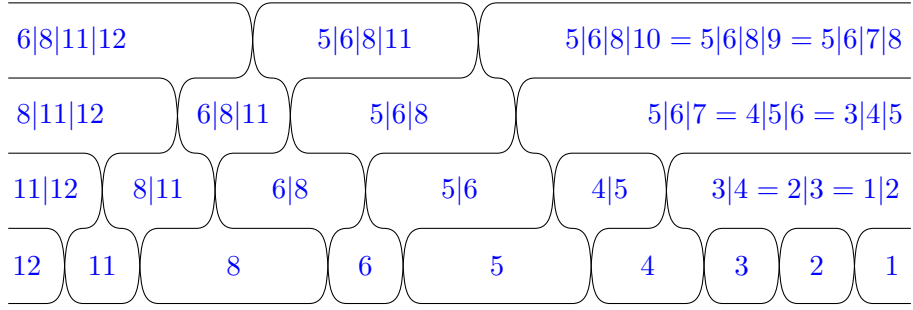

\noindent {\bf Step 2:} Next, we need to construct the inverse map
$$
\Xi: X\left(\longest\beta_{\skewschubert}\right)\times \left(\C^\times\right)^s\to S^{\circ}_{\skewschubert}.
$$
We use Lemma \ref{lem: flags}(c) to think of a point $\Omega$ in $X\left(\longest\br_{\skewschubert}\right)$ as the data of a configuration of flags $\CF^1 \stackrel{\br}{\dashrightarrow}\CF^2$ together with a choice of basis for $\CF^1$. We denote this basis by $v_{b_1},\ldots,v_{b_k}$, and the subspaces in regions of $\beta_{\skewschubert}$ by $L(a,i)$.
Recall that for all $j\in \{1,\ldots,n\}$ either $j\in \{b_1,\ldots,b_k\}$ or $j=a+\mubar_a$ for some $a\in \{1,\ldots,n-k\}$.
The vectors $v_{b_1},\ldots,v_{b_k}$ are known, and we reconstruct the vectors $v_{a+\mubar_a}$. If $\mubar_a=\lambdabar_a$ then we set $v_{a+\mubar_a}=0$. From now on we assume $\mubar_a<\lambdabar_a$.
We claim that the intersection
$$
L=L(a,\mubar_a+1)\cap W_{k-\mubar_a}
$$
is one-dimensional. Indeed, 
$$
\dim L(a,\mubar_a+1)+\dim W_{k-\mubar_a}=(\mubar_a+1)+(k-\mubar_a)=k+1,\ \dim\left(L(a,\mubar_a+1)+W_{k-\mubar_a}\right)\le k
$$
so $\dim L\ge 1$. On the other hand, $\Wop_{\mubar_a}\subset L(a,\mubar_a+1)$ and $\Wop_{\mubar_a}\cap W_{k-\mubar_a}=0$, hence $\Wop_{\mubar_a}\cap L=0$. Since $\dim (\Wop_{\mubar_a}+L)\le \dim L(a,\mubar_a+1)=\mubar_a+1$, we get $\dim L\le 1$. Therefore $\dim L=1$, and we can define $v_{a+\mubar_a}$ up to a scalar  as a basis vector in $L$. 
To fix the scalar, observe that $\sq_{a,\lambdabar_a}\in R^1(\skewschubert)$  and 
$$
I'(a,\lambdabar_a)=\left\{b_1,\ldots,b_{\mubar_a},
a+\mubar_a,a+\mubar_a+1,\ldots,a+\lambdabar_a-1,b_{\lambdabar_a+1},\ldots,b_k\right\}.
$$
Therefore $\Delta_{I'(a,\lambdabar_a)}
$
involves $v_{a+\mubar_a}$ and other vectors which we have reconstructed earlier, so fixing the value of $\Delta_{I'(a,\lambdabar_a)}
$ determines the scalar for $v_{a+\mubar_a}$. 

Altogether, the vectors $v_i$ define a matrix 
\begin{equation}
\label{eq: def xi}
V=\Xi\left(\Omega,\Delta_{I'(a,\lambdabar_a)}, a=1,\ldots,n-k\right).
\end{equation}
By construction, $v_{b_1},\ldots,v_{b_k}$ are linearly independent, so $\rank(V)=k$ and we can interpret $V$ as a point of $\Gr(k,n)$. 
At this stage, we do not know if it belongs to $S^{\circ}_{\skewschubert}$, however, by Lemma \ref{lem: 1d intersection new} we get
\begin{equation}
\label{eq: left inverse}
\Xi\left(\widetilde{\Omega}(V)\right)=V\ \mathrm{for}\ V\in S^{\circ}_{\skewschubert}.
\end{equation}

\noindent {\bf Step 3:} Given $V$ as in \eqref{eq: def xi}, we can define the subspaces $V(a,i)=\spann(v_{\alpha}\mid \alpha\in J(a,i))$. We  prove that $L(a,i)=V(a,i)$ by induction in $i$. For the base case, we can choose $i=\mubar_a+1$ where $L(a,\mubar_a+1)=\Wop_{\mubar_a}+\spann(v_{a+\mubar_a})=V(a,\mubar_a+1)$ holds by construction. For the step of induction, assume that $\sq_{a,i-1}$ and $\sq_{a+1,i-1}$ are in $\skewschubert$. From the braid diagram, we read $L(a,i-1)\subset L(a,i)$, $L(a+1,i-1)\subset L(a,i)$ and $L(a,i-1)\neq L(a+1,i-1)$, therefore $L(a,i)=L(a,i-1)+L(a+1,i-1)$.
By the assumption of induction we have $L(a,i-1)=V(a,i-1)$ and $L(a+1,i-1)=V(a+1,i-1)$ and $$L(a,i)=V(a,i-1)+V(a+1,i-1)=V(a,i).$$ The case when $\sq_{a+1,i-1}\in \mu$ is treated similarly. Note that the same argument shows that $\dim V(a,i)=i$. 

\noindent {\bf Step 4:} By \eqref{eq: left inverse} we have $S^{\circ}_{\skewschubert}\subset \mathrm{Im}\ \Xi$ and the corresponding inclusion of (Zariski) closures
$\overline{S^{\circ}_{\skewschubert}}\subset \overline{\mathrm{Im}\ \Xi}$. On the other hand, observe that $$\dim X\left(\longest\beta_{\skewschubert}\right)=\ell(\beta_{\skewschubert})=|\lambda|-|\mu|-s,\ \dim S^{\circ}_{\skewschubert}=|\lambda|-|\mu|
$$
and both varieties are irreducible.
Therefore   $\overline{\mathrm{Im}\ \Xi}$  is an irreducible closed subset of the Grassmannian $\Gr(k,n)$ of dimension at most $|\lambda|-|\mu|$, hence  
$\overline{S^{\circ}_{\skewschubert}}=\overline{\mathrm{Im}\ \Xi}.$

\noindent{\bf Step 5:} 
Recall that by Definition \ref{def: Grassmann necklace for V} and Theorem \ref{thm: skew as positroid} the skew shaped positroid $S^{\circ}_{\skewschubert}$ is characterized by the non-vanishing of the minors $\Delta_{I'(a,i)}$ for $\sq_{a,i}\in R(\skewschubert)$, and vanishing of minors $\Delta_J$ with $J>I'(a,i)$ in the cyclic order $\leq_{a+i-1}$. The equations $\Delta_J(V)=0$ hold on the closure $\overline{S^{\circ}_{\skewschubert}}$ and therefore for $V$ as in \eqref{eq: def xi} we get $\Delta_J(V)=0$. 

By Step 3, we have $V(d_i,i)=L(d_i,i)$ and by definition of $X\left(\longest\beta_{\skewschubert}\right)$ we have $L(d_i,i)\cap W_{k-i}=0$. Therefore $V(d_i,i)\cap W_{k-i}=0$ and $\Delta_{I'(d_i,i)}(V)\neq 0$. This implies $\Delta_{I'(a,i)}(V)\neq 0$ for $\sq_{a,i}\in R(\skewschubert)$, and  $\mathrm{Im}\ \Xi=S^{\circ}_{\skewschubert}$.

Finally, by applying Step 3 again we conclude that $\widetilde{\Omega}$ and $\Xi$ are inverse to each other.

\end{proof}

\begin{example}
Let us describe the map $\Xi$ in our running example. Given the vectors $v_{b_1},\ldots,v_{b_k}$ and the subspaces $L(a,i)$ as in Figure \ref{fig: skew braid diagram}, we know all vectors $v_i$ right away except for $v_7,v_9$ and $v_{10}$. We also know that 
$$
L(5,3)=\langle v_5,v_6,v_7\rangle,\ L(6,4)=\langle v_5,v_6,v_8,v_9\rangle= 
L(7,4)=\langle v_5,v_6,v_8,v_{10}\rangle.
$$
By Lemma \ref{lem: 1d intersection new}, we can obtain the remaining vectors as intersections
$$
v_7\in L(5,3)\cap W_{3}=\langle v_5,v_6,v_7\rangle\cap \langle v_8,v_{11},v_{12}\rangle,
$$
$$
v_9\in L(6,4)\cap W_{2}=\langle v_5,v_6,v_8,v_9\rangle\cap \langle v_{11},v_{12}\rangle,
$$
$$
v_{10}\in L(7,4)\cap W_{2}=\langle v_5,v_6,v_8,v_{10}\rangle\cap \langle v_{11},v_{12}\rangle.
$$
In fact, the vectors $v_{9}$ and $v_{10}$ are proportional since $L(6,4)=L(7,4)$. To pin down the scaling factors, we first use the minor $\Delta_{5,6,8,10,11}$ to rescale $v_{10}$ (since the vectors $v_5,v_6,v_{8},v_{11}$ are known), then
$\Delta_{5,6,8,9,12}$ for $v_9$ and then
$\Delta_{5,6,7,8,12}$ for $v_7$.
\end{example}

\begin{example}
Alternatively, we can use Lemma \ref{lem: one dimensional} to reconstruct $v_7,v_9$ and $v_{10}$, but this turns out to be a more subtle and recursive process.

Let's start with $v_{10}$. Since $f(10)=12$ by \eqref{eq: baf example horizontal} we have $v_{10}\in \langle v_{11},v_{12}\rangle$. On the other hand, clearly $v_{10}\in \langle v_5,v_6,v_8,v_{10}\rangle$. Observe that $\Delta_{I_{\mu}}=\Delta_{5,6,8,11,12}\neq 0$, so  $\langle v_{11},v_{12}\rangle+\langle v_5,v_6,v_8,v_{10}\rangle=\C^5$.
Therefore the intersection
$\langle v_{11},v_{12}\rangle\cap \langle v_5,v_6,v_8,v_{10}\rangle$  is 1-dimensional and hence spanned by $v_{10}$.

Next, we determine $v_9$. Since $f(9)=10$, we get $v_9\in \langle v_{10}\rangle$.  

Finally, we determine $v_7$. Since $f(7)=9$, we get $v_7\in \langle v_{8},v_{9}\rangle$ while clearly $v_7\in \langle v_5,v_6,v_7\rangle$. Note that 
$\dim \langle v_5,v_6,v_7,v_8\rangle =4,$
so 
$$
\dim \left[\langle v_{8},v_{9}\rangle+\langle v_5,v_6,v_7\rangle\right]\ge 4,\ \dim \left[\langle v_{8},v_{9}\rangle\cap \langle v_5,v_6,v_7\rangle\right]\le 2+3-4=1.
$$
Therefore the intersection is indeed one-dimensional and spanned by $v_7$.
\end{example}

\begin{remark}
\label{rmk: I lambda}
The flag $\CF^0$ on the right boundary of $\Omega(V)$ constructed in Theorem \ref{thm: from skew schubert to braid variety} has the following concise description. Let 
$$
I_{\lambda}=\{d_i+i-1\mid i=1,\ldots,k\}
$$
be the set of labels of the vertical steps in the $\lambda$-labeling. Then the $j$-dimensional subspace in the flag $\CF^0$ is given by
\begin{equation}
\label{eq: flag on the right}
\CF^0_j=V(d_j,j)=\spann\{v_{d_i+i-1}\mid i=1,\ldots,j\}.
\end{equation}
This follows by induction from $V(d_j,j)=\langle v_{d_j+j-1}\rangle \oplus V(d_j,j-1)$ and Lemma \ref{lem: equality V}. In particular,
$v_{d_i+i-1}$ form a basis of $\C^k$, and 
$\Delta_{I_{\lambda}}\neq 0$ on $S_{\skewschubert}^{\circ}$, in agreement with Remark \ref{rmk: grassmannian richardson}.
In our running example, $I_{\lambda}=\{1,2,5,8,11\}$ and we get
$$
\CF^0=\left\{\langle v_1\rangle\subset \langle v_1,v_2\rangle\subset \langle v_1,v_2,v_5\rangle\subset  \langle v_1,v_2,v_5,v_8\rangle\subset  \langle v_1,v_2,v_5,v_8,v_{11}\rangle=\C^5\right\},
$$
compare with Figure \ref{fig: skew braid diagram}.
\end{remark}

\subsection{Skew shaped positroid as a braid variety on $n$ strands}
\label{sec: n strands}

We have realized the skew shaped positroid variety $S^\circ_{\skewschubert}$ as a braid variety on $k$ strands. It can also be realized as a braid variety on $n$ strands, as follows. By definition, we have that the variety
$S^{\circ}_{\skewschubert}$ is the open Richardson variety $R^{\circ}_{w_{\lambda}, w_{\mu}}$ inside the flag variety $\Fl(n)$. Let $w_0^{(n)}$ denote the longest element in $S_n$. As shown in \cite[Theorem 3.14]{CGGLSS22}, we have an isomorphism

\begin{equation}\label{eq:iso-richardson}
\begin{array}{c}
X\left(\lift{w_\mu}\lift{w_{\lambda}^{-1}w_0^{(n)}}\right) \to R^{\circ}_{w_{\lambda}, w_{\mu}} \\
\left(\CF^{\std} \stackrel{\lift{w_\mu}}{\dashrightarrow} \CF^r \stackrel{\lift{w_{\lambda}^{-1}w_0^{(n)}}}{\dashrightarrow} \CF^{\ant}\right) \mapsto \CF^r.
\end{array}
\end{equation}
where $\lift{w_{\mu}}$ is a minimal braid lift of $w_{\mu}$, and similarly for $\lift{w_{\lambda}^{-1}w_0^{(n)}}$. Let us elaborate on this map, in particular, obtaining a $k \times n$ matrix $V$ starting from the flag $\CF^r$. The flag $\CF^r$ belongs to the Schubert cell $\borel(n) w_{\mu}\borel(n)/\borel(n)$. Thus, we can find a matrix $M \in \GL(n)$ such that $\CF^r = \CF(M)$ and $M = (m_{i,j})$ satisfies the following conditions:
\begin{itemize}
\item $m_{i, w_{\mu}(i)} = 1$ for every $i = 1, \dots, n$. These are the \emph{pivots} of $M$.
\item There are $0$'s below and to the right of every pivot of $M$.
\end{itemize}

Note that, since the permutation $w_{\mu}$ is $k$-Grassmannian, every non-pivot entry of the last $(n-k)$-columns of $M$ is zero. In other words, $M = (m_1|m_2|\cdots|m_k|e_{w_{\mu}(k+1)}|\cdots|e_{w_{\mu}(n)})$. 
The $k \times n$-matrix associated to $\CF^r$ is $(m_1|\cdots|m_k)^{T}$.  

Now, by Lemma \ref{lem:additive-factorization} we can find a decomposition $\lift{w_{\mu}} = \lift{w_{\skewschubert}}\lift{w_\lambda}$, so that 
\begin{equation}\label{eq:decomposition-with-w0}
\lift{w_\mu}\lift{w_{\lambda}^{-1}w_0^{(n)}} = \lift{w_{\skewschubert}}\cdot \lift{w_0^{(n)}}.
\end{equation}

We then have the diagram
\begin{center}
\begin{tikzcd}[row sep=tiny]
R^{\circ}_{e, w_{\skewschubert}} & & \arrow{rr} \arrow{ll} X\left(\beta(w_{\skewschubert})\lift{w_{\lambda}}\lift{w_{\lambda}^{-1}w_0^{(n)}}\right) & & R^{\circ}_{w_{\lambda}, w_{\mu}} = S^{\circ}_{\skewschubert} \\
\CF^{r_1} & & \left(\CF^{\std} \stackrel{\beta(w_{\skewschubert})}{\dashrightarrow} \CF^{r_1} \stackrel{\lift{w_{\lambda}}}{\dashrightarrow} \CF^{r_2} \stackrel{\lift{w_{\lambda}^{-1}w_0^{(n)}}}{\dashrightarrow}
\CF^{\ant}\right) \arrow[mapsto]{rr} \arrow[mapsto]{ll} & & \CF^{r_2}.
\end{tikzcd}
\end{center}
where both maps are isomorphisms. 

\begin{remark}
Note that $\lift{w_{\skewschubert}}$ is an $n$-strand braid while $\beta_{\skewschubert}$ is a $k$-strand braid, so these should not be confused.
\end{remark}

\section{Cluster structure on skew shaped positroids}\label{sec: skew cluster}

In this section, we provide an initial seed construction for the skew shaped positroid following Galashin-Lam \cite{GL19} and using the lattice path labelings as in Subsection \ref{subsec:lattice path labeling for skew}. We first state the result in Section \ref{subsec: cluster result} and then give a proof in Section \ref{subsec: cluster proof}.
The reader is invited to skip Section \ref{subsec: cluster proof} at the first reading.

\subsection{Background on cluster algebras}
\label{subsec: cluster} 

Let us now recall the definition of a cluster algebra \cite{FZcluster}. We restrict ourselves to the skew-symmetric case, as this is all that will be needed.

An \emph{ice quiver} is a quiver (i.e., a directed graph, where we allow multiple arrows between vertices) with finite vertex set $Q_0$ that has no loops nor directed $2$-cycles, and where a special subset $Q_0^{f}$ of the vertices of $Q$ is declared to be \emph{frozen}. An element in $Q_0 \setminus Q_0^{f}$ is said to be a \emph{mutable vertex} of $Q$.

Now let $Q$ be an ice quiver, with $|Q_0| = n+m$, in such a way that there are $n$ mutable vertices and $m$ frozen ones. We consider a field $\mathcal{F}$ of transcendence degree $n+m$ over $\C$. A \emph{seed} $\Sigma = (Q, \mathbf{x})$ consists of:
\begin{enumerate}
    \item The ice quiver $Q$, and
    \item A set $\mathbf{x} = \{x_i \mid i \in Q_0\}$ that is a transcendental basis for $\mathcal{F}$, i.e. $\mathcal{F} = \C(x_i \mid i \in Q_0)$. The set $\mathbf{x}$ is known as the set of \emph{cluster variables} of $\Sigma$.
\end{enumerate}

Given a seed $\Sigma = (Q, \mathbf{x})$ and a mutable vertex $k \in Q_0$, the \emph{mutation} of $\Sigma$ in the direction $k$ is the seed $\mu_k(\Sigma) = (\mu_k(Q), \mu_k(\mathbf{x}))$ where:
\begin{enumerate}
    \item $\mu_k(\mathbf{x}) = (\mathbf{x}\setminus\{x_k\})\cup\{x'_k\}$, where $x'_k \in \mathcal{F}$ is defined via the equation
    \[
    x_kx_k' = \prod_{i \to k}x_i + \prod_{k \to j}x_j.
    \]
    \item $\mu_k(Q)$ has the same vertex set as $Q$, but the arrows change via the following three-step procedure:
    \begin{enumerate}
        \item Reverse all arrows incident with $k$.
        \item For any pair of arrows $i \to k \to j$ in $Q$ insert a new arrow $i \to j$, unless both $i$ and $j$ are frozen.
        \item The previous two steps may have created $2$-cycles. Remove all the arrows in a maximal collection of disjoint $2$-cycles. 
    \end{enumerate}
\end{enumerate}
Two seeds $\Sigma$ and $\Sigma'$ are said to be \emph{mutation equivalent} if there is a finite sequence of mutations taking one seed to the other. Since mutation at a fixed vertex $k$ is an involutive operation, this is well-defined. 

\begin{definition}
    Let $\Sigma$ be a seed. The \emph{cluster algebra} $\mathcal{A}(\Sigma)$ is the subalgebra of the field $\mathcal{F}$ generated by the sets of cluster variables in all seeds mutation equivalent to $\Sigma$, as well as by $x_{k}^{-1}$ for $k \in Q_0^{f}$.
    We say that a commutative algebra $\mathcal{A}$ \emph{admits a cluster structure} if there exists a seed $\Sigma$ such that $\mathcal{A} \cong \mathcal{A}(\Sigma)$. Similarly, we say that an affine algebraic variety $X$ admits a cluster structure if the coordinate algebra $\C[X]$ admits a cluster structure.  
\end{definition}

Crucially for this paper, a commutative algebra $\mathcal{A}$ may admit more than one cluster structure, that is, there may exist two seeds $\Sigma, \Sigma'$ which are not mutation equivalent but such that $\mathcal{A}(\Sigma) \cong \mathcal{A} \cong \mathcal{A}(\Sigma')$.

Let $\mathcal{A}(\Sigma)$ and $\mathcal{A}(\Sigma')$ be the cluster algebras associated to the seeds $\Sigma$ and $\Sigma'$, respectively.  Following Fraser \cite{Fraser}, see also \cite[Section 5.2]{LamSpeyerI}, we define a class of particularly well-behaved morphisms between cluster algebras associated to seeds that are not necessarily mutation-equivalent.

For a seed $\seed$ and a mutable vertex $i$, we define the exchange ratio $\widehat{y}_i$ as the ratio $$\widehat{y}_i=\dfrac{\prod_{j\rightarrow i}x_j^{\#\{j\rightarrow i\}}}{\prod_{i\rightarrow j}x_j^{\#\{i\rightarrow j\}}}.$$ 

\begin{definition}\cite{Fraser,FSB22}
Let $\mathcal{A}(\seed)$ and $\mathcal{A}(\seed')$ be cluster algebras of rank $n+m$, both with $m$ frozen variables. Let $\mathbf{x} = \{x_1, \dots, x_{n+m}\}$ be the cluster variables of $\seed$, and $\mathbf{x} = \{x'_1, \dots x'_{n+m}\}$ be the cluster variables of $\seed'$. An algebra isomorphism $f: \mathcal{A}(\seed) \to \mathcal{A}(\seed')$ is said to be a \emph{quasi-cluster isomorphism} if the following conditions are satisfied:
\begin{enumerate}
\item For each frozen variable $x_j \in \mathbf{x}$, $f(x_j)$ is a Laurent monomial in the frozen variables of $\mathbf{x}'$.
\item For each mutable variable $x_i \in \mathbf{x}$, $f(x_i)$ coincides with $x'_i$, up to multiplication by a Laurent monomial in the frozen variables of $\mathbf{x}'$.
\item The exchange ratios are preserved, i.e., for each mutable variable $x_i$ of $\seed$, $f(\widehat{y}_i) = \widehat{y}'_i$.
\end{enumerate}
\end{definition}

We remark that properties (1)--(3) are stable under mutations, i.e., if (1)--(3) holds for two seeds $\seed, \seed'$, then it also holds for $\mu_i(\seed), \mu_i(\seed')$ for every mutable vertex $i$, see \cite[Section 5.2]{LamSpeyerI}.    
\subsection{Cluster structure on skew shaped positroid: statement}
\label{subsec: cluster result}

The following statement is a consequence of the main result in \cite{GL19} which we review in Section \ref{subsec: cluster proof}.

\begin{theorem} \label{thm: skew schubert cluster}
Consider a skew diagram $\skewschubert$. The skew shaped positroid $S^{\circ}_{\skewschubert}$ admits a cluster structure, and an initial seed can be described as follows:

{\bf (a) Cluster variables:} For each box $\sq_{a,i} \in \skewschubert$, we have a cluster variable $x_{a,i} := \minor_{I'(a,i)}$.
In particular, a variable $x_{a,i}$ is frozen if the box $\sq_{a,i}$ belongs to the boundary ribbon $R(\skewschubert)$. Otherwise, a variable $x_{a,i}$ is mutable. 

{\bf (b) Quiver:} The vertices of the quiver $Q_{\skewschubert}$ are in bijection with the boxes $\sq_{a,i} \in \skewschubert$. There are three types of arrows:
\begin{center}
$\sq_{a, i} \to\sq_{a+1, i}$, \qquad \qquad $\sq_{a,i} \to \sq_{a, i-1}$. \qquad \qquad $\sq_{a,i} \to \sq_{a-1, i+1}$
\end{center} 
provided at least one of these boxes corresponds to a mutable vertex.
\end{theorem}

\begin{remark}\label{rmk:Imu2}
In part of the literature, there is an extra frozen variable $x_{\mu} = \minor_{I_{\mu}}$, and consequently the quiver $Q_{\skewschubert}$ has an extra frozen vertex. This yields a cluster structure on the affine cone over $\skewschvar$. We will work inside $\Gr(k,n)$ (as opposed to its affine cone) by setting $\minor_{I_{\mu}} = 1$, so we do not see this frozen variable, see Remark \ref{rmk:Imu1}.
\end{remark}

Recall the variety $S^{\circ, 1}_{\skewschubert}$ is given by the conditions that $\Delta_{I'(a,i)} = 1$ for every box $\sq_{a,i} \in R^{1}(\skewschubert)$. Note that these are all frozen variables for the cluster structure in $S^{\circ}_{\skewschubert}$ and we obtain the following result.

\begin{corollary}\label{cor: skew schubert cluster}
    The skew shaped positroid $S^{\circ,1}_{\skewschubert}$ admits a cluster structure, obtained from the cluster structure on $S^{\circ}_{\skewschubert}$ from Theorem \ref{thm: skew schubert cluster} by specializing the frozen variables $x_{a,i} = 1$, where $\sq_{a,i}\in R^1(\skewschubert)$. 
\end{corollary}

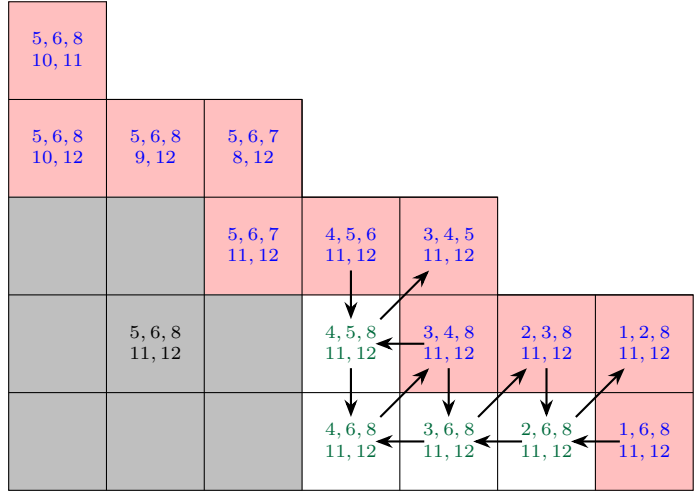
\begin{figure}
    \centering

    \begin{tikzpicture}[scale=0.65]
    \filldraw[color=pink, scale=2,fill opacity=0.5] (7,0) -- (7,2) -- (5,2) -- (5,3) -- (3,3) -- (3,4) -- (1,4) -- (1,5) -- (0,5) -- (0,3) -- (2,3) -- (2,2) -- (4,2) -- (4,1) -- (6, 1) -- (6,0) -- cycle;
        \filldraw[color=lightgray, scale=2] (0,0) -- (0,3) -- (2,3) -- (2,2) -- (3,2) -- (3,0) -- cycle;
\draw[scale=2] (0,0) -- (0,5) -- (1,5) -- (1,4) -- (3,4) -- (3,3) -- (5,3) -- (5,2) -- (7,2) -- (7,0)-- cycle;
\draw[scale=2] (0,1) to (7,1);
\draw[scale=2] (0,2) to (7,2);
\draw[scale=2] (0,3) to (5,3);
\draw[scale=2] (0,4) to (3,4);
\draw[scale=2] (1,0) to (1,5);
\draw[scale=2] (2,0) to (2,4);
\draw[scale=2] (3,0) to (3,4);
\draw[scale=2] (4,0) to (4,3);
\draw[scale=2] (5,0) to (5,2);
\draw[scale=2] (6,0) to (6,2);

\node at (3,3) {\tiny $\begin{matrix} 5, 6, 8 \\ 11, 12 \end{matrix}$};


\node at (13,1) {\tiny $\textcolor{blue}{\begin{matrix} 1, 6, 8 \\ 11, 12 \end{matrix}}$};
\node at (13,3) {\tiny $\textcolor{blue}{\begin{matrix} 1, 2, 8 \\ 11, 12 \end{matrix}}$};
\node at (11,3) {\tiny $\textcolor{blue}{\begin{matrix} 2, 3, 8 \\ 11, 12 \end{matrix}}$};
\node at (9,3) {\tiny $\textcolor{blue}{\begin{matrix} 3, 4, 8 \\ 11, 12 \end{matrix}}$};
\node at (9,5) {\tiny $\textcolor{blue}{\begin{matrix} 3, 4, 5 \\ 11, 12 \end{matrix}}$};
\node at (7,5) {\tiny $\textcolor{blue}{\begin{matrix} 4, 5, 6 \\ 11, 12 \end{matrix}}$};
\node at (5,5) {\tiny $\textcolor{blue}{\begin{matrix} 5, 6, 7 \\ 11, 12 \end{matrix}}$};
\node at (5,7) {\tiny $\textcolor{blue}{\begin{matrix} 5, 6, 7 \\ 8, 12 \end{matrix}}$};
\node at (3,7) {\tiny $\textcolor{blue}{\begin{matrix} 5, 6, 8 \\ 9, 12 \end{matrix}}$};
\node at (1,7) {\tiny $\textcolor{blue}{\begin{matrix} 5, 6, 8 \\ 10, 12 \end{matrix}}$};
\node at (1,9) {\tiny $\textcolor{blue}{\begin{matrix} 5, 6, 8 \\ 10, 11 \end{matrix}}$};


\node at (11,1) {\tiny $\textcolor{cadmiumgreen}{\begin{matrix} 2, 6, 8 \\ 11, 12 \end{matrix}}$};
\node at (9,1) {\tiny $\textcolor{cadmiumgreen}{\begin{matrix} 3, 6, 8 \\ 11, 12 \end{matrix}}$};
\node at (7,1) {\tiny $\textcolor{cadmiumgreen}{\begin{matrix} 4, 6, 8 \\ 11, 12 \end{matrix}}$};
\node at (7,3) {\tiny $\textcolor{cadmiumgreen}{\begin{matrix} 4, 5, 8 \\ 11, 12 \end{matrix}}$};

\draw[thick, -Stealth] (7, 2.5) to (7, 1.5);
\draw[thick, -Stealth] (7, 4.5) to (7, 3.5);

\draw[thick, -Stealth] (9, 2.5) to (9, 1.5);
\draw[thick, -Stealth] (11, 2.5) to (11, 1.5);

\draw[thick, -Stealth] (12.5, 1) to (11.5, 1);
\draw[thick, -Stealth] (10.5, 1) to (9.5, 1);






\draw[thick, -Stealth] (8.5, 1) to (7.5, 1);

\draw[thick, -Stealth] (8.5, 3) to (7.5, 3);

\draw[thick, -Stealth] (7.6, 1.5) to (8.6, 2.5);

\draw[thick, -Stealth] (7.6, 3.5) to (8.6, 4.5);

\draw[thick, -Stealth] (9.6, 1.5) to (10.6, 2.5);

\draw[thick, -Stealth] (11.6, 1.5) to (12.6, 2.5);

    \end{tikzpicture}
    \caption{The quiver for $S^\circ_{\skewschubert}$ for our running example where $\lambda=(7,7,5,3,1),\,\mu=(3,3,2)$ is shown superimposed on the skew diagram. The cluster variables are defined $x_{a,i}=\Delta_{I'(a,i)}$. Here we label the cluster variables only by the indices $I'(a,i)$. \textcolor{blue}{Blue} indices correspond to frozen variables that arise from the boundary ribbon $R(\skewschubert)$, while  \textcolor{cadmiumgreen}{green} indices indicate the mutable vertices of the cluster seed.}
    \label{fig: running quiver}
\end{figure}



\begin{remark}
We have $\Delta_{I_{\lambda}}\neq 0$ on $S_{\skewschubert}^{\circ}$, see Remark \ref{rmk: I lambda}. By \cite[Theorem 1.3(i)]{GLS} any invertible function on a cluster variety is equal to a monomial in frozen variables, and one may wonder how to write $\Delta_{I_{\lambda}}$ as such monomial explicitly.

To do this, we can consider the action of the torus $(\C^{\times})^n$ on $\Gr(k,n)$ which scales the column vectors $v_i$, and compare the weights of frozen variables and of $\Delta_{I_{\lambda}}$. In our running example, we get
$$
\Delta_{I_{\lambda}}=\Delta_{1,2,5,8,11}=\frac{\Delta_{1,2,8,11,12}\Delta_{3,4,5,11,12}\Delta_{5,6,7,8,12}\Delta_{5,6,8,10,11}}{
\Delta_{3,4,8,11,12}\Delta_{5,6,7,11,12}\Delta_{5,6,8,10,12}}\quad \mathrm{on}\quad S_{\skewschubert}^{\circ}.
$$
Note that the factors in the numerator correspond to the outside corners of the ribbon $R(\skewschubert)$ and the factors in the denominator correspond to the inside corners. We leave details to the reader. 
\end{remark}

\subsection{Cluster structure on skew shaped positroid: proof}
\label{subsec: cluster proof}

Cluster structures on positroid varieties are intimately related with Postnikov's plabic graphs \cite{postnikov} which are planar, bicolored graphs in a disk with $n$ edges that touch the boundary of the disk (we call these boundary edges). 

Given a plabic graph $G$, Postnikov defines in \cite[\S 13]{postnikov} the collection of \newword{trips} as follows. First, we label the boundary edges of $G$ from $1$ to $n$. A trip  starts at a boundary edge $i$ and follows the ``rules of the road" by turning right at each black vertex and left at each white vertex. A plabic graph is called \emph{reduced} if the trips in it satisfy certain technical conditions in \cite[Theorem 13.2]{postnikov}.
If $G$ is reduced, the trip starting at $i$ ends at another boundary edge labeled by $\tau_G(i)$, see Figure \ref{fig:skew plabic and a trip} (left) for examples.  If $G$ is reduced then it is known that $\tau_G$ is a permutation in $S_n$, it is called the \emph{trip permutation} for $G$. 

The connected components of the complement to $G$ will be called regions (also known as faces).
The \newword{source labeling} of regions   by subsets of $\{1,\ldots,n\}$ is defined as follows: for each  trip path starting with a boundary edge $i$, we label all the regions to the right of it by $i$. It is known that whenever $G$ is reduced, each region acquires exactly $k$ labels this way, where $k$ is the number of counter-clockwise trips.

Let $f$ be a bounded affine permutation. 
For the positroid variety $\Pi^{\circ}_{f}\subset \Gr(k,n)$, 
consider the family of all reduced plabic graphs $G$ with trip permutation $f \pmod n$.\footnote{For fixed points $i$ of $f \pmod n$, we also need to carry the information of whether the trip permutation carries $i$ to $i$ clockwise or counterclockwise.}  A plabic graph $G$ in this family determines a quiver $Q_G$ with vertices corresponding to the regions of $G$, and cluster variables given by the  Pl\"ucker coordinates indexed by the source labels of regions defined above. By the main result of \cite{GL19}, all these determine (possibly different) seeds for the cluster structure on $\Pi^{\circ}_f$. We refer to \cite{GL19}  for more details.

To prove Theorem \ref{thm: skew schubert cluster}, we will produce a plabic graph with trip permutation $f \pmod n$ whose corresponding quiver and cluster variables coincide with those stated in Theorem \ref{thm: skew schubert cluster}.


\begin{definition}
\label{def: G skew}
Given $\mu \subseteq \lambda$, we define the plabic graph $G_{\skewschubert}$ as follows:
\begin{enumerate}
\item[(a)] Replace each box in $\skewschubert$ with a box with a black-white edge, as in Figure \ref{fig:black and white bridge} (left) and Figure \ref{fig:skew plabic and a trip} (left).
\item[(b)] If $\sq_{a,i}$ and $\sq_{a-1, i}$ are both in $\skewschubert$, join the white vertex of $\sq_{a,i}$ to the black vertex of $\sq_{a-1, i}$. Similarly, if $\sq_{a,i}$ and $\sq_{a,i-1}$ are both in $\skewschubert$, join the white vertex of $\sq_{a,i}$ to the black vertex of $\sq_{a,i-1}$. See Figure \ref{fig:skew plabic and a trip} (left).
\item[(c)] If a portion of the boundary of a box in $\skewschubert$ is contained in the northeast boundary of $\lambda$, we connect the corresponding vertices of $G_{\skewschubert}$ to the boundary, as in Figure \ref{fig:skew plabic and a trip}. 
\item[(d)] If a portion of the boundary of a box in $\mu$ is contained in the northeast boundary of $\lambda$, we add new vertices and connect them to the boundary, as follows, where the thick sides are those belonging to the northeast boundary of $\lambda$:
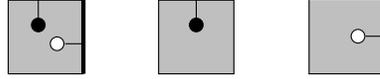
\begin{center}
\begin{tikzpicture}
\filldraw[color=lightgray] (0,0) -- (0,1) -- (1,1) -- (1,0) -- cycle;
\draw[ultra thick] (1,0) to (1,1) to (0,1);
\draw (0,1) to (0,0) to (1,0);
\node [scale =0.5, circle, draw=black, fill=black] (b) at (0.4, 0.65){};
\draw (0.4, 0.73) to (0.4, 1);
\node [scale =0.5, circle, draw=black, fill=white] (b) at (0.65, 0.4){};
\draw (0.7575
, 0.4) to (1, 0.4);

\filldraw[color=lightgray] (2,0) -- (2,1) -- (3,1) -- (3,0) -- cycle;
\draw[ultra thick] (2,1) to (3,1);
\draw (3,1) to (3,0) to (2,0) to (2,1);
\node [scale =0.5, circle, draw=black, fill=black] (b) at (2.5, 0.65){};
\draw (2.5, 0.73) to (2.5, 1);

\filldraw[color=lightgray] (4,0) -- (5,0) -- (5,1) -- (4,1) -- cycle;
\draw[ultra thick] (5,0) to (5,1);
\draw (5,1) to (4,1) to (4,0) to (5,0);
\node [scale =0.5, circle, draw=black, fill=white] (b) at (4.65, 0.5){};
\draw (4.75, 0.5) to (5, 0.5);
\end{tikzpicture}
\end{center}
\end{enumerate}
\end{definition}

We regard $G_{\skewschubert}$ as a subset of $\lambda$, which is homeomorphic to a disk.  Below we will refer to the southwest boundary region $R_{SW}$ which contains $\mu$ (except for vertices and edges in (d) above).
Also note that $G_{\skewschubert}$ has exactly $n$ boundary edges labeled from $1$ to $n$ as in Figure \ref{fig:skew plabic and a trip}, and $R_{SW}$ connects the boundary edges labeled by $1$ and $n$. 

Next, we introduce an equivalent model of $G_{\skewschubert}$ drawn on the lattice.

\begin{definition}\label{def: trip on skew}
Given a skew diagram $\skewschubert$ where $I_\mu=\{b_1,\ldots, b_k\}$, the \newword{lattice source labeling trip} is given by the following: 
\begin{enumerate}
    \item Edge labeling:
    \begin{itemize}
        \item 
        The boundary edges $e_i$ are labeled consecutively from $1$ to $n$, starting from the southeast corner to the northwest corner. \footnote{Note that this gives a \emph{counterclockwise} enumeration of the boundary edges, that is not standard in the plabic graph literature. Had we drawn our Young diagrams in English notation, we would have obtained a clockwise orientation. We choose not to do so to keep in line with the rest of the paper.}
        
    \end{itemize}  
    \item Construction of the lattice source labeling trip $T_i$:\\
    The path $T_i$ moves southwest until it reaches the boundary of $\mu$ at which point the direction of the path reverses to either the north or east.
    \begin{enumerate}
    \item Initial turning rule:
    \begin{itemize}
        \item If the edge $e_i$ is horizontally oriented, then you make a left turn with respect to the orientation in the southern direction.    
        \item If the edge $e_i$ is vertically oriented, then you make a right turn with respect to the orientation in the western direction.  
        \item Continue until you reach the boundary of $\mu$.
    \end{itemize}
    \item Behavior at the boundary of $\mu$:
    \begin{itemize}
        \item Upon reaching an edge that is contained in the boundary of $\mu$, continue to follow the turning rule from step (a), until the next required step would be an edge contained in the interior of $\mu$.
        \item At this point, reflect the direction of the new edge which would be contained in the interior of $\mu$.
        \begin{enumerate}
            \item If the rule would require you to take the  horizontal edge in the western direction, then instead choose the horizontal edge in the eastern direction.
            \item Similarly, if the rule would require you to take the vertical edge in the southern direction, then instead choose the vertical edge in the northern direction.
        \end{enumerate}
    \end{itemize}
    \item Continuation of the path:
    \begin{itemize}
        \item Continue straight until you reach some edge $e_j$ for $j\in[n]\setminus{i}$. 
    \end{itemize}
    \end{enumerate}
    \item Box labeling:
    \begin{itemize}
        \item If the path $T_i$ has a counterclockwise orientation, then label the boxes on the exterior of the path with the integer $i$, denote this set of boxes $X_i$.
        \item If the path $T_i$ has a clockwise orientation, then label the boxes on the interior of the path with the integer $i$, denote this set of boxes $N_i$. 
         
    \end{itemize}
\end{enumerate}
\end{definition}

\begin{remark}\label{rmk: source trip cases}
    If the path $T_i$ satisfies step (b)(i) in Definition \ref{def: trip on skew}, then $i\in I_\mu$. 
\end{remark}

\begin{example}
For our running example, the lattice trips $T_4$ and $T_5$ in $\skewschubert$ along with the lattice source labeling of $4$ and $5$, respectively, are shown on the right in Figure \ref{fig:skew plabic and a trip}. The lattice trip $T_4$ is drawn in \textcolor{babyblueeyes!200}{blue}. Since the path $T_4$ is oriented clockwise, we label the boxes on the interior of the path with \textcolor{babyblueeyes!200}{4} and denote these boxes as $N_4$. Similarly, $T_5$ is drawn in \textcolor{navyblue}{navy blue} with counterclockwise orientation, so we label the boxes on the exterior of the path with \textcolor{navyblue}{5}, these boxes are denoted $X_5$.

\end{example}

\begin{figure}[!hb]
   \begin{minipage}{0.48\textwidth}
     \centering
    \begin{tikzpicture}[scale=.95]
\filldraw[color=lightgray] (0,-2)--(0,-5)--(3,-5)--(3,-3)--(2,-3)--(2,-2)--(0,-2);
\draw[color=black!70] (0,0)--(0,-5)--(7,-5)--(7,-3)--(5,-3)--(5,-2)--(3,-2)--(3,-1)--(1,-1)--(1,0)--(0,0);
\filldraw[color=babyblueeyes!50] (5.3, -3)--(5.25, -3.25)--(4.75, -3.75)--(4.25, -4.25)--(3.75, -4.75)--(3.25, -4.25)--(3.75, -3.75)--(3.25, -3.25)--(3.75, -2.75)--(3.25, -2.25)--(3.3, -2)--(5,-2)--(5,-3)--(5.3, -3);

\node [scale =0.2, circle, draw=black, fill=black] (b12) at (0.3, 0) {};
\draw[red] (0.5,0.25) node {\scriptsize{$12$}};

\node [scale =0.2, circle, draw=black, fill=black] (b11) at (1, -0.7) {};
\draw[red] (1.2,-0.5) node {\scriptsize{$11$}};

\node [scale =0.2, circle, draw=black, fill=black] (b10) at (1.3, -1) {};
\draw[red] (1.5,-0.8) node {\scriptsize{$10$}};

\node [scale =0.2, circle, draw=black, fill=black] (b9) at (2.3, -1) 
{};
\draw[red] (2.5,-0.8) node {\scriptsize{$9$}};

\node [scale =0.2, circle, draw=black, fill=black] (b8) at (3, -1.7) 
{};
\draw[red] (3.2,-1.5) node {\scriptsize{$8$}};

\node [scale =0.2, circle, draw=black, fill=black] (b7) at (3.3, -2) {};
\draw[red] (3.5,-1.8) node {\scriptsize{$7$}};

\node [scale =0.2, circle, draw=black, fill=black] (b6) at (4.3, -2) {};
\draw[red] (4.5,-1.8) node {\scriptsize{$6$}};

\node [scale =0.2, circle, draw=black, fill=black] (b5) at (5, -2.7) {};
\draw[red] (5.2,-2.5) node {\scriptsize{$5$}};

\node [scale =0.2, circle, draw=black, fill=black] (b4) at (5.3, -3) {};
\draw[red] (5.5,-2.8) node {\scriptsize{$4$}};

\node [scale =0.2, circle, draw=black, fill=black] (b3) at (6.3, -3) {};
\draw[red] (6.5,-2.8) node {\scriptsize{$3$}};

\node [scale =0.2, circle, draw=black, fill=black] (b2) at (7, -3.7) {};
\draw[red] (7.2,-3.5) node {\scriptsize{$2$}};

\node [scale =0.2, circle, draw=black, fill=black] (b1) at (7, -4.7) {};
\draw[red] (7.2,-4.5) node {\scriptsize{$1$}};

\node [scale =0.5, circle, draw=black, fill=black] (7b5) at (0.25, -0.25) {};
\node [scale =0.5, circle, draw=black, fill=white] (7w5) at (0.75, -0.75) {};

\node [scale =0.5, circle, draw=black, fill=black] (7b4) at (0.25, -1.25) {};
\node [scale =0.5, circle, draw=black, fill=white] (7w4) at (0.75, -1.75) {};

\node [scale =0.5, circle, draw=black, fill=black] (6b4) at (1.25, -1.25) {};
\node [scale =0.5, circle, draw=black, fill=white] (6w4) at (1.75, -1.75) {};

\node [scale =0.5, circle, draw=black, fill=black] (5b4) at (2.25, -1.25) {};
\node [scale =0.5, circle, draw=black, fill=white] (5w4) at (2.75, -1.75) {};

\node [scale =0.5, circle, draw=black, fill=black] (5b3) at (2.25, -2.25) {};
\node [scale =0.5, circle, draw=black, fill=white] (5w3) at (2.75, -2.75) {};

\node [scale =0.5, circle, draw=black, fill=black] (4b3) at (3.25, -2.25) {};
\node [scale =0.5, circle, draw=black, fill=white] (4w3) at (3.75, -2.75) {};

\node [scale =0.5, circle, draw=black, fill=black] (4b2) at (3.25, -3.25) {};
\node [scale =0.5, circle, draw=black, fill=white] (4w2) at (3.75, -3.75) {};

\node [scale =0.5, circle, draw=black, fill=black] (4b1) at (3.25, -4.25) {};
\node [scale =0.5, circle, draw=black, fill=white] (4w1) at (3.75, -4.75) {};

\node [scale =0.5, circle, draw=black, fill=black] (3b3) at (4.25, -2.25) {};
\node [scale =0.5, circle, draw=black, fill=white] (3w3) at (4.75, -2.75) {};

\node [scale =0.5, circle, draw=black, fill=black] (3b2) at (4.25, -3.25) {};
\node [scale =0.5, circle, draw=black, fill=white] (3w2) at (4.75, -3.75) {};

\node [scale =0.5, circle, draw=black, fill=black] (3b1) at (4.25, -4.25) {};
\node [scale =0.5, circle, draw=black, fill=white] (3w1) at (4.75, -4.75) {};

\node [scale =0.5, circle, draw=black, fill=black] (2b2) at (5.25, -3.25) {};
\node [scale =0.5, circle, draw=black, fill=white] (2w2) at (5.75, -3.75) {};

\node [scale =0.5, circle, draw=black, fill=black] (2b1) at (5.25, -4.25) {};
\node [scale =0.5, circle, draw=black, fill=white] (2w1) at (5.75, -4.75) {};

\node [scale =0.5, circle, draw=black, fill=black] (1b2) at (6.25, -3.25) {};
\node [scale =0.5, circle, draw=black, fill=white] (1w2) at (6.75, -3.75) {};

\node [scale =0.5, circle, draw=black, fill=black] (1b1) at (6.25, -4.25) {};
\node [scale =0.5, circle, draw=black, fill=white] (1w1) at (6.75, -4.75) {};

\draw[thick] (7b5)--(7w5);
\draw[thick] (7b4)--(7w4);
\draw[thick] (6b4)--(6w4);
\draw[thick] (5b4)--(5w4);
\draw[thick] (5b3)--(5w3);
\draw[thick] (4b3)--(4w3);
\draw[thick] (4b2)--(4w2);
\draw[thick] (4b1)--(4w1);
\draw[thick] (3b3)--(3w3);
\draw[thick] (3b2)--(3w2);
\draw[thick] (3b1)--(3w1);
\draw[thick] (2b2)--(2w2);
\draw[thick] (2b1)--(2w1);
\draw[thick] (1b2)--(1w2);
\draw[thick] (1b1)--(1w1);

\draw[thick] (7w5)--(7b4);
\draw[thick] (5w4)--(5b3);
\draw[thick] (4w3)--(4b2);
\draw[thick] (4w2)--(4b1);
\draw[thick] (3w3)--(3b2);
\draw[thick] (3w2)--(3b1);
\draw[thick] (2w2)--(2b1);
\draw[thick] (1w2)--(1b1);

\draw[thick] (7w4)--(6b4);
\draw[thick] (6w4)--(5b4);
\draw[thick] (5w3)--(4b3);
\draw[thick] (4w3)--(3b3);
\draw[thick] (4w2)--(3b2);
\draw[thick] (3w2)--(2b2);
\draw[thick] (2w2)--(1b2);
\draw[thick] (4w1)--(3b1);
\draw[thick] (3w1)--(2b1);
\draw[thick] (2w1)--(1b1);

\draw[thick, red] (b1)--(1w1);
\draw[thick, red] (b2)--(1w2);
\draw[thick, red] (b3)--(1b2);
\draw[thick, red] (b4)--(2b2);
\draw[thick, red] (b5)--(3w3);
\draw[thick, red] (b6)--(3b3);
\draw[thick, red] (b7)--(4b3);
\draw[thick, red] (b8)--(5w4);
\draw[thick, red] (b9)--(5b4);
\draw[thick, red] (b10)--(6b4);
\draw[thick, red] (b11)--(7w5);
\draw[thick, red] (b12)--(7b5);
 
\draw[color=black!70] (0,-4)--(7,-4);
\draw[color=black!70] (0,-3)--(5,-3);
\draw[color=black!70] (0,-2)--(3,-2);
\draw[color=black!70] (0,-1)--(1,-1);
\draw[color=black!70] (1,-5)--(1,-1);
\draw[color=black!70] (2,-5)--(2,-1);
\draw[color=black!70] (3,-5)--(3,-2);
\draw[color=black!70] (4,-5)--(4,-2);
\draw[color=black!70] (5,-5)--(5,-3);
\draw[color=black!70] (6,-5)--(6,-3);
\draw[line width=2] (0,0)--(1,0)--(1,-1)--(3,-1)--(3,-2)--(5,-2)--(5,-3)--(7,-3)--(7,-5);
\draw[line width=2,babyblueeyes!200,-Stealth] (b4)--(2b2)--(4w1)--(4b1)--(4w2)--(4b2)--(4w3)--(4b3);
\draw[line width=2,babyblueeyes!200] (4b3)--(b7);

\draw[babyblueeyes!250] (3.75,-2.25) node {4};
\draw[babyblueeyes!250] (4.75,-2.25) node {4};
\draw[babyblueeyes!250] (3.75,-3.25) node {4};
\draw[babyblueeyes!250] (3.75,-4.25) node {4};
\draw[babyblueeyes!250] (4.75,-3.25) node {4};
\draw (0.5,0.5) node {\scriptsize{$ $}};
\end{tikzpicture}

\begin{tikzpicture}[scale=.95]
\filldraw[color=lightgray] (0,-2)--(0,-5)--(3,-5)--(3,-3)--(2,-3)--(2,-2)--(0,-2);
\filldraw[color=navyblue!20] 
(5,-2.7)--(4.75, -2.75)--(4.25, -3.25)--(3.75, -3.75)--(3.25, -4.25)--(3.75, -4.75)--(4.25, -4.25)--(4.75, -4.75)--(5.25, -4.25)--(5.75, -4.75)--(6.25, -4.25)--(6.75, -4.75)--(7, -4.7)--(7,-5)--(3,-5)--(3,-3)--(2,-3)--(2,-2)--(0,-2)--(0,0)--(1,0)--(1,-1)--(3,-1)--(3,-2)--(5,-2)--(5,-2.7);
\draw[color=black!70] (0,0)--(0,-5)--(7,-5)--(7,-3)--(5,-3)--(5,-2)--(3,-2)--(3,-1)--(1,-1)--(1,0)--(0,0);

\node [scale =0.2, circle, draw=black, fill=black] (b12) at (0.3, 0) {};
\draw[red] (0.5,0.25) node {\scriptsize{$12$}};

\node [scale =0.2, circle, draw=black, fill=black] (b11) at (1, -0.7) {};
\draw[red] (1.2,-0.5) node {\scriptsize{$11$}};

\node [scale =0.2, circle, draw=black, fill=black] (b10) at (1.3, -1) {};
\draw[red] (1.5,-0.8) node {\scriptsize{$10$}};

\node [scale =0.2, circle, draw=black, fill=black] (b9) at (2.3, -1) 
{};
\draw[red] (2.5,-0.8) node {\scriptsize{$9$}};

\node [scale =0.2, circle, draw=black, fill=black] (b8) at (3, -1.7) 
{};
\draw[red] (3.2,-1.5) node {\scriptsize{$8$}};

\node [scale =0.2, circle, draw=black, fill=black] (b7) at (3.3, -2) {};
\draw[red] (3.5,-1.8) node {\scriptsize{$7$}};

\node [scale =0.2, circle, draw=black, fill=black] (b6) at (4.3, -2) {};
\draw[red] (4.5,-1.8) node {\scriptsize{$6$}};

\node [scale =0.2, circle, draw=black, fill=black] (b5) at (5, -2.7) {};
\draw[red] (5.2,-2.5) node {\scriptsize{$5$}};

\node [scale =0.2, circle, draw=black, fill=black] (b4) at (5.3, -3) {};
\draw[red] (5.5,-2.8) node {\scriptsize{$4$}};

\node [scale =0.2, circle, draw=black, fill=black] (b3) at (6.3, -3) {};
\draw[red] (6.5,-2.8) node {\scriptsize{$3$}};

\node [scale =0.2, circle, draw=black, fill=black] (b2) at (7, -3.7) {};
\draw[red] (7.2,-3.5) node {\scriptsize{$2$}};

\node [scale =0.2, circle, draw=black, fill=black] (b1) at (7, -4.7) {};
\draw[red] (7.2,-4.5) node {\scriptsize{$1$}};

\node [scale =0.5, circle, draw=black, fill=black] (7b5) at (0.25, -0.25) {};
\node [scale =0.5, circle, draw=black, fill=white] (7w5) at (0.75, -0.75) {};

\node [scale =0.5, circle, draw=black, fill=black] (7b4) at (0.25, -1.25) {};
\node [scale =0.5, circle, draw=black, fill=white] (7w4) at (0.75, -1.75) {};

\node [scale =0.5, circle, draw=black, fill=black] (6b4) at (1.25, -1.25) {};
\node [scale =0.5, circle, draw=black, fill=white] (6w4) at (1.75, -1.75) {};

\node [scale =0.5, circle, draw=black, fill=black] (5b4) at (2.25, -1.25) {};
\node [scale =0.5, circle, draw=black, fill=white] (5w4) at (2.75, -1.75) {};

\node [scale =0.5, circle, draw=black, fill=black] (5b3) at (2.25, -2.25) {};
\node [scale =0.5, circle, draw=black, fill=white] (5w3) at (2.75, -2.75) {};

\node [scale =0.5, circle, draw=black, fill=black] (4b3) at (3.25, -2.25) {};
\node [scale =0.5, circle, draw=black, fill=white] (4w3) at (3.75, -2.75) {};

\node [scale =0.5, circle, draw=black, fill=black] (4b2) at (3.25, -3.25) {};
\node [scale =0.5, circle, draw=black, fill=white] (4w2) at (3.75, -3.75) {};

\node [scale =0.5, circle, draw=black, fill=black] (4b1) at (3.25, -4.25) {};
\node [scale =0.5, circle, draw=black, fill=white] (4w1) at (3.75, -4.75) {};

\node [scale =0.5, circle, draw=black, fill=black] (3b3) at (4.25, -2.25) {};
\node [scale =0.5, circle, draw=black, fill=white] (3w3) at (4.75, -2.75) {};

\node [scale =0.5, circle, draw=black, fill=black] (3b2) at (4.25, -3.25) {};
\node [scale =0.5, circle, draw=black, fill=white] (3w2) at (4.75, -3.75) {};

\node [scale =0.5, circle, draw=black, fill=black] (3b1) at (4.25, -4.25) {};
\node [scale =0.5, circle, draw=black, fill=white] (3w1) at (4.75, -4.75) {};

\node [scale =0.5, circle, draw=black, fill=black] (2b2) at (5.25, -3.25) {};
\node [scale =0.5, circle, draw=black, fill=white] (2w2) at (5.75, -3.75) {};

\node [scale =0.5, circle, draw=black, fill=black] (2b1) at (5.25, -4.25) {};
\node [scale =0.5, circle, draw=black, fill=white] (2w1) at (5.75, -4.75) {};

\node [scale =0.5, circle, draw=black, fill=black] (1b2) at (6.25, -3.25) {};
\node [scale =0.5, circle, draw=black, fill=white] (1w2) at (6.75, -3.75) {};

\node [scale =0.5, circle, draw=black, fill=black] (1b1) at (6.25, -4.25) {};
\node [scale =0.5, circle, draw=black, fill=white] (1w1) at (6.75, -4.75) {};

\draw[thick] (7b5)--(7w5);
\draw[thick] (7b4)--(7w4);
\draw[thick] (6b4)--(6w4);
\draw[thick] (5b4)--(5w4);
\draw[thick] (5b3)--(5w3);
\draw[thick] (4b3)--(4w3);
\draw[thick] (4b2)--(4w2);
\draw[thick] (4b1)--(4w1);
\draw[thick] (3b3)--(3w3);
\draw[thick] (3b2)--(3w2);
\draw[thick] (3b1)--(3w1);
\draw[thick] (2b2)--(2w2);
\draw[thick] (2b1)--(2w1);
\draw[thick] (1b2)--(1w2);
\draw[thick] (1b1)--(1w1);

\draw[thick] (7w5)--(7b4);
\draw[thick] (5w4)--(5b3);
\draw[thick] (4w3)--(4b2);
\draw[thick] (4w2)--(4b1);
\draw[thick] (3w3)--(3b2);
\draw[thick] (3w2)--(3b1);
\draw[thick] (2w2)--(2b1);
\draw[thick] (1w2)--(1b1);

\draw[thick] (7w4)--(6b4);
\draw[thick] (6w4)--(5b4);
\draw[thick] (5w3)--(4b3);
\draw[thick] (4w3)--(3b3);
\draw[thick] (4w2)--(3b2);
\draw[thick] (3w2)--(2b2);
\draw[thick] (2w2)--(1b2);
\draw[thick] (4w1)--(3b1);
\draw[thick] (3w1)--(2b1);
\draw[thick] (2w1)--(1b1);

\draw[thick, red] (b1)--(1w1);
\draw[thick, red] (b2)--(1w2);
\draw[thick, red] (b3)--(1b2);
\draw[thick, red] (b4)--(2b2);
\draw[thick, red] (b5)--(3w3);
\draw[thick, red] (b6)--(3b3);
\draw[thick, red] (b7)--(4b3);
\draw[thick, red] (b8)--(5w4);
\draw[thick, red] (b9)--(5b4);
\draw[thick, red] (b10)--(6b4);
\draw[thick, red] (b11)--(7w5);
\draw[thick, red] (b12)--(7b5);
 
\draw[color=black!70] (0,-4)--(7,-4);
\draw[color=black!70] (0,-3)--(5,-3);
\draw[color=black!70] (0,-2)--(3,-2);
\draw[color=black!70] (0,-1)--(1,-1);
\draw[color=black!70] (1,-5)--(1,-1);
\draw[color=black!70] (2,-5)--(2,-1);
\draw[color=black!70] (3,-5)--(3,-2);
\draw[color=black!70] (4,-5)--(4,-2);
\draw[color=black!70] (5,-5)--(5,-3);
\draw[color=black!70] (6,-5)--(6,-3);
\draw[line width=2] (0,0)--(1,0)--(1,-1)--(3,-1)--(3,-2)--(5,-2)--(5,-3)--(7,-3)--(7,-5);
\draw[line width=2,navyblue,-Stealth] (b5)--(3w3)--(3b2)--(4w2)--(4b1)--(4w1)--(3b1)--(3w1)--(2b1)--(2w1)--(1b1)--(1w1);
\draw[line width=2,navyblue] (1w1)--(b1);

\draw[navyblue] (3.75,-2.25) node {5};
\draw[navyblue] (4.75,-2.25) node {5};
\draw[navyblue] (3.75,-3.25) node {5};
\draw[navyblue] (2.75,-1.25) node {5};
\draw[navyblue] (2.75,-2.25) node {5};
\draw[navyblue] (0.75,-0.25) node {5};
\draw[navyblue] (0.75,-1.25) node {5};
\draw[navyblue] (1.75,-1.25) node {5};
\draw[navyblue] (2.25,-2.75) node {5};

\draw (0.5,0.5) node {\scriptsize{$ $}};
\end{tikzpicture}

   \end{minipage}\hfill
   \begin{minipage}{0.48\textwidth}
     \centering
    \begin{tikzpicture}[scale=.95]
    \label{fig: trip on skew}
\filldraw[color=lightgray] (0,-2)--(0,-5)--(3,-5)--(3,-3)--(2,-3)--(2,-2)--(0,-2);
\filldraw[color=babyblueeyes!50](5,-3) -- (5,-4) -- (4,-4) -- (4,-5) -- (3,-5) -- (3,-2) -- (5,-2) -- (5,-3);
\draw[color=black!70] (0,0)--(0,-5)--(7,-5)--(7,-3)--(5,-3)--(5,-2)--(3,-2)--(3,-1)--(1,-1)--(1,0)--(0,0);
\draw[color=black!70] (0,-4)--(7,-4);
\draw[color=black!70] (0,-3)--(5,-3);
\draw[color=black!70] (0,-2)--(3,-2);
\draw[color=black!70] (0,-1)--(1,-1);
\draw[color=black!70] (1,-5)--(1,-1);
\draw[color=black!70] (2,-5)--(2,-1);
\draw[color=black!70] (3,-5)--(3,-2);
\draw[color=black!70] (4,-5)--(4,-2);
\draw[color=black!70] (5,-5)--(5,-3);
\draw[color=black!70] (6,-5)--(6,-3);
\draw[line width=2] (0,0)--(1,0)--(1,-1)--(3,-1)--(3,-2)--(5,-2)--(5,-3)--(7,-3)--(7,-5);
\draw[red,ultra thick] (0,0) -- (0,-0.5); 
\draw[red, ultra thick] (1,-1) -- (0.5,-1); 
\draw[red,ultra thick] (1,-1) -- (1,-1.5); 
\draw[red,ultra thick] (2,-1) -- (2,-1.5); 
\draw[red,ultra thick] (3,-2) -- (2.5,-2); 
\draw[red,ultra thick] (3,-2) -- (3,-2.5); 
\draw[red,ultra thick] (4,-2) -- (4,-2.5); 
\draw[red,ultra thick] (4.5,-3) -- (5,-3); 
\draw[red,ultra thick] (5,-3) -- (5,-3.5); 
\draw[red,ultra thick] (6,-3) -- (6,-3.5); 
\draw[red,ultra thick] (6.5,-4) -- (7,-4); 
\draw[red,ultra thick] (7,-5) -- (6.5,-5); 
\draw[line width=2,babyblueeyes!200,-Stealth] (5,-3) -- (5,-4) -- (4,-4) -- (4,-5) -- (3,-5) -- (3,-2);
\draw[red] (6.75,-4.8) node {\tiny{$1$}};
\draw[red] (6.75,-3.8) node {\tiny{$2$}};
\draw[red] (6.2,-3.25) node {\tiny{$3$}};
\draw[red] (5.2,-3.25) node {\tiny{$4$}};
\draw[red] (4.75,-2.8) node {\tiny{$5$}};
\draw[red] (4.2,-2.25) node {\tiny{$6$}};
\draw[red] (3.2,-2.25) node {\tiny{$7$}};
\draw[red] (2.75,-1.8) node {\tiny{$8$}};
\draw[red] (2.2,-1.25) node {\tiny{$9$}};
\draw[red] (1.2,-1.25) node {\tiny{$10$}};
\draw[red] (0.75,-0.8) node {\tiny{$11$}};
\draw[red] (0.2,-0.25) node {\tiny{$12$}};
\draw[babyblueeyes!250] (3.5,-2.5) node {4};
\draw[babyblueeyes!250] (4.5,-2.5) node {4};
\draw[babyblueeyes!250] (3.5,-3.5) node {4};
\draw[babyblueeyes!250] (3.5,-4.5) node {4};
\draw[babyblueeyes!250] (4.5,-3.5) node {4};
\draw (0.5,0.5) node {\scriptsize{$ $}};

\end{tikzpicture}

\begin{tikzpicture}[scale=.95]
    \label{fig: trip on skew}
\filldraw[color=lightgray] (0,-2)--(0,-5)--(3,-5)--(3,-3)--(2,-3)--(2,-2)--(0,-2);
\filldraw[color=navyblue!20](5,-3) -- (4,-3) -- (4,-4) -- (3,-4) -- (3,-3) -- (2,-3) -- (2,-2) -- (0,-2)--(0,0)--(1,0)--(1,-1)--(3,-1)--(3,-2)--(5,-2)--(5,-3);
\draw[color=black!70] (0,0)--(0,-5)--(7,-5)--(7,-3)--(5,-3)--(5,-2)--(3,-2)--(3,-1)--(1,-1)--(1,0)--(0,0);
\draw[color=black!70] (0,-4)--(7,-4);
\draw[color=black!70] (0,-3)--(5,-3);
\draw[color=black!70] (0,-2)--(3,-2);
\draw[color=black!70] (0,-1)--(1,-1);
\draw[color=black!70] (1,-5)--(1,-1);
\draw[color=black!70] (2,-5)--(2,-1);
\draw[color=black!70] (3,-5)--(3,-2);
\draw[color=black!70] (4,-5)--(4,-2);
\draw[color=black!70] (5,-5)--(5,-3);
\draw[color=black!70] (6,-5)--(6,-3);
\draw[line width=2] (0,0)--(1,0)--(1,-1)--(3,-1)--(3,-2)--(5,-2)--(5,-3)--(7,-3)--(7,-5);
\draw[red,ultra thick] (0,0) -- (0,-0.5); 
\draw[red, ultra thick] (1,-1) -- (0.5,-1); 
\draw[red,ultra thick] (1,-1) -- (1,-1.5); 
\draw[red,ultra thick] (2,-1) -- (2,-1.5); 
\draw[red,ultra thick] (3,-2) -- (2.5,-2); 
\draw[red,ultra thick] (3,-2) -- (3,-2.5); 
\draw[red,ultra thick] (4,-2) -- (4,-2.5); 
\draw[red,ultra thick] (4.5,-3) -- (5,-3); 
\draw[red,ultra thick] (5,-3) -- (5,-3.5); 
\draw[red,ultra thick] (6,-3) -- (6,-3.5); 
\draw[red,ultra thick] (6.5,-4) -- (7,-4); 
\draw[red,ultra thick] (7,-5) -- (6.5,-5); 
\draw[line width=2,navyblue,-Stealth] (5,-3) -- (4,-3) -- (4,-4) -- (3,-4) -- (3,-5) -- (7,-5);
\draw[red] (6.75,-4.8) node {\tiny{$1$}};
\draw[red] (6.75,-3.8) node {\tiny{$2$}};
\draw[red] (6.2,-3.25) node {\tiny{$3$}};
\draw[red] (5.2,-3.25) node {\tiny{$4$}};
\draw[red] (4.75,-2.8) node {\tiny{$5$}};
\draw[red] (4.2,-2.25) node {\tiny{$6$}};
\draw[red] (3.2,-2.25) node {\tiny{$7$}};
\draw[red] (2.75,-1.8) node {\tiny{$8$}};
\draw[red] (2.2,-1.25) node {\tiny{$9$}};
\draw[red] (1.2,-1.25) node {\tiny{$10$}};
\draw[red] (0.75,-0.8) node {\tiny{$11$}};
\draw[red] (0.2,-0.25) node {\tiny{$12$}};
\draw[navyblue] (0.5,-0.5) node {5};
\draw[navyblue] (0.5,-1.5) node {5};
\draw[navyblue] (1.5,-1.5) node {5};
\draw[navyblue] (2.5,-1.5) node {5};
\draw[navyblue] (2.5,-2.5) node {5};
\draw[navyblue] (3.5,-2.5) node {5};
\draw[navyblue] (4.5,-2.5) node {5};
\draw[navyblue] (3.5,-3.5) node {5};
\draw[navyblue] (1.5,-3.5) node {5};
\draw (0.5,0.5) node {\scriptsize{$ $}};

\end{tikzpicture}
\end{minipage}
 \caption{On the left, the plabic graph is shown superimposed on the skew diagram for $\skewschubert$ where $\lambda=(7,7,5,3,1)$ and $\mu=(3,3,2)$, with the path for $T_4$ highlighted in \textcolor{babyblueeyes!250}{blue} and the path for $T_5$ highlighted in \textcolor{navyblue}{navy blue}. This path is determined using the standard ``rules of the road" described in \cite{postnikov}, where the ``driver" keeps to the right side of the path. When  approaching a black dot, the driver continues straight, whereas approaching a white dot, the driver turns left. On the right, the lattice trip for $T_4$ and $T_5$ is depicted using lattice source labeling.}
    \label{fig:skew plabic and a trip}   
\end{figure}
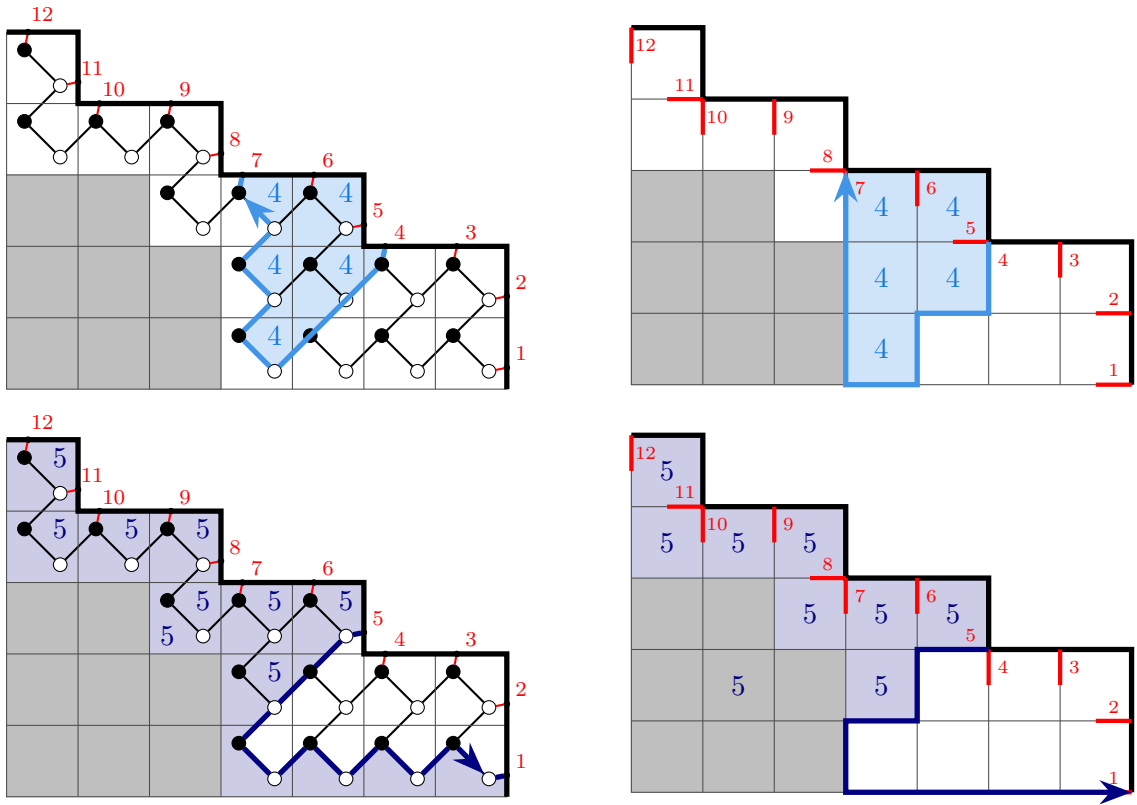

\begin{lemma}
\label{lem: source vs lattice source}

a) There is a bijection between the boxes in $\skewschubert$ and the regions in $G_{\skewschubert}$, except for the southwestern region $R_{SW}$. 

b) There is a bijection between the trips in $G_{\skewschubert}$ and lattice trips $T_i$ in $\skewschubert$ which agrees with the bijection for regions in (a).

c) Given a boundary edge $i$, the source labeling of regions for $G_{\skewschubert}$ and the lattice source labeling of boxes in $\skewschubert$ as in Definition \ref{def: trip on skew} agree. 

d) The southwestern region $R_{SW}$ in $G_{\skewschubert}$ is labeled by $I_{\mu}$.
\end{lemma}

\begin{proof}
Parts (a) and (b) are visually clear, see Figure \ref{fig:skew plabic and a trip}. For each box in $\skewschubert$ there is a unique region for $G_{\skewschubert}$ containing its northeast corner. The steps in a lattice trip are obtained from the steps in the corresponding trip for $G_{\skewschubert}$ by a $45$ degree rotation as in Figure \ref{fig:black and white bridge}. Part (c) is clear from (a) and (b). 

For (d) we observe that $R_{SW}$ is labeled by the sources of all counterclockwise trips in $G_{\skewschubert}$, while the counterclockwise lattice trips are labeled by the elements of $I_{\mu}$, see Remark \ref{rmk: source trip cases}. 
\end{proof}

\begin{remark}
A cautious reader might notice that the plabic graph $G_{\skewschubert}$ in the left of Figure \ref{fig:skew plabic and a trip} is not isomorphic to square grid (for example, most regions have 6 sides). Still, since we only use regions and their labels, we can use Lemma \ref{lem: source vs lattice source} to replace the data of $G_{\skewschubert}$ by its lattice analogue.
\end{remark}

\begin{figure}[!h]
    \centering
    \begin{tikzpicture}
    \draw[black!70] (0,0)--(0,1)--(1,1)--(1,0)--(0,0);
    
    \node [scale =0.5, circle, draw=black, fill=black] (b) at (0.25, 0.75) {};
    \node [scale =0.5, circle, draw=black, fill=white] (w) at (0.75, 0.25) {};
    
    \node [scale =0.2, circle, draw=black, fill=black] (b1) at (0.5, 1) {};
    \node [scale =0.2, circle, draw=black, fill=black] (b2) at (0, 0.5) {};
    \node [scale =0.2, circle, draw=black, fill=black] (b3) at (0.5, 0) {};
    \node [scale =0.2, circle, draw=black, fill=black] (b4) at (1, 0.5) {};
    
    \draw[thick] (b)--(w);
    \draw[thick] (b1)--(b)--(b2);
    \draw[thick] (b3)--(w)--(b4);
\end{tikzpicture}
\qquad\qquad
\begin{tikzpicture}
    \draw[black!70] (0,0)--(0,1)--(1,1)--(1,0)--(0,0);

    \node [scale =0.5, circle, draw=black, fill=black] (b) at (0, 1) {};
    \node [scale =0.5, circle, draw=black, fill=white] (w) at (0, 0) {};

    \node [scale =0.2,circle, draw=black, fill=black] (p1) at (1,1) {};
    \node [scale =0.2,circle, draw=black, fill=black] (p2) at (1,0) {};
    
    \draw[thick] (b)--(w);
    \draw[thick] (b)--(p1);
    \draw[thick] (w)--(p2);
\end{tikzpicture}
    \caption{Identification between plabic graph on $\skewschubert$ with $\skewschubert$}
    \label{fig:black and white bridge}
\end{figure}




\begin{lemma}
\label{lem: trip is f}
The trip permutation for the plabic graph $G_{\skewschubert}$ agrees  with bounded affine permutation $f_{\skewschubert}$.
\end{lemma}

\begin{proof}
By Lemma \ref{lem: source vs lattice source} we can replace the trips in $G_{\skewschubert}$ by the lattice trips $T_i$. So it is sufficient to prove that $T=f_{\skewschubert}\pmod n$.

We can label all the edges between boxes in $\skewschubert$ by their (lattice) distance to the southeast corner. Specifically, we label the southern (horizontal) and the eastern (vertical) boundary of $\sq_{a,i}$ by $a+i-1$. 
Along the trip $T_j$ we first move along the diagonal staircase  (see Figure \ref{fig:skew plabic and a trip}, right) where all edges have the same labels. In particular, the initial label $j$ equals the label on the boundary of $\mu$ that the trip hits. If we hit a vertical step of $\mu$, we move horizontally, and if we hit a horizontal one we move vertically - this agrees with the definition of $f_{\skewschubert}$ in Figure \ref{fig:wiring-diag-affine}. Therefore $T=f_{\skewschubert}\pmod n$.
\end{proof}
\begin{remark}
\label{rem: lollipop}
The fixed points of $f_{\skewschubert}\pmod n$ correspond to special ``lollipop" vertices in Definition \ref{def: G skew}(d). One can check that their decorations (as in \cite{postnikov}) agree as well.
\end{remark}

\begin{corollary}\label{cor: positroid var coincide}
The positroid variety associated to the plabic graph $G_{\skewschubert}$ coincides with $S^{\circ}_{\skewschubert}$.
\end{corollary}

\begin{proof}
One can check that the plabic graph $G_{\skewschubert}$ is reduced by verifying the conditions for trips in \cite[Theorem 13.2]{postnikov}. By \cite{postnikov} the positroid variety associated to a reduced plabic graph is determined by the decorated trip permutation, and the result follows from Lemma \ref{lem: trip is f} and Remark \ref{rem: lollipop}.
\end{proof}

\begin{lemma}\label{lem: cluster labeling coincide}
For given $\skewschubert$, the lattice source labeling at the square $\sq_{a,i}$ in $\skewschubert$ coincides with $I'(a,i)$. The source labeling at $\mu$ is $I_{\mu}$. 
\end{lemma}

\begin{proof}
In line with Remark \ref{rmk: source trip cases}, we consider two cases for a trip $T_l$.

{\bf Case A:} Let $l=b_m$ for some $1\le m\le k$, i.e., let $l\in I_\mu$. Define $\sq_{\widehat{a},\widehat{i}}$ as the first box where its rightmost edge is in $P_{\mu}\cap T_l$. 
Since $l\in I_\mu$, the trip $T_l$ has a counterclockwise orientation and we can group the boxes of $X_l$ into four types. The first type are boxes $\sq_{a,i}$ directly above the diagonal path, denoted $\overline{X}_l$. Note that $\sq_{a,i}\in\overline{X}_l$ then the $i$-th element (in increasing order) is $b_m$, since the vertical boundary of $\mu_{a,i}$ contains $\sq_{a,i}\in\overline{X}_l$.
The second type are boxes $\sq_{a,i}$ above $\overline{X}_l$, where $a<\widehat{a}$ and $i>\widehat{i}$. For these boxes, we have $l\in I'(a,i)$ from Lemma \ref{lem: J inductive}. The third type are boxes $\sq_{a,i}\in X_l$ such that $a\ge \widehat{a}$. For these boxes, the first $m$ terms of $I'(a,i)$ are $\{b_1,\dots,b_m\}$ since the boxes are above $\mu$ and after the $l$-th step in the path $P_l$, therefore, $l=b_m\in I'(a,i)$. The final type of boxes to consider in $X_l$ are boxes $\sq_{a,i}$ where $a<\widehat{a}$ and $i<\widehat{i}$. For these boxes, the last $k-m+1$ terms in $I'(a,i)$ are given by $\{b_m,\dots,b_k\}$ since they are below $\sq_{\widehat{a},\widehat{i}}$ which gives the vertical step associated to $l=b_m$ and have that $l\in I'(a,i)$.

{\bf Case B:} Let $l\ne b_m$ for some $1\le m\le k$, i.e., let $l\not\in I_\mu$. Then the trip $T_l$ forms a clockwise-orientated path and the boxes are given by $N_l$. Here we have two types of boxes. The first type are squares $\sq_{a,i}\in N_l$ that sit directly above the diagonal path formed from $T_l$ which are denoted $\overline{N}_l$. Similar to Case A, we have that $\sq_{a,i}\in \overline{N}_l$ if and only if the $i$-th term of $I'(a,i)$ (in increasing order) is $l$; therefore, $l\in I'(a,i)$. The second type are boxes $\sq_{a,i}$ above $\overline{N}_l$ up to the boundary of $\lambda$. For these boxes we use Lemma \ref{lem: J inductive} to see that $l\in I'(a,i)$.

\end{proof}
We are now ready to prove Theorem \ref{thm: skew schubert cluster}. 
\begin{proof}[Proof of Theorem \ref{thm: skew schubert cluster}]
The part (b) is a consequence of Corollary \ref{cor: positroid var coincide}; the quiver for $\skewschvar$ is the dual graph of $G_{\skewschubert}$ labeled by the source labeling. Lemma \ref{lem: cluster labeling coincide} implies that the cluster variables from $G_{\skewschubert}$ by source labeling (or, equivalently, lattice source labeling) are $\Delta_{I'(a,i)}$, which proves the part (a).
Under the identification of regions of the plabic graph $G_{\skewschubert}$ with boxes in $\skewschubert$ as in Lemma \ref{lem: source vs lattice source}, the arrow rule of $Q_{\skewschubert}$ in part (c) is a consequence of \cite{scott} and \cite[Section 14]{postnikov}. 
\end{proof}

\subsection{Cluster variable comparison} From Theorem \ref{thm: from skew schubert to braid variety}, the skew shaped positroid $S^{\circ}_{\skewschubert}$ is very closely related to the braid variety $X(\longest\br_{\skewschubert})$, that also admits a cluster structure, cf. \cite{CGGLSS22,SW}. In this section, we briefly compare the cluster variables in both constructions. 

An initial seed for the cluster structure on $X\left(\longest\br_{\skewschubert}\right)$ is explicitly constructed in \cite{SW}, see \cite[Proposition 5.20]{CGGLSS22}. The cluster variables in $X\left(\longest\br_{\skewschubert}\right)$ are in bijection with the crossings of the braid $\br_{\skewschubert}$. Let $\br_{\skewschubert} = \sigma_{i_1}\cdots \sigma_{i_{r}}$. If the flag right after the $j$-th crossing of $\br_{\skewschubert}$ is given by $(w_1|\cdots|w_k)$ then, up to a sign, the corresponding cluster variable is the $i_j$-th principal minor of the matrix $(w_1|\cdots|w_k)$. Note that this coincides with the determinant of $(w_1|\cdots|w_{i_j}|e_{i_j+1}|\dots|e_{k})$, where $e_1, \dots, e_k$ is the standard basis of $\C^k$.

If $(v_1|\dots|v_n)$ is a matrix representing an element of $S^{\circ}_{\skewschubert}$, we may normalize so that $v_{b_i} = e_i$, where $I_{\mu} = \{b_1 < \cdots < b_k\}$. Recall that, for a box $\sq_{a,i} \in \skewschubert$, the corresponding cluster variable 
is $\minor_{I'(a,i)}$. From Proposition \ref{prop: labeling relationship between I'(a,i) and I mu}, $I'(a,i) = J(a,i) \cup \{b_{i+1}, \dots, b_{k}\}$. Thus,
\[
\minor_{I'(a,i)}(V)=\det(v_{j_1}|\cdots|v_{j_i}|v_{b_{i+1}}|\cdots|v_{b_k}) = \det(v_{j_1}|\cdots|v_{j_i}|e_{i+1}|\cdots|e_{k}).
\]
where $J(a, i) = \{j_1, \dots, j_i\}$. Now, the space $V(a,i)$ is spanned by the vectors $j_1, \dots, j_i$, see Definition \ref{def: V(a,i)}. Since a braid box $\sq_{a,i} \in \skewschubert$ corresponds to a crossing between the $i$-th and $(i+1)$-st strands 
of $\br_{\skewschubert}$ we see that, upon setting the cluster variables corresponding to non-braid boxes equal to $1$, the cluster variables in the seed of $S^{\circ}_{\skewschubert}$ given by Theorem \ref{thm: skew schubert cluster} coincide, up to signs, with the cluster variables in the seed of $X\left(\longest\br_{\skewschubert}\right)$ obtained in \cite{SW}. 

Moreover, it follows by construction that, upon deleting the frozen vertices corresponding to non-braid boxes of $\skewschubert$, the quiver $Q_{\skewschubert}$ coincides with the \emph{opposite} of the quiver $Q_{X\left(\longest\br_{\skewschubert}\right)}$. Summarizing:

\begin{theorem}
    Consider the isomorphism $\widetilde{\Omega}^{1}: S^{\circ, 1}_{\skewschubert} \to X\left(\longest\br_{\skewschubert}\right)$. The dual map $\widetilde{\Omega}^{1,\ast}:\C\left[X\left(\longest\br_{\skewschubert}\right)\right] \to \C\left[S^{\circ, 1}_{\skewschubert}\right]$ sends, up to signs, the cluster variables in the initial seed of $X\left(\longest\br_{\skewschubert}\right)$ from \cite{SW} to the cluster variables in the initial seed of $S^{\circ, 1}_{\skewschubert}$ given by Corollary \ref{cor: skew schubert cluster}.
\end{theorem}

\section{Splicing for skew shaped positroids}
\label{sec: splicing}
In this section, we generalize the splicing map introduced in \cite{GS} to a splicing map for $S_{\skewschubert}^{\circ}$ and investigate its cluster algebraic interpretation.
\subsection{Freezing a column}

Let $S_{\skewschubert}^{\circ}$ be the skew shaped positroid in the Grassmannian $\Gr(k,n)$. 

\begin{definition}\label{def: open subset in skew}
    Given $1\le a\le n-k$, we define the open subset $U_{a}\subset S_{\skewschubert}^{\circ}$ by the inequalities \begin{equation}
        U_{a}:=\left\{V\in S_{\skewschubert}^{\circ}: \Delta_{I'(a,i)}\neq 0,\text{ for }\mubar_a+1\le i\le \lambdabar_a \right\}.
    \end{equation}

\end{definition}
Given a skew diagram $\skewschubert$, we cut it along the boundary of the $(a-1)$-st and $a$-th column (from the right) where the $(a-1)$-st column is contained in the right part of the diagram and the $a$-th column is contained in the left part of the diagram.
We get two skew diagrams, $\lambda^{a,L}/\mu^{a,L}$ on the left and $\lambda^{a,R}/\mu^{a,R}$ on the right.

\begin{remark}
    We can explicitly determine the resulting $\lambda$ and $\mu$ for the new diagrams using
    \begin{align*}
    &\lambda^{a,L}_i=\begin{cases}
        n-k-a+1,&\text{if }\lambda_i>n-k-a,\\\lambda_i,&\text{otherwise}
    \end{cases}\\
    &\mu^{a,L}_i=\begin{cases}
        n+k+a+1,&\text{if }\mu_i>n-k-a\\\mu_i,&\text{otherwise}
    \end{cases}\\    
    &\lambda^{a,R}_i=\begin{cases}
        \lambda_i-n+k+a-1,&\text{if }\lambda_i>n-k-a\\0,&\text{otherwise}
    \end{cases}\\
    &\mu^{a,R}_i=\begin{cases}
        \mu_i-n+k+a-1,&\text{if }\mu_i>n-k-a\\0,&\text{otherwise}.
    \end{cases}
    \end{align*}
\end{remark}

It is clear from definitions that  the corresponding braids satisfy
$$
\beta_{\skewschubert}=\beta_{L}\beta_{R},\ \mathrm{where}\ 
\beta_L=\beta_{\lambda^{a,L}/\mu^{a,L}},\ \beta_R=\beta_{\lambda^{a,R}/\mu^{a,R}}. 
$$

\begin{definition}
Given $1\le a\le n-k$, we define the \newword{flag at the cut} $\CF^{a}$ as the following sequence of subspaces:
\begin{equation}
\label{eq: flag cut}
\Wop_1   \subset\cdots \subset \Wop_{\mubar_a} \subset V\left(a,\mubar_a+1\right)\subset \cdots\subset  V\left(a,\lambdabar_a\right) \subset V\left(d_{\lambdabar_a+1},\lambdabar_a+1\right)\subset \cdots V(d_{k},k)
\end{equation}
where $\left(d_{\lambdabar_a+1},\lambdabar_a+1\right),\ldots (d_{k},k)$ are the rightmost boxes in rows $\lambdabar_a+1,\ldots,k$ as in \eqref{eq: def di}. 
\end{definition}
One can check (see Example \ref{ex: flag cut} below) that $\CF^a$ is the flag corresponding to the boundary between $\beta_L$ and $\beta_R$ under the map in Theorem \ref{thm: from skew schubert to braid variety}. In particular, \eqref{eq: flag cut} indeed defines a complete flag in $\C^k$.

\begin{lemma}
\label{lem: freeze column}
We have $V\in U_a$ if and only if the flag $\CF^a$  is transversal to the flag $\CF^W$.
\end{lemma}

\begin{proof}
The two flags are transversal if, for all $i$, we have $$\CF^a_i+W_{k-i}=\CF^a_i+\left\langle v_{b_{j+1}},\ldots,v_{b_k}\right\rangle=\C^k.$$ We have the following cases:

{\bf Case 1:} For $1\le i\le \mubar_a$, we have $\langle v_{b_1},\ldots,v_{b_i}\rangle+\langle v_{b_{i+1}},\ldots,v_{b_k}\rangle=\C^k$, this is automatic.

{\bf Case 2:} For $\mubar_a+1\le i\le \lambdabar_a$, we have $I'(a,i)=J(a,i)\cup \{b_{i+1},\ldots,b_k\}$ and $V(a,i)=\spann J(a,i)$. Therefore $V(a,i)+W_{k-i}=\C^k$ if and only if $\Delta_{I'(a,i)}\neq 0$, these are precisely the conditions defining $U_a$.

{\bf Case 3:} For $\lambdabar_a+1\le i\le k$, the box $\sq_{d_i,i}$ belongs to the boundary ribbon $R(\skewschubert)$, so $\Delta_{I'(d_i,i)}\neq 0$ and similarly to Case 2 we have 
$V(d_i,i)+W_{k-i}=\C^k$. 
\end{proof}

\begin{example}
\label{ex: flag cut}
In our running example, choose $a=6$. The subset $U_a$ is given by one condition $\Delta_{5,6,8,9,12}\neq 0$. The flag $\CF^a$ is given by 
$$
\langle v_5\rangle \subset \langle v_5,v_6\rangle\subset
\langle v_5,v_6,v_8\rangle\subset
\langle v_5,v_6,v_8,v_9\rangle\subset 
\langle v_5,v_6,v_8,v_{10},v_{11}\rangle,
$$
while the flag $\CF^W$ is given by 
$$
\langle v_{12}\rangle \subset \langle v_{12},v_{11}\rangle\subset
\langle v_{12},v_{11},v_8\rangle\subset
\langle v_{12},v_{11},v_8,v_6\rangle\subset \langle v_{12},v_{11},v_8,v_6,v_5\rangle.
$$
Note that 
$$
\langle v_5\rangle+\langle v_{12},v_{11},v_8,v_6\rangle=\langle v_5,v_6\rangle+\langle v_{12},v_{11},v_8\rangle=\langle v_5,v_6,v_8\rangle+\langle v_{12},v_{11}\rangle=\C^5
$$
since $\Delta_{5,6,8,11,12}\neq 0$. Furthermore,
$\langle v_5,v_6,v_8,v_9\rangle+\langle v_{12}\rangle =\C^5$ if and only if $\Delta_{5,6,8,9,12}\neq 0$ which is the defining condition for $U_a$. Finally, $\langle v_5,v_6,v_8,v_{10},v_{11}\rangle=\C^5$ since $\Delta_{5,6,8,10,11}\neq 0$ on $S^{\circ}_{\skewschubert}$.
\end{example}

\subsection{Splicing: motivation}

Recall Theorem \ref{thm: from skew schubert to braid variety} which relates the skew shaped positroid $S^{\circ}_{\skewschubert}$ to the braid variety $X\left(\longest\beta_{\skewschubert}\right)$ via the map $\Omega$.  Given $V\in S^{\circ}_{\skewschubert}$, we can draw the diagram 
\eqref{eq: flags from V} as follows:
\begin{equation}
\Omega(V)=\left[\CF^{W}\stackrel{\longest}{\dashrightarrow}\CF^{\Wop}\stackrel{\beta_{L}}{\dashrightarrow}\CF^a\stackrel{\beta_{R}}{\dashrightarrow}\CF^0\right].
\end{equation}
 By Lemma \ref{lem: freeze column}, we have $V\in U_a$ if and only if $\CF^a\pitchfork \CF^W$ in $\Omega(V)$. 

\begin{definition}
We define the \newword{splicing map}
$$
\Phi^{\mathrm{flag}}_a:U_a\to X\left(\longest\beta_L\right)\times X\left(\longest\beta_R\right),\ \Phi_a^{\mathrm{flag}}(\Omega(V))=\left(\Omega^L(V),\Omega^R(V)\right)
$$
as follows: 
\begin{equation}
\label{eq: Phi flags}
\Omega^L(V)=\left[\CF^{W}\stackrel{\longest}{\dashrightarrow}\CF^{\Wop}\stackrel{\beta_{L}}{\dashrightarrow}\CF^a\right],\quad  \Omega^R(V)=\left[\CF^{W}\stackrel{\longest}{\dashrightarrow}\CF^{a}\stackrel{\beta_{R}}{\dashrightarrow}\CF^0\right].
\end{equation}
\end{definition}

Note that the condition $\CF^a\pitchfork \CF^W$ ensures that $\Omega^L(V)$ and $\Omega^R(V)$ yield well defined points in the respective braid varieties, and the framings on the flag $\CF^W$ in $\Omega(V),\Omega^L(V)$ and $\Omega^R(V)$ agree.

\begin{remark}
The diagrams \eqref{eq: Phi flags} are defined in terms of unframed flags. We can upgrade them to framed flags using the standard framing on $\CF^W$ on the left, which can be propagated to the right using Remark \ref{rem: propagate framing}.

Note that the framings for all the flags between $\CF^{\Wop}$ and $\CF^a$ in $\Omega(V)$ and $\Omega^L(V)$ agree. However, $\Omega^R$ induces a possibly different framing on $\CF^a$ and on all flags between $\CF^a$ and $\CF^0$.
\end{remark}

Below we use Theorem \ref{thm: from skew schubert to braid variety} again to reinterpret $\Phi^{\mathrm{flag}}$ as a map
$$
\Phi_a:U_a\to S_{{\lambda^{a,L}/\mu^{a,L}}}^{\circ}\times S_{\lambda^{a,R}/\mu^{a,R}}^{\circ},
$$
and explore its cluster-theoretic properties. More precisely, we get the commutative diagram:
\begin{equation}
\label{eq: splicing diagram}
\begin{tikzcd}
U_a \arrow{r}{\Phi_a} \arrow{d}{\Omega} & S_{{\lambda^{a,L}/\mu^{a,L}}}^{\circ}\times S_{\lambda^{a,R}/\mu^{a,R}}^{\circ} \arrow{d}{(\Omega,\Omega)}\\
X\left(\longest\beta_{\skewschubert}\right) \arrow{r}{\Phi_a^{\mathrm{flag}}} & X\left(\longest\beta_L\right)\times X\left(\longest\beta_R\right).
\end{tikzcd}
\end{equation}
If $\Phi_a(V)=(V^L,V^R)$ then $\Omega\left(V^L\right)=\Omega^L,\Omega\left(V^R\right)=\Omega^R$.

\subsection{Splicing: Left diagram}
\label{sec: splicing left}

We have 
$$
\Omega^L(V)=\left[\CF^{W}\stackrel{\longest}{\dashrightarrow}\CF^{\Wop}\stackrel{\beta_{L}}{\dashrightarrow}\CF^a\right],
$$
in particular the pair of transversal flags $\CF^{\Wop}\pitchfork\CF^W$ still corresponds to the basis $v_{b_1},\ldots,v_{b_k}$. We need to understand what happens to subspaces $V(a,i)$.

\begin{lemma}
\label{lem: left diagram}
Assume that $a\le a'\le n-k$, then the long labels $I'(a',i)$ starts with $b_1,\ldots,b_{\mubar_a}$. Furthermore,
\begin{align}
\label{eq: left long}
I'_L(a',i)&=\left[I'(a',i)\setminus \{b_1,\ldots,b_{\mubar_a}\}-a+1\right]\cup \{1,\ldots,\mubar_a\},\\
I_{\mu^L}&=
    \left[I_{\mu}\setminus \{b_1,\ldots,b_{\mubar_a}\}-a+1\right]\cup \{1,\ldots,\mubar_a\}.
\end{align}
\end{lemma}

\begin{proof}
For $a+1\le a'\le n-k$ the path $P_{a',i}$ starts with the boundary of $\mu$ contained to the right of the cut, with vertical steps $b_1,\ldots,b_{\mubar_a}$. The similar path $P^{L}_{a',i}$ in the left diagram starts with $\mubar_a$ vertical steps labeled by $1,\ldots,\mubar_a$ instead. The remainders of the paths coincide, but the labels are shifted by $a-1$ which is the number of horizontal steps in $P_{a',i}$ before the cut.
\end{proof}

\begin{definition} 
Given $V=(v_1|\cdots|v_n)$ we define $V^{L}=(w_1|\cdots|w_{n-a+1})\in \Gr(k,n-a+1)$ where 
\begin{equation}
\label{eq: def w}
w_1=v_{b_1},\ldots,w_{\mubar_a}=v_{b_{\mubar_a}}\quad \mathrm{and}\quad 
w_j=v_{j+a-1},\quad \mubar_a<j\le n-a+1.
\end{equation}
\end{definition}

\begin{corollary}
The minors for $V^L$ and $V$ in the left diagram $\lambda^{a,L}/\mu^{a,L}$ coincide after relabeling. In particular, $V^L$ corresponds to a well-defined point in the  skew shaped positroid $S_{\lambda^{a,L}/\mu^{a,L}}^\circ$.
\end{corollary}

\begin{example}
In our running example, choose $a=6$ and get the following left diagram
\begin{center}
\begin{tikzpicture}
\filldraw[lightgray](0,0)--(0,3)--(2,3)--(2,0)--(0,0);
\draw (0,0)--(0,5)--(1,5)--(1,4)--(2,4)--(2,0)--(0,0);  
\draw (1,0)--(1,4);
\draw (0,1)--(2,1);
\draw (0,2)--(2,2);
\draw (0,3)--(2,3);
\draw (0,4)--(1,4);
\draw (1.5,1.5) node {\scriptsize 12367};
\draw (1.5,3.5) node {\scriptsize 12347};
\draw (0.5,3.5) node {\scriptsize 12357};
\draw (0.5,4.5) node {\scriptsize 12356};
\end{tikzpicture}
\end{center}
which defines a skew shaped positroid in $\Gr(5,7)$.
If we denote
$$
\mathbf{w_1=v_5,\ w_2=v_6,\ w_3=v_8},\ 
w_4=v_9,\ w_5=v_{10},\ w_6=v_{11},\ w_7=v_{12},\ 
$$
then the labels would match the ones in the left side of Example \ref{ex: running labels}. Note that $\mubar_a=3$ and $(b_1,b_2,b_3)=(5,6,8)$. 
\end{example}

\medskip

\begin{remark}
If $\mubar_a=0$, so that the cut is performed to the right of $\mu$, we get 
$V^L=\left(w_1|\cdots|w_{n-a+1}\right)=(v_a|\cdots|v_n)$ as in \cite{GS}.
\end{remark}

\begin{remark}
We have $\minor_{I_{\mu^L}}\left(V^L\right)=\minor_{I_{\mu}}(V)=1$, which is consistent with Remark \ref{rmk:Imu1}. 
\end{remark}

\subsection{Splicing: Right diagram}
\label{sec: splicing right}

We have 
$$
\Omega^R(V)=\left[\CF^{W}\stackrel{\longest}{\dashrightarrow}\CF^{a}\stackrel{\beta_{R}}{\dashrightarrow}\CF^0\right].
$$
We define a point $V^R=(u_1|\cdots|u_{k+a-1})$ in the skew shaped positroid $S_{\skewright}^\circ\subset  \Gr(k,a+k-1)$ in two steps. Observe that 
$$
I_{\mu^R}=\left\{b_1^R,\ldots,b_k^R\right\}=
\left\{b_1,\ldots,b_{\mubar_a},a+\mubar_a,\ldots,a+k-1\right\}
.$$ As in Lemma \ref{lem: flags}, we can obtain the basis $u_{b_1},\ldots,u_{b_{\mubar_a}},u_{a+\mubar_a},\ldots,u_{a+k-1}$ by intersecting two transverse flags $\CF^a$ and $\CF^W$. 

\begin{definition}
\label{def: right step 1}
For $i=1,\ldots,k$, we define $u_{b_i^R}$ as the unique vector such that
$$
u_{b_i^R}\in \CF^a_{i}\cap W_{k-i+1}
$$
and
\begin{equation}
\label{eq: step 1 triang}
u_{b_i^R}\in v_{b_i}+\left\langle v_{b_{i+1}},\ldots,v_{b_{k}}\right\rangle,\quad i=1,\ldots,k.
\end{equation}
\end{definition}
The condition \eqref{eq: step 1 triang} ensures that $v_{b_i}$ and $u_{b_i^R}$ induce the same framing on the flag $\CF^W$.
Note that we get 
$$
u_{b_1}=v_{b_1},\ldots,u_{b_{\mubar_a}}=v_{b_{\mubar_a}},
$$
$$
u_{a+i-1}\in V(a,i)\cap W_{k-i+1},\quad \mubar_a+1\le i\le \lambdabar_a
$$
and 
$$
u_{a+i-1}\in V(d_i,i)\cap W_{k-i+1},\quad \lambdabar_a+1\le i\le k.
$$

\begin{lemma}
Suppose that $\mubar_a+1\le i\le \lambdabar_a$. 
Then
\begin{equation}
\label{eq: step 1 other triang}
u_{a+i-1}\in\begin{cases}
\frac{\Delta_{I'(a,i-1)}}{\Delta_{I'(a,i)}}v_{a+i-1}+V(a,i-1) & \mathrm{if}\ i>\mubar_a+1,\\
\frac{\Delta_{I_{\mu}}}{\Delta_{I'(a,i)}}v_{a+i-1}+\left\langle v_{b_1},\ldots,v_{b_{\mubar_a}}\right\rangle 
 & \mathrm{if}\ i=\mubar_a+1.
\end{cases}
\end{equation}
\end{lemma}
\begin{proof}
Assume $i>\mubar_a+1$ first. Recall that $V\in U_a$, so $\Delta_{I'(a,i)}(V)\neq 0$ and the subset $I'(a,i)=J(a,i)\cup \{b_{i+1},\ldots,b_k\}$ labels a basis in $\C^k$. If we expand $v_{b_i}$ in this basis, we get
\begin{equation}
\label{eq: cramer}
v_{b_i}=\sum_{\alpha\in J(a,i)}c_{\alpha}v_{\alpha}+c_{b_{i+1}}v_{b_{i+1}}+\cdots+c_{b_k}v_{b_k},
\end{equation}
and 
$$
u_{a+i-1}=v_{b_i}-c_{b_{i+1}}v_{b_{i+1}}-\cdots-c_{b_k}v_{b_k}=\sum_{\alpha\in J(a,i)}c_{\alpha}v_{\alpha}.
$$
Recall that $a+i-1\in J(a,i)$.
By Cramer's Rule the coefficient $c_{a+i-1}$ at $v_{a+i-1}$ in \eqref{eq: cramer} equals
$$
\frac{\Delta_{I'(a,i)
\cup b_i\setminus (a+i-1)}}{\Delta_{I'(a,i)}}=\frac{\Delta_{I'(a,i-1)}}{\Delta_{I'(a,i)}}.
$$
The case $i=\mubar_a+1$ is almost identical, but in this case 
$I'(a,i)\cup b_i\setminus (a+i-1)=I_{\mu}$.
\end{proof}

\begin{example}
\label{ex: right easy}
Following Example \ref{ex: flag cut}, we get
$$
u_5=v_5,\ u_6=v_6,\ u_8=v_8,\ u_9\in V(6,4)\cap W_{2}=\langle v_5,v_6,v_8,v_9\rangle\cap \langle v_{11},v_{12}\rangle, u_{10}=v_{12}.
$$
By Lemma \ref{lem: 1d intersection new} we get that $u_9$ is proportional to $v_9$. By \eqref{eq: step 1 other triang} we get
$$
u_9=\frac{\Delta_{5,6,8,11,12}}{\Delta_{5,6,8,9,12}}v_9.
$$
\end{example}

Equation \eqref{eq: step 1 triang} immediately implies the following.

\begin{lemma}
\label{lem: step 1 wedge}
For all $1\le i\le k$ we have $u_{b_i^R}\wedge \cdots u_{b_k^R}=v_{b_i}\wedge \cdots v_{b_k}$. In particular, 
$$
\Delta_{I_{\mu^R}}\left(V^R\right)=\Delta_{I_{\mu}}(V)=1.
$$
\end{lemma}

Let us reconstruct the rest of the vectors $u_i$, following the proof of Theorem \ref{thm: from skew schubert to braid variety}. Recall that all  $u$'s not described by Definition \ref{def: right step 1} can be written as $u_{a'+\mubar_{a'}}$ for $1\le a'\le a-1$. 

\begin{lemma}
\label{lem: step 2 triang}
We can choose $u_{a'+\mubar_{a'}}=v_{a'+\mubar_{a'}}$ for all $a'$. This is a unique (up to scalar) basic vector in the intersection 
$V(a',\mubar_{a'}+1)\cap \left\langle u_{b^R_{\mubar_{a'}+1}},\ldots,u_{b^R_{k}}\right\rangle.$
\end{lemma}

\begin{proof}
By \eqref{eq: step 1 triang} we get
$$
\bigg\langle u_{b^R_{\mubar_{a'}+1}},\ldots,u_{b^R_{k}}\bigg\rangle=\left\langle v_{b_{\mubar_{a'}+1}},\ldots,v_{b_{k}}\right\rangle=W_{k-i}.
$$
Now the statement follows from Lemma \ref{lem: 1d intersection new}.
\end{proof}

Next, we need to compute the minors in the right diagram. 

\begin{lemma}
\label{lem: minors triang}
Let $1\le a'\le a-1$. Then 
$$
\Delta_{I'_R(a',i)}\left(V^R\right)=\Delta_{I'(a',i)}(V)A_{a'+\mubar_{a'}}\cdots A_{a'+i-1},
$$
where the scalars $A_t$ are defined by \eqref{eq: def A} below.
\end{lemma}
\begin{proof}
By \eqref{eq: step 1 other triang} and Lemma \ref{lem: step 2 triang} the basis $\{u_t\}$ is triangular with respect to $\{v_t\}$, up to some relabeling. More precisely, for $a'+\mubar_{a'}\le t\le a'+i-1$   we have 
\begin{equation}
\label{eq: triang summary}
u_t\in A_tv_t+\spann(v_{b_j}\mid b_j<t)+\spann(v_s\mid \mubar_a<s<t).
\end{equation}
where the leading coefficient $A_t$ is given by 
\begin{equation}
\label{eq: def A}
A_t=\begin{cases}
\Delta_{I'(a,t-a)}/\Delta_{I'(a,t-a+1)} & \mathrm{if}\ t\in \{a+\mubar_a,\ldots,a+\lambdabar_a-1\},\\
1 & \mathrm{otherwise}.\\
\end{cases}
\end{equation}
We claim that for all $t\in I'(a',i)$ all lower degree terms in the expansion \eqref{eq: triang summary} have indices which are also contained in $I'(a',i)$. This would ensure that the minors $\Delta_{I'_R(a',i)}\left(V^R\right)$ and $\Delta_{I'(a',i)}(V)$ are proportional.
Indeed, we have 
$$
I'_R(a',i)=J(a',i)\cup\{b_{i+1}^R,\ldots,b_k^R\}=\{b_1,\ldots,b_{\mubar_{a'}}\}\cup \{a'+\mubar_{a'},\ldots,a'+i-1\}\cup \{b_{i+1}^R,\ldots,b_k^R\},
$$
while
$$
I'(a',i)=J(a',i)\cup\{b_{i+1},\ldots,b_k\}=\{b_1,\ldots,b_{\mubar_{a'}}\}\cup \{a'+\mubar_{a'},\ldots,a'+i-1\}\cup \{b_{i+1},\ldots,b_k\}.
$$
Therefore 
\begin{multline*}
u_{b_1}\wedge\cdots \wedge u_{b_{\mubar_{a'}}}\wedge u_{a'+\mubar_{a'}}\wedge\cdots \wedge u_{a'+i-1}=(1\cdots 1\cdot A_{a'+\mubar_{a'}}\cdots A_{a'+i-1})\times \\
v_{b_1}\wedge\cdots \wedge v_{b_{\mubar_{a'}}}\wedge v_{a'+\mubar_{a'}}\wedge\cdots \wedge v_{a'+i-1},
\end{multline*}
Finally, by Lemma \ref{lem: step 1 wedge} we get
$$
u_{b_{i+1}^R}\wedge\cdots \wedge u_{b_k^R}=v_{b_{i+1}}\wedge\cdots \wedge v_{b_k},
$$
and the result follows.
\end{proof}

Next, we consider the quiver $Q$ for the cluster structure on $S^{\circ}_{\skewschubert}$ and the quiver $Q^R$ for the cluster structure on
$S^{\circ}_{\lambda^R/\mu^R}$. Clearly, $Q^R$ is a subgraph of $Q$. We denote by $\widehat{y}(a',i)$ and $\widehat{y}^R(a',i)$ the  exchange ratios at the vertices $I'(a',i)$ and $I'_R(a',i)$ respectively.

\begin{lemma}
\label{lem: exchange ratios not bdry}
Assume that $a'<a-1$ and $I'_R(a',i)$ is a mutable box of $Q^R$. Then the exchange ratios at $(a',i)$ agree: $\widehat{y}(a',i)=\widehat{y}^R(a',i)$.
\end{lemma}

\begin{proof}
Let us denote the factor in Lemma \ref{lem: minors triang} by
$$
A(a',i)=A_{a'+\mubar_{a'}}\cdots A_{a'+i-1}.
$$
Note that $A(a',i+1)=A(a',i)\cdot A_{a'+i}$. Now we need to compute the exchange ratios  for a mutable box $\sq_{a',i}$. Since it is mutable, it is not in the boundary ribbon and $\sq_{a'-1,i},\sq_{a'-1,i+1},\sq_{a',i+1}$ are all in $\lambda_R/\mu_R$. For the other neighbors of  $\sq_{a',i}$, we make the following observations:
\begin{itemize}
\item[(a)] Either the box $\sq_{a',i-1}$ is in $\lambda_R/\mu_R$, or it is in $\mu_R$. In the latter case $i=\mubar_{a'}+1$ and $\sq_{a'+1,i-1}$ is also  in $\mu$.
\item[(b)] If $a'<a-1$, then either the box $\sq_{a'+1,i}$ is in $\lambda_R/\mu_R$, or it is in $\mu_R$. In the latter case $\sq_{a'+1,i-1}$ is also in $\mu$, and we have $b_i=a'+i$ (since this is the vertical step at the boundary of $\mu$ at height $i$).
\end{itemize}

\noindent {\bf Case 1}: All six neighbors of $\sq_{a',i}$ are present:
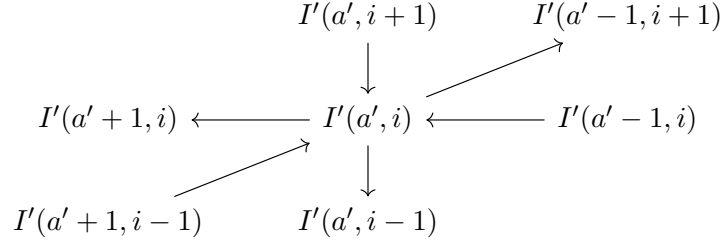
\begin{center}
\begin{tikzcd}
 & I'(a',i+1) \arrow{d} & I'(a'-1,i+1)\\
I'(a'+1,i) & I'(a',i) \arrow{ur} \arrow{d} \arrow{l} & I'(a'-1,i) \arrow{l}\\
I'(a'+1,i-1) \arrow{ur} & I'(a',i-1) & \\
\end{tikzcd}
\end{center}
The exchange ratio gets multiplied by
$$
\frac{A(a'-1,i)}{A(a'-1,i+1)}\cdot \frac{A(a',i+1)}{A(a',i-1)}\cdot \frac{A(a'+1,i-1)}{A(a'+1,i)}=\frac{1}{A_{a'+i-1}}\cdot (A_{a'+i}A_{a'+i-1})\cdot \frac{1}{A_{a'+i}}=1.
$$
{\bf Case 2}: 
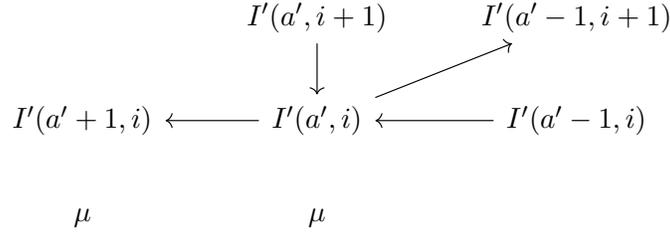
\begin{center}
\begin{tikzcd}
 & I'(a',i+1) \arrow{d} & I'(a'-1,i+1)\\
I'(a'+1,i) & I'(a',i) \arrow{ur}   \arrow{l} & I'(a'-1,i) \arrow{l}\\
\mu& \mu & \\
\end{tikzcd}
\end{center}
In this case  the exchange ratio gets multiplied by
$$
\frac{A(a'-1,i)}{A(a'-1,i+1)}\cdot \frac{A(a',i+1)}{A(a'+1,i)}=\frac{1}{A_{a'+i-1}}\cdot \frac{A_{a'+i}A_{a'+i-1}}{A_{a'+i}}=1.
$$

{\bf Case 3:}
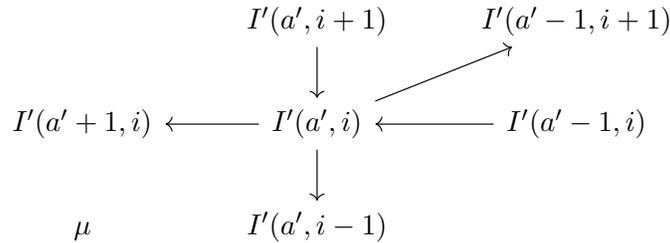
\begin{center}
\begin{tikzcd}
 & I'(a',i+1) \arrow{d} & I'(a'-1,i+1)\\
I'(a'+1,i) & I'(a',i) \arrow{ur}   \arrow{l} \arrow{d}& I'(a'-1,i) \arrow{l}\\
\mu& I'(a',i-1) & \\
\end{tikzcd}
\end{center}
The exchange ratio gets multiplied by
$$
\frac{A(a'-1,i)}{A(a'-1,i+1)}\cdot \frac{A(a',i+1)}{A(a',i-1)}\cdot \frac{1}{A(a'+1,i)}=\frac{1}{A_{a'+i-1}}\cdot (A_{a'+i}A_{a'+i-1})\cdot \frac{1}{A_{a'+i}}=1.
$$
{\bf Case 4:}
\begin{center}
\begin{tikzcd}
 & I'(a',i+1) \arrow{d} & I'(a'-1,i+1)\\
\mu & I'(a',i) \arrow{ur}    \arrow{d}& I'(a'-1,i) \arrow{l}\\
\mu& I'(a',i-1) & \\
\end{tikzcd}
\end{center}
The exchange ratio gets multiplied by
$$
\frac{A(a'-1,i)}{A(a'-1,i+1)}\cdot \frac{A(a',i+1)}{A(a',i-1)}=\frac{1}{A_{a'+i-1}}\cdot (A_{a'+i}A_{a'+i-1})=A_{a'+i}.
$$
But by observation (b) above we have $a'+i=b_i$ and $A_{a'+i}=A_{b_i}=1$.

\noindent {\bf Case 5:}
\begin{center}
\begin{tikzcd}
 & I'(a',i+1) \arrow{d} & I'(a'-1,i+1)\\
\mu & I'(a',i) \arrow{ur}     & I'(a'-1,i) \arrow{l}\\
\mu& \mu & \\
\end{tikzcd}
\end{center}
$$
\frac{A(a'-1,i)}{A(a'-1,i+1)}\cdot  A(a',i+1)=\frac{1}{A_{a'+i-1}}\cdot (A_{a'+i}A_{a'+i-1})=A_{a'+i}.
$$
Similarly to the previous case, we have $A_{a'+i}=1$.
\end{proof}

Now we deal with the left boundary where $a'=a-1$.

\begin{lemma}
\label{lem: exchange ratios bdry}
Suppose that $a'=a-1$ and $I'_R(a',i)$ is a mutable box of $Q^R$. Then the exchange ratios at $(a',i)$ agree: $\widehat{y}(a',i)=\widehat{y}^R(a',i)$.
\end{lemma}

\begin{proof}
 We follow the cases in the proof of Lemma \ref{lem: exchange ratios not bdry}. In each case, we describe the contribution from $I'(a'+1,i)$ and $I'(a'+1,i-1)$ which are present in $\lambda/\mu$ but missing in $\lambda_R/\mu_R$ and call these ``missing factors".

\noindent {\bf Case 1:} The missing factor is $\Delta_{I'(a,i-1)}/\Delta_{I'(a,i)}$ while the rest of the exchange ratio is multiplied by
$$
\frac{A(a'-1,i)}{A(a'-1,i+1)}\cdot \frac{A(a',i+1)}{A(a',i-1)}=\frac{1}{A_{a'+i-1}}\cdot (A_{a'+i}A_{a'+i-1})=A_{a'+i}.
$$
Note that $a'+i=a+i-1\in \{a+\mubar_a,\ldots,a+\lambdabar_a-1\}$, so by \eqref{eq: def A} we get
$$
A_{a'+i}=\Delta_{I'(a,i-1)}/\Delta_{I'(a,i)},
$$
so the missing factor is restored.

\noindent {\bf Case 2:} The missing factor is $1/\Delta_{I'(a'+1,i)}$  
while the rest of the exchange ratio is multiplied by
$$
\frac{A(a'-1,i)}{A(a'-1,i+1)}\cdot A(a',i+1) =\frac{1}{A_{a'+i-1}}\cdot (A_{a'+i}A_{a'+i-1})=A_{a'+i}.
$$
Note that $a'+i=a+\mubar_a$, so $A_{a'+i}=\Delta_{I_{\mu}}/\Delta_{I'(a,i)}=1/\Delta_{I'(a,i)},$ 

{\bf Case 3} is similar to {\bf Case 2}, and {\bf Case 4} and {\bf Case 5} follow Lemma \ref{lem: exchange ratios not bdry} verbatim since there are no missing factors.

\end{proof}

\subsection{Example}
\label{sec: intro example details}

Here we discuss Example \ref{ex: intro} in detail. Recall that $\lambda=(7,7,5,3,1), \mu=(3,1)$ and $a=6$.  We can label the boxes of $\skewschubert$ by $I'(a,i)$ and draw the quiver as follows:

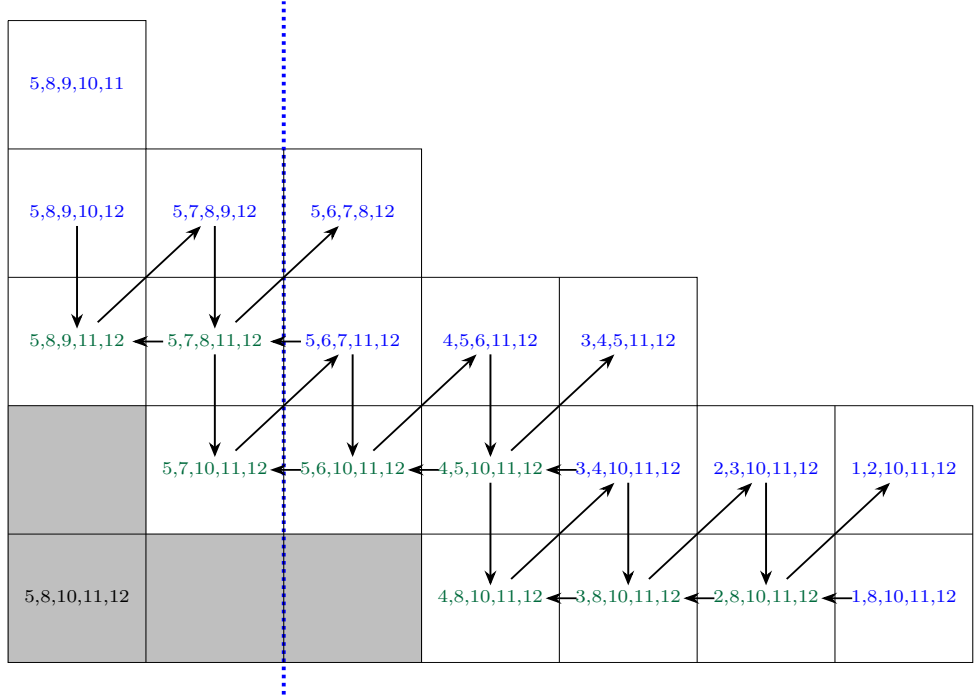
\begin{center}
\resizebox{5.1in}{3.65in}{
\begin{tikzpicture}
\filldraw[color=lightgray,scale=2] (0,-3)--(0,-5)--(3,-5)--(3,-4)--(1,-4)--(1,-3)--(0,-3);
\draw[scale=2] (0,0)--(0,-5)--(7,-5)--(7,-3)--(5,-3)--(5,-2)--(3,-2)--(3,-1)--(1,-1)--(1,0)--(0,0);
\draw [dotted, ultra thick, color=blue, scale=2] (2,-5.25)--(2,0.15);
\draw[scale=2] (0,-4)--(7,-4);
\draw[scale=2] (0,-3)--(5,-3);
\draw[scale=2] (0,-2)--(3,-2);
\draw[scale=2] (0,-1)--(1,-1);
\draw[scale=2] (1,-5)--(1,-1);
\draw[scale=2] (2,-5)--(2,-1);
\draw[scale=2] (3,-5)--(3,-2);
\draw[scale=2] (4,-5)--(4,-2);
\draw[scale=2] (5,-5)--(5,-3);
\draw[scale=2] (6,-5)--(6,-3);
\draw (13,-9) node {\textcolor{blue}{\scriptsize $1{,}8{,}10{,}11{,}12$}};
\draw (11,-9) node {\textcolor{cadmiumgreen}{\scriptsize $2{,}8{,}10{,}11{,}12$}};
\draw (9,-9) node {\textcolor{cadmiumgreen}{\scriptsize $3{,}8{,}10{,}11{,}12$}};
\draw (7,-9) node {\textcolor{cadmiumgreen}{\scriptsize $4{,}8{,}10{,}11{,}12$}};
\draw (13,-7) node {\textcolor{blue}{\scriptsize $1{,}2{,}10{,}11{,}12$}};
\draw (11,-7) node {\textcolor{blue}{\scriptsize $2{,}3{,}10{,}11{,}12$}};
\draw (9,-7) node {\textcolor{blue}{\scriptsize $3{,}4{,}10{,}11{,}12$}};
\draw (7,-7) node {\textcolor{cadmiumgreen}{\scriptsize $4{,}5{,}10{,}11{,}12$}};
\draw (5,-7) node {\textcolor{cadmiumgreen}{\scriptsize $5{,}6{,}10{,}11{,}12$}};
\draw (3,-7) node {\textcolor{cadmiumgreen}{\scriptsize $5{,}7{,}10{,}11{,}12$}};

\draw (9,-5) node {\textcolor{blue}{\scriptsize $3{,}4{,}5{,}11{,}12$}};
\draw (7,-5) node {\textcolor{blue}{\scriptsize $4{,}5{,}6{,}11{,}12$}};
\draw (5,-5) node {\textcolor{blue}{\scriptsize $5{,}6{,}7{,}11{,}12$}};
\draw (3,-5) node {\textcolor{cadmiumgreen}{\scriptsize $5{,}7{,}8{,}11{,}12$}};
\draw (1,-5) node {\textcolor{cadmiumgreen}{\scriptsize $5{,}8{,}9{,}11{,}12$}};

\draw (5,-3) node {\textcolor{blue}{\scriptsize $5{,}6{,}7{,}8{,}12$}};
\draw (3,-3) node {\textcolor{blue}{\scriptsize $5{,}7{,}8{,}9{,}12$}};
\draw (1,-3) node {\textcolor{blue}{\scriptsize $5{,}8{,}9{,}10{,}12$}};
\draw (1,-1) node {\textcolor{blue}{\scriptsize $5{,}8{,}9{,}10{,}11$}};
\draw (1,-9) node{\scriptsize $5{,}8{,}10{,}11{,}12$};

\draw[thick, -Stealth] (12.25,-9) to (11.8,-9);
\draw[thick, -Stealth] (10.25,-9) to (9.8,-9);
\draw[thick, -Stealth] (8.25,-9) to (7.8,-9);

\draw[thick, -Stealth] (8.25,-7) to (7.8,-7);
\draw[thick, -Stealth] (6.25,-7) to (5.8,-7);
\draw[thick, -Stealth] (4.25,-7) to (3.8,-7);

\draw[thick, -Stealth] (4.25,-5) to (3.8,-5);
\draw[thick, -Stealth] (2.25,-5) to (1.8,-5);

\draw[thick, -Stealth] (11,-7.2) to (11,-8.8);
\draw[thick, -Stealth] (9,-7.2) to (9,-8.8);
\draw[thick, -Stealth] (7,-7.2) to (7,-8.8);

\draw[thick, -Stealth] (7,-5.2) to (7,-6.8);
\draw[thick, -Stealth] (5,-5.2) to (5,-6.8);
\draw[thick, -Stealth] (3,-5.2) to (3,-6.8);

\draw[thick, -Stealth] (3,-3.2) to (3,-4.8);
\draw[thick, -Stealth] (1,-3.2) to (1,-4.8);

\draw[thick, -Stealth] (11.3,-8.7) to (12.8,-7.2);
\draw[thick, -Stealth] (9.3,-8.7) to (10.8,-7.2);
\draw[thick, -Stealth] (7.3,-8.7) to (8.8,-7.2);

\draw[thick, -Stealth] (7.3,-6.7) to (8.8,-5.2);
\draw[thick, -Stealth] (5.3,-6.7) to (6.8,-5.2);
\draw[thick, -Stealth] (3.3,-6.7) to (4.8,-5.2);

\draw[thick, -Stealth] (3.3,-4.7) to (4.8,-3.2);
\draw[thick, -Stealth] (1.3,-4.7) to (2.8,-3.2);
\end{tikzpicture}}
\end{center}

The open subset $U_a$ is defined by inequalities 
$$
\Delta_{5,7,10,11,12}\neq 0, \Delta_{5,7,8,11,12}\neq 0,
\Delta_{5,7,8,9,12}\neq 0.
$$
The flag at the cut is given by 
$$
\CF^a=\left\{\langle v_5\rangle\subset \langle v_5,v_7\rangle\subset  \langle v_5,v_7, v_8\rangle\subset  \langle v_5,v_7,v_8,v_9\rangle\subset \C^5\right\},
$$
while the flag $\CF^W$ is defined by
$$
\CF^W=\left\{\langle v_{12}\rangle\subset \langle v_{11},v_{12}\rangle\subset  \langle v_{10},v_{11}, v_{12}\rangle\subset  \langle v_8,v_{10},v_{11},v_{12}\rangle\subset \C^5\right\}.
$$
It is easy to see that $\CF^a$ is transversal to $\CF^W$ precisely when $V$ is in the open subset $U_a$. The basis $u_5,u_7,u_8,u_9,u_{10}$ can be obtained by intersecting $\CF^a$ and $\CF^W$:
$$
u_5=v_5, u_7\in \langle v_5,v_7\rangle\cap \langle v_8,v_{10},v_{11},v_{12}\rangle, u_8\in \langle v_5,v_7, v_8\rangle\cap \langle v_{10},v_{11}, v_{12}\rangle, 
$$
$$ 
u_9\in \langle v_5,v_7,v_8,v_9\rangle\cap \langle v_{11},v_{12}\rangle, u_{10}=v_{12}.
$$ 
Note that by Lemma \ref{lem: 1d intersection new} we have $v_7\in \langle v_5,v_7\rangle\cap \langle v_8,v_{10},v_{11},v_{12}\rangle$, 
so the vector $u_7$ is proportional to $v_7,$   and the coefficient is given by \eqref{eq: step 1 other triang}:
\begin{equation}
\label{eq: u7}
u_7=\frac{\Delta_{5,8,10,11,12}}{\Delta_{5,7,10,11,12}}v_7=\frac{1}{\Delta_{5,7,10,11,12}}v_7\quad \mathrm{since}\ \Delta_{I_{\mu}}=1.
\end{equation}

To find $u_8$ and $u_9$, we can expand $v_9$ and $v_{10}$ in the basis $(v_5,v_7,v_8,v_{11},v_{12})$ using Cramer's Rule:
\begin{equation}
\label{eq: v9}v_9=\frac{\Delta_{9,7,8,11,12}}{\Delta_{5,7,8,11,12}}v_5+\frac{\Delta_{5,9,8,11,12}}{\Delta_{5,7,8,11,12}}v_7+\frac{\Delta_{5,7,9,11,12}}{\Delta_{5,7,8,11,12}}v_8+\frac{\Delta_{5,7,8,9,12}}{\Delta_{5,7,8,11,12}}v_{11}+\frac{\Delta_{5,7,8,11,9}}{\Delta_{5,7,8,11,12}}v_{12},
\end{equation}
\begin{equation}
\label{eq: v10}
v_{10}=\frac{\Delta_{10,7,8,11,12}}{\Delta_{5,7,8,11,12}}v_5+\frac{\Delta_{5,10,8,11,12}}{\Delta_{5,7,8,11,12}}v_7+\frac{\Delta_{5,7,10,11,12}}{\Delta_{5,7,8,11,12}}v_8+\frac{\Delta_{5,7,8,10,12}}{\Delta_{5,7,8,11,12}}v_{11}+\frac{\Delta_{5,7,8,11,10}}{\Delta_{5,7,8,11,12}}v_{12},
\end{equation}

Given that $u_8\in\langle v_5,v_7,v_8\rangle\cap\langle v_{10},v_{11},v_{12}\rangle$ and $u_9\in\langle v_5,v_7,v_8,v_9\rangle\cap\langle v_{11},v_{12}\rangle$, we rewrite \eqref{eq: v9} and \eqref{eq: v10} as

\begin{equation}
\label{eq: pre u8}
    v_{10}-\frac{\Delta_{5{,}7{,}8{,}10{,}12}}{\Delta_{5{,}7{,}8{,}11{,}12}}v_{11}-\frac{\Delta_{5{,}7{,}8{,}11{,}10}}{\Delta_{5{,}7{,}8{,}11{,}12}}v_{12}=\frac{\Delta_{10{,}7{,}8{,}11{,}12}}{\Delta_{5{,}7{,}8{,}11{,}12}}v_5+
\frac{\Delta_{5{,}10{,}8{,}11{,}12}}{\Delta_{5{,}7{,}8{,}11{,}12}}v_7+\frac{\Delta_{5{,}7{,}10{,}11{,}12}}{\Delta_{5{,}7{,}8{,}11{,}12}}v_8,
\end{equation}

\begin{equation}
\label{eq: pre u9}
    v_9-\frac{\Delta_{9,7,8,11,12}}{\Delta_{5,7,8,11,12}}v_5-
\frac{\Delta_{5,9,8,11,12}}{\Delta_{5,7,8,11,12}}v_7-\frac{\Delta_{5,7,9,11,12}}{\Delta_{5,7,8,11,12}}v_8=\frac{\Delta_{5,7,8,9,12}}{\Delta_{5,7,8,11,12}}v_{11}+\frac{\Delta_{5,7,8,11,9}}{\Delta_{5,7,8,11,12}}v_{12}.
\end{equation}

We find that \eqref{eq: pre u8} defines $u_8$ under the normalization of \eqref{eq: step 1 triang}, whereas \eqref{eq: pre u9} requires rescaling by $\minor_{5,7,8,11,12}/\minor_{5,7,8,10,12}$ to obtain the following:
\begin{equation}
\label{eq: u8}
u_8=v_{10}-\frac{\Delta_{5{,}7{,}8{,}10{,}12}}{\Delta_{5{,}7{,}8{,}11{,}12}}v_{11}-\frac{\Delta_{5{,}7{,}8{,}11{,}10}}{\Delta_{5{,}7{,}8{,}11{,}12}}v_{12}=\frac{\Delta_{10{,}7{,}8{,}11{,}12}}{\Delta_{5{,}7{,}8{,}11{,}12}}v_5+
\frac{\Delta_{5{,}10{,}8{,}11{,}12}}{\Delta_{5{,}7{,}8{,}11{,}12}}v_7+\frac{\Delta_{5{,}7{,}10{,}11{,}12}}{\Delta_{5{,}7{,}8{,}11{,}12}}v_8,
\end{equation}
\begin{equation}
\label{eq: u9}
u_9=v_{11}+\frac{\Delta_{5,7,8,11,9}}{\Delta_{5,7,8,9,12}}v_{12}=\frac{\Delta_{5,7,8,11,12}}{\Delta_{5,7,8,9,12}}v_9-\frac{\Delta_{9,7,8,11,12}}{\Delta_{5,7,8,9,12}}v_5-
\frac{\Delta_{5,9,8,11,12}}{\Delta_{5,7,8,9,12}}v_7-\frac{\Delta_{5,7,9,11,12}}{\Delta_{5,7,8,9,12}}v_8.
\end{equation}

After the cut, we get the left matrix $V^L=(w_1,\ldots,w_7)$ where
$$
w_1=v_5, w_2=v_7, w_3=v_8, w_4=v_9, w_5=v_{10}, w_6=v_{11}, w_7=v_{12},
$$
 and the right matrix $V^R=(u_1,\ldots,u_{10})$ where $u_7,u_8,u_9$ are given by \eqref{eq: u7},\eqref{eq: u8}, \eqref{eq: u9} and 
$$
 u_i=v_i \quad  \text{for $1\le i\le 6$} \qquad \text{and}\quad u_{10}=v_{12}. 
$$
We can fill in the diagrams $\lambda^{6,L}/\mu^{6,L}$ and $\lambda^{6,R}/\mu^{6,R}$ with $I'_L(a',i)$ and $I'_R(a',i)$ respectively: 

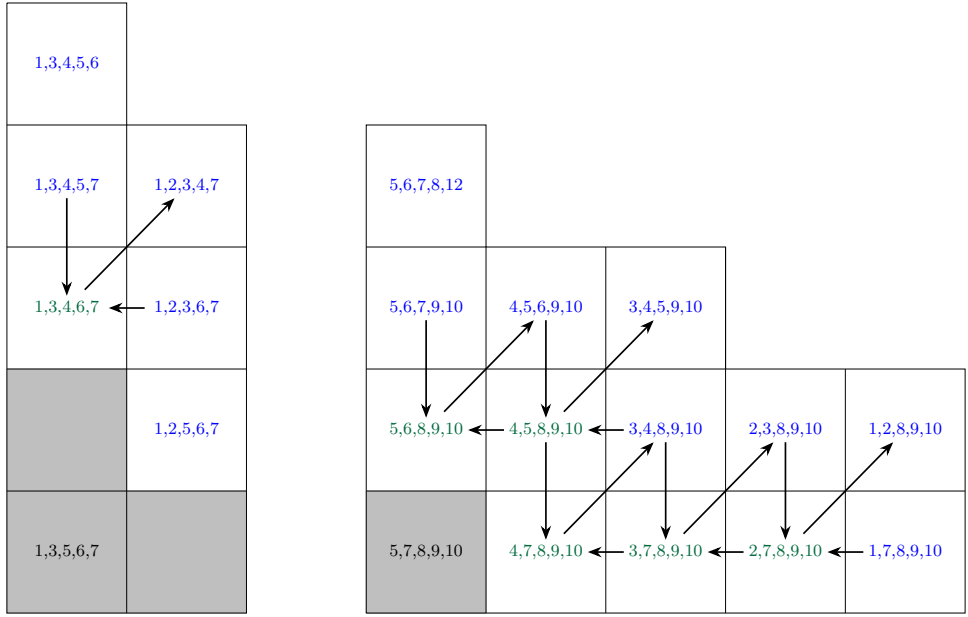
\begin{center}
\resizebox{5.1in}{3.2in}{
\begin{tikzpicture} 
\filldraw[color=lightgray, scale=2] (0,-3)--(0,-5)--(2,-5)--(2,-4)--(1,-4)--(1,-3)--(0,-3);
\draw[scale=2] (0,0)--(0,-5)--(2,-5)--(2,-1)--(1,-1)--(1,0)--(0,0);
\draw[scale=2] (0,-4)--(2,-4);
\draw[scale=2] (0,-3)--(2,-3);
\draw[scale=2] (0,-2)--(2,-2);
\draw[scale=2] (0,-1)--(1,-1);
\draw[scale=2] (1,-5)--(1,-1);

\filldraw[color=lightgray,scale=2] (3,-4)--(3,-5)--(4,-5)--(4,-4)--(3,-4);
\draw[scale=2] (3,-5)--(3,-1)--(4,-1)--(4,-2)--(6,-2)--(6,-3)--(8,-3)--(8,-5)--(3,-5);
\draw[scale=2] (3,-4)--(8,-4);
\draw[scale=2] (3,-3)--(6,-3);
\draw[scale=2] (3,-2)--(4,-2);
\draw[scale=2] (4,-5)--(4,-2);
\draw[scale=2] (5,-5)--(5,-2);
\draw[scale=2] (6,-5)--(6,-2);
\draw[scale=2] (7,-5)--(7,-3);
\draw (15,-9) node {\textcolor{blue}{\scriptsize $1{,}7{,}8{,}9{,}10$}};
\draw (13,-9) node {\textcolor{cadmiumgreen}{\scriptsize $2{,}7{,}8{,}9{,}10$}};
\draw (11,-9) node {\textcolor{cadmiumgreen}{\scriptsize $3{,}7{,}8{,}9{,}10$}};
\draw (9,-9) node {\textcolor{cadmiumgreen}{\scriptsize $4{,}7{,}8{,}9{,}10$}};

\draw (15,-7) node {\textcolor{blue}{\scriptsize $1{,}2{,}8{,}9{,}10$}};
\draw (13,-7) node {\textcolor{blue}{\scriptsize $2{,}3{,}8{,}9{,}10$}};
\draw (11,-7) node {\textcolor{blue}{\scriptsize $3{,}4{,}8{,}9{,}10$}};
\draw (9,-7) node {\textcolor{cadmiumgreen}{\scriptsize $4{,}5{,}8{,}9{,}10$}};
\draw (7,-7) node {\textcolor{cadmiumgreen}{\scriptsize $5{,}6{,}8{,}9{,}10$}};

\draw (11,-5) node {\textcolor{blue}{\scriptsize $3{,}4{,}5{,}9{,}10$}};
\draw (9,-5) node {\textcolor{blue}{\scriptsize $4{,}5{,}6{,}9{,}10$}};
\draw (7,-5) node {\textcolor{blue}{\scriptsize $5{,}6{,}7{,}9{,}10$}};
\draw (7,-3) node {\textcolor{blue}{\scriptsize $5{,}6{,}7{,}8{,}12$}};

\draw (7,-9) node {\scriptsize $5{,}7{,}8{,}9{,}10$};

\draw[thick, -Stealth] (14.3,-9) to (13.7,-9);
\draw[thick, -Stealth] (12.3,-9) to (11.7,-9);
\draw[thick, -Stealth] (10.3,-9) to (9.7,-9);

\draw[thick, -Stealth] (10.3,-7) to (9.7,-7);
\draw[thick, -Stealth] (8.3,-7) to (7.7,-7);

\draw[thick, -Stealth] (13,-7.2) to (13,-8.8);
\draw[thick, -Stealth] (11,-7.2) to (11,-8.8);
\draw[thick, -Stealth] (9,-7.2) to (9,-8.8);

\draw[thick, -Stealth] (9,-5.2) to (9,-6.8);
\draw[thick, -Stealth] (7,-5.2) to (7,-6.8);

\draw[thick, -Stealth] (13.3,-8.7) to (14.8,-7.2);
\draw[thick, -Stealth] (11.3,-8.7) to (12.8,-7.2);
\draw[thick, -Stealth] (9.3,-8.7) to (10.8,-7.2);

\draw[thick, -Stealth] (9.3,-6.7) to (10.8,-5.2);
\draw[thick, -Stealth] (7.3,-6.7) to (8.8,-5.2);
 

\draw (3,-7) node {\textcolor{blue}{\scriptsize $1{,}2{,}5{,}6{,}7$}};
\draw (3,-5) node {\textcolor{blue}{\scriptsize $1{,}2{,}3{,}6{,}7$}};
\draw (1,-5) node {\textcolor{cadmiumgreen}{\scriptsize $1{,}3{,}4{,}6{,}7$}};
\draw (3,-3) node {\textcolor{blue}{\scriptsize $1{,}2{,}3{,}4{,}7$}};
\draw (1,-3) node {\textcolor{blue}{\scriptsize $1{,}3{,}4{,}5{,}7$}};
\draw (1,-1) node {\textcolor{blue}{\scriptsize $1{,}3{,}4{,}5{,}6$}};
\draw (1,-9) node{\scriptsize $1{,}3{,}5{,}6{,}7$};

\draw[thick, -Stealth] (2.3,-5) to (1.7,-5);
\draw[thick, -Stealth] (1,-3.2) to (1,-4.8);
\draw[thick, -Stealth] (1.3,-4.7) to (2.8,-3.2);
\end{tikzpicture}
}
\end{center}

The reader is invited to check that the corresponding minors are rescaled as in Lemma \ref{lem: minors triang}.


\subsection{Proof of Theorem \ref{thm: main}}

In this section we prove Theorem \ref{thm: main} which we repeat for the reader's convenience.
\begin{theorem}
    Let $n > 0$ and $0 < k \leq n$. Let $\mu \subseteq \lambda$ be partitions fitting inside a $k \times (n-k)$-rectangle. Choose $1 \leq a \leq n-k$, and let $U_a$ be the principal open set in $S_{\skewschubert}^\circ$ defined by the non-vanishing of the cluster variables in the $a$-th column of $\skewschubert$  (from the right). Then, we have an isomorphism
\[
\Phi_a: U_a  \xrightarrow{\sim} S_{\skewleft}^\circ \times S_{\skewright}^\circ. 
\]
Furthermore, $\Phi_a$ is a quasi-cluster equivalence.
\end{theorem}

\begin{proof}
We use the commutative diagram \eqref{eq: splicing diagram} to summarize the results of this section. By \eqref{eq: Phi flags}, the configurations of flags $\Omega^L$ and $\Omega^R$ yield well-defined points in the respective braid varieties $X(\longest\beta^L)$ and $X(\longest\beta^R)$. We can formally restate the constructions in Sections \ref{sec: splicing left} and \ref{sec: splicing right} by writing 
$$
V^L=\Xi(\Omega^L),\ V^R=\Xi(\Omega^R)
$$
where $\Xi=\widetilde{\Omega}^{-1}$ is defined in the proof of Theorem \ref{thm: from skew schubert to braid variety}. In particular, the proof of Theorem \ref{thm: from skew schubert to braid variety} implies that 
$$
V^L\in S_{\skewleft}^\circ,\ V^R\in S_{\skewright}^\circ
$$
and $\Omega(V^L)=\Omega^L,\Omega(V^R)=\Omega^R$. Therefore $\Phi_a$ is well defined. Note that applying row operations on $V$ corresponds to simultaneous row operations on $V^L$ and $V^R$.

Next, we prove that $\Phi_a$ is an isomorphism. The matrices $V^L$ and $V^R$ are defined up to row operations on both (independent from each other), but we can fix a unique representative for $V^L$ by requiring that $v_{b_1},\ldots,v_{b_k}$ is a standard basis. This determines the (framed) flag $\CF^a$, and we can apply unique row operations on $V^R$ which match both framed flags $\CF^W$ and $\CF^a$. Now all vectors $v_i$ are contained in the union of columns of $V^L$ and $V^R$.

Lemma \ref{lem: minors triang} implies that all cluster variables in $Q^R$ are multiplied by monomials in frozens, and 
by Lemmas \ref{lem: exchange ratios not bdry} and \ref{lem: exchange ratios bdry} the exchange ratios are preserved. So $\Phi_a$ is a quasi-cluster isomorphism.
\end{proof}

\subsection{Splicing on $n$-strands} As mentioned in Section \ref{sec: n strands}, the skew shaped positroid $S^{\circ}_{\skewschubert}$ is isomorphic to the braid variety $X\left(\lift{w_{\skewschubert}}\lift{w_0^{(n)}}\right)$, where $w_{\skewschubert} \in S_n$, cf. \eqref{eq:decomposition-with-w0}.
Selecting $1 \leq a \leq n-k$, we have
\[
w_{\lambda^{a, L} / \mu^{a, L}} \in S_{n-a}, \qquad  w_{\lambda^{a, R} / \mu^{a, R}} \in S_{a + k}.
\]

We choose an embedding $S_{a+k} \hookrightarrow S_n$ as the subgroup of permutations that leave fixed the elements $a+k+1, \dots, n$, and we embed $S_{n-a} \hookrightarrow S_{n}$ as the subgroup of elements that leave fixed the elements $1, \dots, a$. Upon these embeddings, we have a length-additive decomposition
\begin{equation}\label{eq:red-dec-w-skewschubert}
w_{\skewschubert} = w_{\lambda^{a, L}/\mu^{a, L}}w_{\lambda^{a, R}/\mu^{a, R}},
\end{equation}
see Figure \ref{fig:red-dec-w-skewschubert}. By results of \cite{forthcoming}, the decomposition \eqref{eq:red-dec-w-skewschubert} implies that we have an open embedding
\[
X\left(\lift{w_{\lambda^{a, L}/\mu^{a, L}}}\lift{w_0^{(n)}}\right) \times X\left(\lift{w_{\lambda^{a, R}/\mu^{a, R}}}\lift{w_0^{(n)}}\right) \hookrightarrow X\left(\lift{w_{\skewschubert}}\lift{w_0^{(n)}}\right)
\]
that induces a quasi-cluster equivalence between the left-hand side an an open subset of the right-hand side given by the freezing of some cluster variables. We do not know whether, upon the corresponding identifications, this coincides with the embedding $S^{\circ}_{\lambda^{a,L}/\mu^{a,L}} \times S^{\circ}_{\lambda^{a, R}/\mu^{a, R}} \hookrightarrow S^{\circ}_{\skewschubert}$.
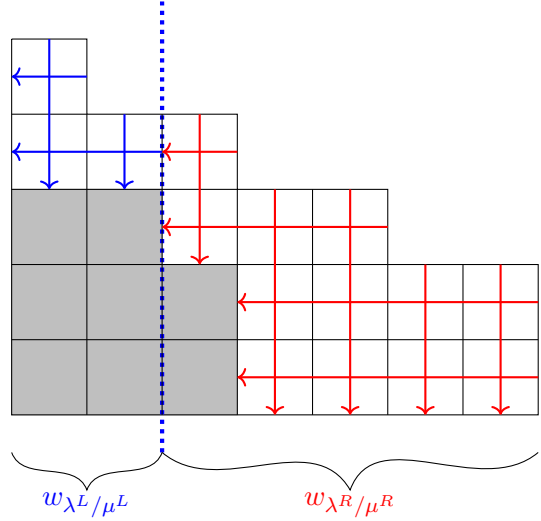
\begin{figure}
    \centering
    \begin{tikzpicture}
\filldraw[color=lightgray] (0,-2)--(0,-5)--(3,-5)--(3,-3)--(2,-3)--(2,-2)--(0,-2);
\draw (0,0)--(0,-5)--(7,-5)--(7,-3)--(5,-3)--(5,-2)--(3,-2)--(3,-1)--(1,-1)--(1,0)--(0,0);
\draw [dotted, ultra thick, color=blue] (2,-5.5)--(2,0.5);
\draw (0,-4)--(7,-4);
\draw (0,-3)--(5,-3);
\draw (0,-2)--(3,-2);
\draw (0,-1)--(1,-1);
\draw (1,-5)--(1,-1);
\draw (2,-5)--(2,-1);
\draw (3,-5)--(3,-2);
\draw (4,-5)--(4,-2);
\draw (5,-5)--(5,-3);
\draw (6,-5)--(6,-3);
\draw[color=red, thick, ->](7, -3.5) -- (3, -3.5);
\draw[color=red, thick, ->](7, -4.5) -- (3, -4.5);
\draw[color=red, thick, ->] (5, -2.5) -- (2, -2.5);
\draw[color=red, thick, ->] (3, -1.5) -- (2, -1.5);
\draw[color=red, thick, ->] (6.5, -3) -- (6.5, -5);
\draw[color=red, thick, ->] (5.5, -3) -- (5.5, -5);
\draw[color=red, thick, ->] (4.5, -2) -- (4.5, -5);
\draw[color=red, thick, ->] (3.5, -2) -- (3.5, -5);
\draw[color=red, thick, ->] (2.5, -1) -- (2.5, -3);

\draw[color=blue, thick, ->] (1.5, -1) -- (1.5, -2);
\draw[color=blue, thick, ->] (0.5, 0) -- (0.5, -2);
\draw[color=blue, thick, ->] (2, -1.5) -- (0, -1.5);
\draw[color=blue, thick, ->] (1, -0.5) -- (0, -0.5);

\draw (0, -5.5) to[out=290, in=110] (1, -6) to [out=70, in=250] (2, -5.5);
\node at (1, -6.2) {\color{blue}$w_{\lambda^{L}/\mu^{L}}$};

\draw (2, -5.5) to[out=320,in=120] (4.5, -6) to[out=70, in=220] (7, -5.5);
\node at (4.5, -6.2) {\color{red}$w_{\lambda^{R}/\mu^{R}}$};
\end{tikzpicture}
    \caption{Length-additive decomposition $w_{\skewschubert} = w_{\lambda^{a, L}/\mu^{a, L}}w_{\lambda^{a, R}/\mu^{a, R}}$.}
    \label{fig:red-dec-w-skewschubert}
\end{figure}

\bibliographystyle{plain}
\bibliography{skew.bibliography.bib}
\end{document}